\documentclass[10pt]{amsart}
\usepackage[a4paper, total={6in,8in}]{geometry}
\usepackage[utf8]{inputenc}
\usepackage[T1]{fontenc}
\usepackage[english]{babel}
\usepackage{amsmath}
\usepackage{amsfonts}
\usepackage{amssymb}
\usepackage{amsthm}
\usepackage{bbm}
\usepackage{csquotes}
\usepackage{lipsum}
\usepackage{esint}
\usepackage{mathtools}
\usepackage{tikz}
\usepackage{enumitem}
\usetikzlibrary{arrows}

\newcommand{\abar}{\overline{a}}
\newcommand{\bbar}{\overline{b}}

\newcommand{\Abar}{\overline{A}}
\newcommand{\Bbar}{\overline{B}}

\newcommand{\pdv}[2]{\frac{\partial #1}{\partial #2}}
\newcommand{\mpdv}[3]{\frac{\partial^2 #1}{\partial #2 \partial #3}}
\newcommand{\intxv}{\int_{xv}}
\newcommand{\inttxv}{\int_{txv}}
\newcommand{\intxvv}{\int_{xvv_*}}
\newcommand{\weakstarto}{\overset{\ast}{\rightharpoonup}}
\newcommand{\weakto}{\rightharpoonup}
\newcommand{\Rv}{\mathbb{R}^N_v}
\newcommand{\Rx}{\mathbb{R}^N_x}
\newcommand{\Rxv}{\mathbb{R}^{2N}_{x,v}}
\newcommand{\Rxvv}{\mathbb{R}^{3N}_{xvv_*}}

\newcommand{\weightxv}{\langle x, v \rangle}

\newcommand{\vertiii}[1]{{\left\vert\kern-0.25ex\left\vert\kern-0.25ex\left\vert #1
		\right\vert\kern-0.25ex\right\vert\kern-0.25ex\right\vert}}

\makeatletter
\newcommand{\proofpart}[2]{%
	\vspace{8pt}
	\par
	\addvspace{\medskipamount}%
	\noindent\emph{Step #1: #2}\par\nobreak
	\addvspace{\smallskipamount}%
	\@afterheading
}
\makeatother

\DeclareMathOperator{\supp}{supp}
\DeclareMathOperator{\Div}{div}

\newtheorem{definition}{Definition}
\newtheorem{lemma}{Lemma}
\newtheorem{theorem}{Theorem}
\newtheorem{proposition}{Proposition}
\newtheorem{remark}{Remark}

\author{Paulo Sampaio}
\title{Global solutions to the Landau-Fermi-Dirac equation}
\date{}
\address{Centre de Mathématiques Laurent Schwartz (CMLS), CNRS, École polytechnique, Institut Polytechnique de Paris, Palaiseau, France}
\email{paulo-elpidio.alves-sampaio@polytechnique.edu}

\begin{document}
\begin{abstract}
In this work, we prove the existence of global weak solutions to the non-homogeneous Landau-Fermi-Dirac (LFD) equation in the case of a Coulomb potential, as well as modified weak solutions in the case of very soft potentials.
A compactness result for Landau-like hypoelliptic equations is established by combining the partial ellipticity of the collision kernel with velocity averaging results for transport equations.
Then, by constructing global solutions to a regularized approximate equation, we use this compactness to pass to the limit.

A major obstacle in this process is the Fermi-Dirac saturation effect, which eliminates velocity diffusion and prevents global strong compactness.
For general initial data and very soft potentials, we obtain a weak solution to a modified LFD equation containing an indicator function that excludes the fully saturated regions.
In the specific case of the Coulomb potential, we show that strictly non-saturated initial data never completely saturates, allowing us to maintain strong compactness and obtain a standard global weak solution.
\end{abstract}

\maketitle
\tableofcontents

\section{Introduction}
\label{sec:introduction}

Kinetic theory aims to describe the dynamics of a gas through a particle number density function $f = f(t,x,v)$.
That is, at each time $t$, $f(t,x,v)dxdv$ represents the number of particles within an infinitesimal spatial volume $dx$ centered at $x$, with velocities in an infinitesimal range $dv$ centered at $v$.

The evolution of $f$ is traditionally described by the Boltzmann equation.
However, hot gases of charged particles, such as plasmas, present a regime where long-range Coulomb interactions dominate, rendering this model inadequate.
To address this, Lev Landau \cite{landau-translated-1936} introduced in 1936 a specialized kinetic equation, now known as the Landau equation.

Despite its success, Landau's original formulation remains a purely classical model governed by Maxwell-Boltzmann statistics.
However, in many situations of interest, the particles are fermions, which obey Fermi-Dirac statistics.
For these systems, quantum effects like the Pauli exclusion principle become significant and must be incorporated, leading to the Landau-Fermi-Dirac (LFD) equation.

In this paper, we study the Cauchy problem for this equation.
Specifically, given initial data $f_0$, we seek a function $f$ satisfying
\begin{equation}
	\label{eq:lfd}
	\begin{cases}
	\displaystyle
	\pdv{f}{t} + v \cdot \nabla_x f
	= \Div_v
	\left[
		\int a(v - v_*) \Big(
					f_* (1-\delta f_*) \nabla_v f -
					f (1 - \delta f) \nabla_{v_*}f_*
					\Big) dv_*
	\right],\\

	f|_{t=0} = f_0.
	\end{cases}
\end{equation}
Throughout, we adopt the standard shorthand notation $f = f(t,x,v)$ and $f_* = f(t,x,v_*)$.
The collision kernel $a = (a_{ij})_{ij}$ is given by
\begin{equation}
	\label{collision-kernel-formula}
	a (z) = \Gamma(|z|) \left( I - \frac{z \otimes z}{|z|^2}\right),
\end{equation}
where the cross section $\Gamma$ depends on the type of interaction between the particles.
The parameter $\delta > 0$ is proportional to the cube of Planck's constant $\hbar$, and thus quantifies the quantum effects within the system.
Accordingly, we recover the classical Landau equation in the limit $\delta \to 0$.

If the particle interaction arises from a power law potential in $N$ dimensions, the cross section is given by
\begin{equation}
	\label{property:cross-section}
	\Gamma(|z|) = |z|^{\gamma+2}, \quad\text{with } -N \leq \gamma \leq 1.
\end{equation}
In the kinetic theory literature, we find that these interactions are classified according to the range of $\gamma$.
The case $0 < \gamma \leq 1$ corresponds to hard potentials, while $\gamma = 0$ defines Maxwellian molecules.
For negative values, the potential is classified as moderately soft when $-2 \leq \gamma < 0$, and very soft when $-N < \gamma < -2$.
Finally, the most singular cross section dealt with in the literature is when $\gamma = -N$, which corresponds to a Coulomb potential.

Just as in the classical Landau equation, the structure of the collision kernel $a$ allows us to formally deduce the conservation laws
\[
	\frac{d}{dt} \iint f(t) \varphi \, dxdv = 0
\]
for $\varphi = 1, v_i, |v|^2$ or $|x - vt|^2$.
These choices correspond physically to the conservation of mass, linear momentum, kinetic energy and moment of inertia, respectively.

Furthermore, introducing the quantum entropy,
\[
	S(t) = \frac{1}{\delta}
		\iint
		\delta f \log (\delta f) + (1-\delta f) \log (1-\delta f)
		\, dxdv,
\]
it follows that, at least formally, this quantity should be non-increasing in time.

Since these quantities are either conserved or dissipated, they remain uniformly bounded for all $t > 0$, provided they are finite at time $t = 0$.
Therefore, it is natural to expect that any physically meaningful solution to the Cauchy problem will satisfy the uniform \emph{a priori} bound
\begin{equation}
	\label{natural-bounds}
	\sup_{t > 0} \iint f (1 + |x-vt|^2 + |v|^2 + \log f) \, dxdv < +\infty.
\end{equation}

For the Landau equation with highly singular kernels, most notably in the Coulomb case ($\gamma = -N$), the bounds implied by \eqref{natural-bounds} are insufficient to define a weak solution, even in the sense of distributions.
To overcome this problem, Lions \cite{lions_1994} introduced the notion of a \emph{renormalized solution}, in the same spirit as his earlier work with DiPerna for the Boltzmann equation \cite{diperna-lions-1989}.

To establish the existence of such solutions, a standard strategy for nonlinear PDEs is to derive a \emph{sequential stability} result.
This entails proving that any sequence of solutions to this equation is compact in some suitable topological space.
By constructing smooth solutions to a well-chosen approximation scheme, such a compactness result allows us to pass to the limit and recover a solution to the original, unregularized problem.

However, a significant difficulty arises for the Landau equation.
While a sequential stability result for its solutions exists \cite{lions_1994}, the nonlinear structure of the renormalized formulation complicates the passage to the limit.
As demonstrated by Villani \cite{villani_1996}, this limiting process gives rise to a defect measure, necessitating an even weaker notion of solution.

Fermions, on the other hand, obey the \emph{Pauli Exclusion Principle}, which states that no two particles can occupy the same quantum state simultaneously.
Mathematically, this constraint imposes a saturation limit on the distribution function, expressed as
\begin{equation}
	\label{pauli-bound}
	0 \leq f \leq \frac{1}{\delta}.
\end{equation}
This additional $L^\infty$ estimate for the solutions provides sufficient control to ensure that the collision integral of LFD is well-defined even in the presence of the Coulomb singularity.
Therefore, for the Landau-Fermi-Dirac equation, this saturating effect suggests that the existence of solutions can be established in the typical weak setting, avoiding both the nonlinear complications of a renormalized formulation and the appearance of defect measures.

The mathematical study of the Cauchy problem for quantum fluids dates back to the work of Dolbeault \cite{dolbeault_1994}, who established existence and uniqueness of solutions to the Boltzmann-Fermi-Dirac (BFD) equation with cutoff satisfying the conservation laws and decay of entropy.
Subsequently, Alexandre \cite{alexandre_2000} extended these results to the non-cutoff BFD equation.

For the Landau-Fermi-Dirac equation, the first significant results were obtained by Bagland \cite{bagland_2004}, who proved the existence of solutions for hard potentials in the spatially homogeneous setting.
More recently, multiple results on the Cauchy problem for homogeneous LFD have been proven: Alonso, Bagland, Desvillettes and Lods \cite{alonso_2022} established existence for soft potentials, while Golding, Gualdani and Zamponi \cite{golding_2022} have shown existence and smoothness for a Coulomb potential.

The non-homogeneous case of the LFD equation, however, remains poorly understood and, to the author’s knowledge, there are no results on the Cauchy problem in this case.

In this paper, we establish the existence of global weak solutions to the non-homogeneous LFD equation.
We are primarily interested in the regime where the particles interact through a Coulomb potential, corresponding to $\gamma = -N$.
This regime serves as the primary focus for Theorem \ref{existence-lfd}, which yields a standard global existence result for non-saturated initial data.
However, our framework also accommodates general initial data and extends to the range of very soft potentials, at the cost of obtaining a modified version of the LFD equation in Theorem \ref{existence-lfd-indicator}.

\section{Main results}
\label{sec:main-results}

Throughout the paper, unless otherwise specified, we will use the repeated index summation convention.
We will also denote the phase space by $\Rxv = \mathbb{R}^N_x \times \mathbb{R}^N_v$.
Within this section, the notation $\iint \, dxdv$ denotes integration over $\Rxv$.

We first observe that the parameter $\delta$ can be set to $1$ without loss of generality.
Indeed, relabeling $\delta f$ as $f$ and $a_{ij}$ as $a_{ij}/\delta$, the equation takes the form
\begin{equation}
	\label{eq:lfd-first}
	\pdv{f}{t} + v \cdot \nabla_x f = \pdv{}{v_i} \left(\abar_{ij} \pdv{f}{v_j} - \bbar_i f(1-f) \right),
\end{equation}
where
$ \abar_{ij} = a_{ij} \ast_v (f(1-f)) $ and $\bbar_i = \partial_{v_j} a_{ij} \ast_v f$.
Therefore, we shall assume $\delta = 1$ throughout the remainder of the paper to simplify the notation.

In this work, we focus exclusively on the \emph{very soft and Coulomb potentials}.
Specifically, we consider the cross section $\Gamma(|z|)$ as defined in \eqref{property:cross-section}, where the exponent satisfies
\[
	-N \leq \gamma < -2.
\]

Let us now state explicitly what we mean by a solution of the equations we study in this paper.
\begin{definition}
	\label{def-weak-solution}
	Assume the coefficients $\Abar_{ij}, \Bbar_i$ belong to $L^1_{loc}((0,\infty) \times \Rxv)$ and assume that for every $i = 1, 2, \dots, N$ we have $\sum_j \partial_{v_j} \Abar_{ij} \in L^1_{loc}((0,\infty) \times \Rxv)$.

	A function $f = f(t,x,v)$ in $L^\infty_{loc}((0,\infty) \times \Rxv) \cap C([0,\infty);\mathcal{D}'(\Rxv))$ is called a (global) weak solution of the equation
	\[
		\pdv{f}{t} + v \cdot \nabla_x f = \pdv{}{v_i} \left( \Abar_{ij} \pdv{f}{v_j} - \Bbar_i f(1-f) \right)
	\]
	with initial data $f_0$ if, for every $t > 0$ and every test function $\varphi \in \mathcal{D}([0,\infty) \times \mathbb{R}^N_x \times \mathbb{R}^N_v)$, we have
	\begin{equation*}
		\begin{split}
			&\iint f(t) \varphi(t) \, dxdv
			-\iint f_0 \varphi(0) \, dxdv
			-\int_0^t \iint f \pdv{\varphi}{\tau} \, dxdvd\tau
			-\int_0^t \iint f v_i \pdv{\varphi}{x_i} \, dxdvd\tau \\
			&\quad = \int_0^t \iint \Abar_{ij} f \mpdv{\varphi}{v_i}{v_j} \, dxdvd\tau
			+ \int_0^t \iint \left(
					\pdv{\Abar_{ij}}{v_j} f
					+ \Bbar_i f (1-f)
				\right) \pdv{\varphi}{v_i} \, dxdvd\tau
		\end{split}
	\end{equation*}
\end{definition}

To show the existence of weak solutions to the LFD equation, we proceed by a compactness method.
First, we construct solutions to a well-behaved approximate equation.
After establishing the compactness of these solutions in an appropriate topological space, we pass to the limit in the approximate equation.

One major obstacle for the LFD equation, which is not present in its classical Landau counterpart, is the \emph{saturation} effect.
When $f \to 1$, the coefficient $\abar$ vanishes and the LFD equation loses its diffusion in the $v$ variable.
As a consequence, we cannot expect to establish strong compactness in the saturated case.

This effect was previously observed by Lu \cite{lu-2006} in the context of the Boltzmann-Fermi-Dirac equation.
Nevertheless, he succeeded in constructing solutions to the BFD equation using only weak compactness \cite{lu-2008}, a result made possible because the nonlinear products within the BFD collision operator are non-local in the velocity variable.

For the LFD equation, however, the nonlinear terms $f(1-f)$ are all evaluated at the same velocity point.
Weak convergence $f^n \weakto f$ implies $f^n (1-f^n) \weakto f(1-f)$ if and only if the sequence converges strongly.
Therefore, a priori, we cannot pass the whole equation to the limit without strong convergence.

This loss of compactness can be taken to its extreme by considering fully saturated solutions.
For any $E \subset \Rxv$ bounded and measurable, observe that the function
\[
	f(t,x,v)
	= \mathbbm{1}_E (x - tv, v)
\]
annihilates the right-hand side of \eqref{eq:lfd-first}.
Because it is also a solution to the free transport equation $\partial_t f + v \cdot \nabla_x f = 0$, this provides an example of a weak solution of LFD that has constant quantum entropy and satisfies all conservation laws, but lacks diffusion everywhere.

Since arbitrary weak-* limits of such indicator functions do not generally solve the LFD equation, a universal framework is out of reach, indicating that strong compactness may only be reachable by restricting the scope of the problem.

One possible alternative is to show that if the initial data is strictly non-saturated ($f_0 \leq 1 - \delta$), then the solutions stay non-saturated, thus preserving strong compactness.
Specifically, the sequence of approximate solutions converges strongly in $L^1_{loc}((0,\infty); L^1(\Rxv))$, allowing us to pass to the limit everywhere.

To propagate this non-saturation, we apply a maximum principle to an approximate equation.
Interestingly, this approach relies on the boundedness of certain terms of the equation, which we are only able to guarantee in the Coulomb case ($\gamma = -N$).
It remains unclear at present whether the inability to propagate the non-saturation to the class of very soft potentials is a limitation of the techniques used here or a fundamental physical feature of the equation.

Our first result then establishes the existence of global weak solutions strictly for the Coulomb potential, provided the initial data is uniformly non-saturated.
\begin{theorem}
	\label{existence-lfd}
	Let $N \geq 3$ and $f_0 \in L^1(\Rxv; (1+|x|^2+|v|^2) dxdv)$ be such that $0 \leq f_0 \leq 1 - \delta$, for some $\delta > 0$.
	There exists a weak solution $f \in C((0,\infty);\mathcal{D}'(\Rxv)) \cap L^\infty((0,\infty); L^1(\Rxv))$ to the LFD equation in the Coulomb case ($\gamma = -N$) with initial data $f_0$, in the sense of Definition \ref{def-weak-solution}, which satisfies, for a.e. $t \geq 0$,

	1) Pauli exclusion principle:
	\[
		0 \leq f(t) \leq 1.
	\]

	2) Conservation of mass and linear momentum:
	\[
		\iint f(t,x,v) \, dxdv
		= \iint f_0(x,v) \,dxdv,
	\]\[
		\iint f(t,x,v) v_i \, dxdv
		= \iint f_0(x,v) v_i \, dxdv
		\hspace{10pt} \forall i \in \{ 1, \cdots, N \}.
	\]

	3) Decay of kinetic energy and moment of inertia:
	\begin{align*}
		\iint f(t,x,v) |v|^2 \,dxdv
		&\leq \iint f_0(x,v) |v|^2 \, dxdv, \\[5pt]
		\intertext{\centering and\vspace{5pt}}
		\iint f(t,x,v) |x-tv|^2 \,dxdv
		&\leq \iint f_0(x,v) |x-tv|^2 \, dxdv.
	\end{align*}

	4) Entropy inequality:
	\[
		\iint s(t,x,v)\,dxdv + \int_{0}^{t}\iint d(\tau,x,v)\,dxdvd\tau \leq \iint s(0,x,v)\,dxdv,
	\]
	where
	\[
		s(t,x,v) = f\log f + (1-f)\log(1-f)
	\]
	is the quantum entropy	and
	\begin{equation}
		\label{dissipation-def}
		\begin{split}
			d(t,x,v) =
			\int &\bigg[
				a_{ij}(v-v_*) f (1- f) f_* (1 -  f_*)
				\left(
					\frac{\partial_{v_i} f}{f(1- f)} - \frac{\partial_{v_{*,i}}f_*}{f_*(1- f_*)}
				\right) \\
			&\quad \times \left(
				\frac{\partial_{v_j} f}{f(1- f)} - \frac{\partial_{v_{*,j}}f_*}{f_*(1- f_*)}
			\right)
			\bigg] \,dv_*
		\end{split}
	\end{equation}
	is the quantum entropy dissipation.
\end{theorem}

\begin{remark}
	\label{rem:entopy-dissipation-definition}
	Given that weak solutions may have no regularity in $v$, one might wonder what the expression in \eqref{dissipation-def} means.
	In the proof of the entropy inequality, we will see that the integral $\int_{0}^{t}\iint d(\tau,x,v)\,dxdvd\tau$ represents the $L^2((0,t) \times \mathbb{R}^N_x \times \mathbb{R}^N_v \times \mathbb{R}^N_{v_*})$ norm of
 	\begin{multline*}
			2 \Div_v \left[ \sqrt{a(v-v_*)} \sqrt{f_* (1-f_*)} \arcsin \sqrt{f} \right]
			- 2\Div_{v_*} \left[ \sqrt{a(v-v_*)} \sqrt{f (1-f)} \arcsin \sqrt{f_*} \right]\\
			- 2\Div_v \left( \sqrt{a(v-v_*)} \right) \sqrt{f_* (1-f_*)} \arcsin \sqrt{f}
			+ 2\Div_{v_*} \left(\sqrt{a(v-v_*)}\right) \sqrt{f (1-f)} \arcsin \sqrt{f_*},
	\end{multline*}
	where the divergence of a matrix $A$ is defined as $\Div_v A = \sum_{j} \pdv{A_{ij}}{v_j}$.
	If $f$ possesses sufficient regularity in $v$, these two expressions coincide, which justifies the formal notation in \eqref{dissipation-def}.
\end{remark}

If we include the range of very soft potentials or if we remove the strict non-saturation assumption on the initial data, strong compactness can no longer be guaranteed everywhere.
Because of this, we rely on a weaker argument where the strong convergence in $L^1_{loc}$ is only guaranteed on the restricted domain $\mathcal{S} \times \mathbb{R}^N_v$, where $\mathcal{S}$ is a specific subset of the time-space coordinates.
Outside of this set, we only have weak convergence, and an indicator function then appears in the limit equation.

\begin{theorem}
	\label{existence-lfd-indicator}
	Let $N \geq 3$ and $f_0 \in L^1(\Rxv; (1+|x|^2+|v|^2) dxdv)$ be such that $0 \leq f_0 \leq 1$.
	There exists a measurable set $\mathcal{S} \subset \{(t,x) \in (0,\infty) \times \mathbb{R}^N_x : \int_{\mathbb{R}^N_v} f(1-f) \, dv > 0 \; \text{or } \int_{\mathbb{R}^N_v} f \, dv = 0\}$ and a weak solution $f \in C((0,\infty);\mathcal{D}'(\Rxv)) \cap L^\infty((0,\infty); L^1(\Rxv))$ to the equation
	\[
		\pdv{f}{t} + v \cdot \nabla_x f
		= \mathbbm{1}_{\mathcal{S}}(t,x)
		\pdv{}{v_i} \left(\abar_{ij} \pdv{f}{v_j} - \bbar_i f(1-f) \right)
	\]
	with initial data $f_0$, in the sense of Definition \ref{def-weak-solution}.
	This solution satisfies, for a.e. $t \geq 0$, the properties 1) through 4) of Theorem \ref{existence-lfd}, with the modified quantum entropy dissipation
	\begin{equation*}
		\begin{split}
			d(t,x,v) =
			\mathbbm{1}_{\mathcal{S}}(t,x)
			\int &\bigg[
				a_{ij}(v-v_*) f (1- f) f_* (1 -  f_*)
				\left(
					\frac{\partial_{v_i} f}{f(1- f)} - \frac{\partial_{v_{*,i}}f_*}{f_*(1- f_*)}
				\right) \\
			&\quad \times \left(
				\frac{\partial_{v_j} f}{f(1- f)} - \frac{\partial_{v_{*,j}}f_*}{f_*(1- f_*)}
			\right)
			\bigg] \,dv_*.
		\end{split}
	\end{equation*}
\end{theorem}

The remainder of the paper is organized as follows.
In Section \ref{sec:compactness-lfd}, we prove a compactness result for the solutions of LFD-like equations, reminiscent of the works of Lions \cite{lions_1994} and Villani \cite{villani_1996} on renormalized solutions of the Landau equation.
This involves studying the parabolic and transport parts of this equation separately in order to achieve compactness for the sequence of solutions.

In Section \ref{sec:cauchy-approximated-equation}, we study the Cauchy problem for an approximate equation.
We construct regular solutions to this equation and prove that they satisfy the conservation laws (mass, linear momentum, kinetic energy and moment of inertia) in an approximate sense, as well as an entropy inequality.

We conclude by proving the main theorems in reverse order.
In Section \ref{sec:existence-theorem-indicator}, we prove Theorem \ref{existence-lfd-indicator}.
By taking a sequence of approximate solutions constructed in Section \ref{sec:cauchy-approximated-equation} with progressively weaker regularization, we apply the compactness result of Section \ref{sec:compactness-lfd} to pass to the limit.
Finally, in Section \ref{sec:existence-theorem-coulomb} we prove Theorem \ref{existence-lfd} by propagating the non-saturation of the initial data for any finite time, demonstrating that the set $\mathcal{S}$ from Theorem \ref{existence-lfd-indicator} is, in fact, the entire space $(0,\infty) \times \mathbb{R}^N_x$.

\section{A compactness result for the LFD equation}
\label{sec:compactness-lfd}

We begin by establishing several notational conventions that will be used throughout the remainder of the paper.
First, the integral over the full phase space $\Rxv$ will be denoted simply by $\intxv$.

Next, for any matrix-valued function $A = (A_{ij}(v))_{ij}$, we define its divergence $\Div_v A$ as the vector whose $i$-th component is given by the row-wise divergence $(\Div_v A)_i = \sum_{j=1}^N \partial_{v_j} A_{ij}$, for every $i = 1, 2, \dots, N$.
We denote by $\weightxv = (1 + |x|^2 + |v|^2)^{1/2}$ the Japanese bracket in two variables $x$ and $v$.
Finally, we denote by $B_R$ the closed ball in $\mathbb{R}^N$ of radius $R$ centered at the origin.

In this section, our main goal is to prove the following compactness result for LFD-like equations.
\begin{proposition}
	\label{prop-compactness}
	Let $(g^n)_n$ be a sequence of weak solutions in $L^2_{loc}((0,\infty) \times \mathbb{R}^N_x; H^1_{loc}(\mathbb{R}^N_v))$ to the equation
	\begin{equation}
		\label{eq-landau-like}
		\begin{cases}
		\displaystyle
		\pdv{g^n}{t} + v \cdot \nabla_x g^n
		= \Div_v \left[ \Abar^n \nabla_v g^n - \Bbar^n g^n (1-g^n) \right]\\

		g^n|_{t=0} = g^n_0,
		\end{cases}
	\end{equation}
	where the coefficients $\Abar^n$, $\Div_v \Abar^n$, $\Bbar^n$, $\Div_v \Bbar^n$ belong to $L^\infty_{loc}((0,\infty) \times \Rxv)$.

	Assume that the following hypotheses hold uniformly in $n$:
	\begin{enumerate}[label=\alph*., ref=\alph*, font=\normalfont]
		\item For every $T > 0$, there exists a constant $C > 0$ such that
		\begin{equation}
		\label{eq:uniform-a-priori}
			\sup_{t \in [0,T]} \intxv g^n (1 + |x - vt|^2 + |v|^2) \leq C,
			\quad \text{and} \quad
			0 \leq g^n \leq 1.
		\end{equation}
		\item For every $T, R, R' > 0$ and $\varepsilon>0$, there exists a $\nu = \nu(R, R', \varepsilon) > 0$, a measurable set $E \subset [0,T] \times B_R$ with $|E| < \varepsilon$ and an integer $n_0$ such that for all $(t,x,v) \in E^c \times B_R$, $\eta \in \mathbb{R}^N$ and $n \geq n_0$,
		\begin{equation}
			\label{ineq-uniform-ellip}
			\Abar^n_{ij} \eta_i \eta_j \geq \nu \int_{|v_*| \leq R'}
			\left[ 1 - \biggl( \frac{v-v_*}{|v-v_*|} \cdot \frac{\eta}{|\eta|} \biggr)^2 \right] |\eta|^2 g^n_* (1-g^n_*) \, dv_*.
		\end{equation}
		\item The sequences $( \Abar^n)_n$, $( \Div_v \Abar^n)_n$, $(\Bbar^n)_n$, and $(\Div_v \Bbar^n)_n$ are uniformly bounded in the space $L^1_{loc}((0,\infty) \times \Rxv)$.
	\end{enumerate}

	Suppose $g^n(1 - g^n) \weakstarto G$ weakly-* in $L^\infty((0,\infty) \times \Rxv)$.
	Then the sequence $(g^n(1 - g^n))_n$ is strongly compact in $L^1_{loc}((0,\infty); L^1(\Rxv))$.
	For every $T > 0$, the sequence $(g^n)_n$ is weakly compact in $L^1_{loc}((0,\infty); L^1(\Rxv))$ and strongly compact in $L^1(\mathcal{R}_T \times \mathbb{R}^N_v)$, where $\mathcal{R}_T = \{(t,x) \in (0,T) \times \mathbb{R}^N_x : \int G \, dv > 0 \text{ or } \int g \, dv = 0 \}$.
\end{proposition}

Our approach to proving this proposition will be to achieve compactness step by step.
First, we use the transport structure of the equation to achieve compactness in space-time of averages in velocity, by means of the now famous \emph{velocity averaging} results.
In a second step, we study the parabolic part of the equation to deduce a certain compactness in the velocity variable.
Finally, we unify these two compactness properties with a proposition, leading to the ``full'' compactness, in all three variables.

\subsection{Velocity averages: compactness in time and position}

Velocity averaging is a standard technique in the study of transport equations, and in particular kinetic equations, which allows us to establish compactness for the solutions once integrated over the velocity space.

The main idea is that if $f = f(t,x,v)$ is a solution to the transport equation $\partial_t f + v \cdot \nabla_x f = g$, the velocity average $\int f \varphi(v) \, dv$, for $\varphi \in C^\infty_c(\mathbb{R}^N_v)$, exhibits higher regularity than $f$ and $g$ themselves.
For a sequence of solutions $f^n$, this translates into compactness for the averaged sequence $\int f^n \varphi(v) \; dv$.

The first results in this direction were established independently in the mid-1980s by Agoshkov \cite{agoshkov_1984} and Golse, Perthame and Sentis in \cite{golse_1988} in an $L^2$ setting.
We will use a generalized version taken from \cite{bouchut-2000}, which is reproduced here for the reader's convenience.

\begin{theorem}[Theorem 1.1.8 from \cite{bouchut-2000}]
	\label{bouchut-vel-avg}
	Let $\Omega$ be an open set of $\mathbb{R}_t \times \mathbb{R}^N_x$, $a: \mathbb{R}^M \to \mathbb{R}^N$ verify $\partial^\alpha a \in L^\infty_{loc}$ for $|\alpha| \leq m$, $m\in\mathbb{N}$, and
	\[
		\forall \sigma \in S^{N-1} \;
		\forall u \in \mathbb{R} \;
		|{v ; a(v) \cdot \sigma = u}| = 0.
	\]

	Let $f_n$ be a bounded sequence in $L^p_{loc}(\Omega \times \mathbb{R}^M)$ with $p>1$ verifying
	\[
		\partial_t f_n + \Div_x a(v)f_n =
		\sum_{|\alpha| \leq m} \partial^\alpha_v g_n^{(\alpha)}
		\text{ in }
		\Omega \times \mathbb{R}^M
	\]
	where the $g_n^{(\alpha)}$ are locally bounded in the space of measures $\mathcal{M}_{loc}(\Omega \times \mathbb{R}^M)$.

	Then, for any function $\psi \in C^\infty_c(\mathbb{R}^M)$,
	\[
		\rho^n_\psi (t,x) = \int f_n(t,x,v) \psi(v) \, dv
	\]
	is compact in $L^q_{loc}(\Omega)$ for any $q < p$.
\end{theorem}

The first step in the proof of Proposition \ref{prop-compactness} is to apply Theorem \ref{bouchut-vel-avg} to establish the compactness of the velocity averages for both $g^n$ and $g^n (1-g^n)$.
While the result for $g^n$ is a straightforward consequence of equation \eqref{eq-landau-like}, for the nonlinear quantity $g^n(1-g^n)$, we need to find its equation in the sense of distributions, which we deduce by a standard mollification argument.
As a natural byproduct, the equation for $g^n(1-g^n)$ also yields the necessary gradient bound \eqref{ineq-aij-deriv}.

\begin{lemma}
	\label{lemma-vel-avg}
	Let $(g^n)_n$ be a sequence of solutions to \eqref{eq-landau-like} as in Proposition \ref{prop-compactness}.
	For every $\varphi \in C^\infty_c(\Rv)$, the sequences of velocity averages
	\[
		\int g^n \varphi \, dv
		\quad\text{and}\quad
		\int g^n (1-g^n) \varphi\, dv
	\]
	are compact in $L^1_{loc}((0,\infty) \times \mathbb{R}^N_x)$. Moreover, for every $R, T > 0$, we have
	\begin{equation}
		\label{ineq-aij-deriv}
		\int_0^T \iint_{B_R \times B_R} \Abar^n_{ij} \pdv{g^n}{v_i} \pdv{g^n}{v_j} \; dxdvdt \leq C_{T,R},
	\end{equation}
	for some constant $C_{T,R} > 0$, independent of $n$.
\end{lemma}
\begin{proof}
	The compactness of the first sequence of velocity averages can be directly obtained by rewriting equation \eqref{eq-landau-like} in the sense of distributions.
	We have that
	\[
		\pdv{g^n}{t} + v \cdot \nabla_x g^n =
		\mpdv{(\Abar^n_{ij} g^n)}{v_i}{v_j}
		- \pdv{}{v_i} \left(
			\pdv{\Abar^n_{ij}}{v_j} g^n
			+ \Bbar^n_i g^n (1-g^n)
		\right).
	\]
	Since $0 \leq g^n \leq 1$, we have from the uniform boundedness of the sequences $\Abar^n$, $\Div_v \Abar^n$ and $\Bbar^n$ in $L^1_{loc}((0,\infty) \times \Rxv)$ that the terms inside the derivatives on the right-hand side are uniformly bounded in this space.
	By Theorem \ref{bouchut-vel-avg}, this implies the compactness of the velocity averages $\int g^n \varphi \, dv$ in $L^1_{loc}$.

	To apply Theorem \ref{bouchut-vel-avg}, we will need an equation where the transport acts on $G^n \coloneq g^n (1-g^n)$.
	However, we cannot use \eqref{eq-landau-like} directly.
	The solutions lack sufficient regularity in $t$ and $x$, meaning we cannot justify a product rule to compute the derivatives of $G^n$.

	To overcome this, we consider smooth approximations of the $g^n$ functions.
	Let $\chi = \chi(t,x)$ be a positive $C^\infty_c((0,\infty) \times \mathbb{R}^N_x)$ function such that $\int \chi \; dtdx = 1$.
	We introduce the mollifying sequence $\chi_m(t,x) = m^{N+1} \chi(mt,mx)$ and define $g^n_m = \chi_m * g^n$.
	Note that to compute this convolution we extend the functions $g^n$ by zero for $t \leq 0$.

	We thus have $g^n_m \to g^n$ in $L^2_{loc}((0,\infty) \times \Rxv)$.
	Taking equation \eqref{eq-landau-like} in $\mathcal{D}'((0,\infty) \times \Rxv)$ and applying the approximations we obtain, commuting convolutions with derivatives
	\[
		\pdv{g^n_m}{t} + v \cdot \nabla_x g^n_m =
		\pdv{}{v_i} \left[ \chi_m * \left(\Abar^n_{ij} \pdv{g^n}{v_j}\right) \right]
		- \pdv{}{v_i} \left[ \chi_m * \left(\Bbar^n_i g^n (1-g^n)\right) \right].
	\]

	Multiplying both sides of the equation by $(1-2g^n_m)$ and denoting $G^n_m \equiv g^n_m (1-g^n_m)$, we have
	\begin{equation*}
	\begin{split}
		\pdv{G^n_m}{t} + v \cdot \nabla_x G^n_m
		&= \pdv{}{v_i} \left[
				\chi_m * \left(	\Abar^n_{ij} \pdv{g^n}{v_j} \right) (1-2g^n_m)
		\right]
		+ 2 \chi_m * \left(\Abar^n_{ij} \pdv{g^n}{v_j}\right) \pdv{g^n_m}{v_i}\\
		&\quad- \pdv{}{v_i} \left[
				\chi_m * \left(	\Bbar^n_i g^n(1-g^n) \right) (1-2g^n_m)
		\right]
		- 2 \chi_m * \left( \Bbar^n_i g^n(1-g^n) \right) \pdv{g^n_m}{v_i}.
	\end{split}
	\end{equation*}
	We multiply this equation on both sides by $\varphi \in \mathcal{D}((0,\infty) \times \Rxv)$ and integrate by parts to find a formulation in the sense of distributions.
	Denoting by $\inttxv$ the integral $\int_{0}^{\infty} \iint_{\Rxv} dxdvdt$, we have
	\begin{equation*}
	\begin{split}
		- \inttxv G^n_m \pdv{\varphi}{t} - \inttxv G^n_m v_i \pdv{\varphi}{x_i}
		&= - \inttxv \chi_m * \left(	\Abar^n_{ij} \pdv{g^n}{v_j} \right) (1-2g^n_m) \pdv{\varphi}{v_i} \\
		&\quad + 2 \inttxv \chi_m * \left(\Abar^n_{ij} \pdv{g^n}{v_j}\right) \pdv{g^n_m}{v_i} \varphi \\
		&\quad + \inttxv \chi_m * \left( \Bbar^n_i g^n(1-g^n) \right) (1-2g^n_m) \pdv{\varphi}{v_i} \\
		&\quad - 2 \inttxv \chi_m * \left( \Bbar^n_i g^n(1-g^n) \right) \pdv{g^n_m}{v_i} \varphi
	\end{split}
	\end{equation*}
	Applying standard properties of approximation by convolution, we pass to the limit $m \to \infty$ to find that
	\begin{equation}
	\label{eq:Gn-distributions-proto-integral}
	\begin{split}
		- \inttxv G^n \pdv{\varphi}{t} - \inttxv G^n v_i \pdv{\varphi}{x_i}
		&= - \inttxv \Abar^n_{ij} \pdv{g^n}{v_j} (1-2g^n) \pdv{\varphi}{v_i}
		+ 2 \inttxv \Abar^n_{ij} \pdv{g^n}{v_j} \pdv{g^n}{v_i} \varphi \\
		&\quad + \inttxv \Bbar^n_i g^n(1-g^n) (1-2g^n) \pdv{\varphi}{v_i}
		- 2 \inttxv \Bbar^n_i g^n(1-g^n) \pdv{g^n}{v_i} \varphi.
	\end{split}
	\end{equation}

	Since $g^n, 1-g^n \in L^2_{loc}((0,\infty) \times \Rx; H^1_{loc}(\Rv))$, the product rule is valid with respect to the velocity variable.
	Therefore, $\partial_{v_k} (g^n(1-g^n)) \in L^1_{loc}((0,\infty) \times \Rxv)$ for every $k = 1, 2, \dots, N$ and
	\[
		\pdv{(g^n(1-g^n))}{v_k} = (1-2g^n) \pdv{g^n}{v_k}.
	\]

	Let $\psi_m = \psi_m(t,x,v)$ be a mollifying sequence and consider the regularized diffusion coefficient $(\Abar^n_{ij})_m \equiv \psi_m * \Abar^n_{ij}$.
	By standard properties of convolutions, we have $(\Abar^n_{ij})_m \to \Abar^n_{ij}$ and $\partial_{v_k} (\Abar^n_{ij})_m \to \partial_{v_k} \Abar^n_{ij}$ in $L^p_{loc}((0,\infty) \times \Rxv)$ for any $p<\infty$.
	Since the product $(\Abar^n_{ij})_m \partial_{v_i} \varphi$ belongs to $C^\infty_c((0,\infty) \times \Rxv)$, we then have, by the definition of weak derivative,
	\begin{align*}
		\inttxv (\Abar^n_{ij})_m \pdv{g^n}{v_j} (1-2g^n) \pdv{\varphi}{v_i}
		&= \inttxv \pdv{(g^n(1-g^n))}{v_j} (\Abar^n_{ij})_m \pdv{\varphi}{v_i}\\
		&= - \inttxv g^n(1-g^n) \pdv{}{v_j} \left( (\Abar^n_{ij})_m \pdv{\varphi}{v_i} \right)\\
		&= - \inttxv g^n(1-g^n) \pdv{(\Abar^n_{ij})_m}{v_j} \pdv{\varphi}{v_i}
			- \inttxv g^n(1-g^n) (\Abar^n_{ij})_m \mpdv{\varphi}{v_i}{v_j}.
	\end{align*}
	Passing to the limit $m \to \infty$, we obtain
	\[
		\inttxv \Abar^n_{ij} \pdv{g^n}{v_j} (1-2g^n) \pdv{\varphi}{v_i}
		= - \inttxv g^n(1-g^n) \pdv{\Abar^n_{ij}}{v_j} \pdv{\varphi}{v_i}
			- \inttxv g^n(1-g^n) \Abar^n_{ij} \mpdv{\varphi}{v_i}{v_j}.
	\]
	Similarly, we can show that for the $\Bbar$ term,
	\[
		\inttxv \Bbar^n_i g^n(1-g^n) \pdv{g^n}{v_i} \varphi
		= - \inttxv \pdv{\Bbar^n_i}{v_i} \left( \frac{(g^n)^2}{2} - \frac{(g^n)^3}{3} \right) \varphi
			- \inttxv \Bbar^n_i \left( \frac{(g^n)^2}{2} - \frac{(g^n)^3}{3} \right) \pdv{\varphi}{v_i}.
	\]
	Substituting these identities back into \eqref{eq:Gn-distributions-proto-integral} yields
	\begin{equation}
	\begin{split}
		\label{eq-Gn-distributions-integral}
		- \inttxv G^n \pdv{\varphi}{t}
		- \inttxv G^n v_i \pdv{\varphi}{x_i}
		&= \inttxv \Abar^n_{ij} g^n(1-g^n) \mpdv{\varphi}{v_i}{v_j}
		+ \inttxv \big(\Div \Abar^n \big)_i \, g^n(1-g^n) \pdv{\varphi}{v_i} \\
		&\quad + 2 \inttxv \Abar^n_{ij} \pdv{g^n}{v_i} \pdv{g^n}{v_j} \varphi
		+ \inttxv \Bbar^n_i g^n(1-g^n)(1-2g^n) \pdv{\varphi}{v_i} \\
		&\quad + 2 \inttxv
			\bigg[
				\Div \Bbar^n \varphi
				+ \Bbar_i^n \pdv{\varphi}{v_i}
			\bigg]
			\left(
				\frac{(g^n)^2}{2}
				- \frac{(g^n)^3}{3}
			\right).
	\end{split}
	\end{equation}

	Before applying the velocity averaging lemma, we pause to note that this formulation conveniently yields the bound \eqref{ineq-aij-deriv}.
	While this bound could be deduced from classical energy estimates on the original equation, we can rearrange \eqref{eq-Gn-distributions-integral} and use that $0 \leq g^n \leq 1$ to obtain
	\begin{align*}
		\inttxv \Abar^n_{ij} \pdv{g^n}{v_i} \pdv{g^n}{v_j} \varphi
		\leq &\left\| \pdv{\varphi}{t} \right\|_{L^1}
			+ \| v \cdot \nabla_x \varphi \|_{L^1}
			+ \left\| \big(\Div \Abar^n \big)_i \right\|_{L^1(\supp \varphi)} \left\| \pdv{\varphi}{v_i} \right\|_{L^\infty}\\
			&+ \left\| \Abar^n_{ij} \right\|_{L^1(\supp \varphi)} \left\| \mpdv{\varphi}{v_i}{v_j} \right\|_{L^\infty}
			+ 3 \left\| \Bbar^n_i \right\|_{L^1(\supp \varphi)} \left\| \pdv{\varphi}{v_i} \right\|_{L^\infty}\\
			&+ 2 \left\| \Div \Bbar \right\|_{L^1(\supp \varphi)} \| \varphi \|_{L^\infty}.
	\end{align*}
	Given the uniform boundedness of the coefficients, the inequality \eqref{ineq-aij-deriv} follows by taking a nonnegative test function $\varphi$ such that $\varphi \equiv 1$ in $[0,T] \times B_R \times B_R$ and using the nonnegativity of $\Abar^n_{ij} \partial_{v_i} g^n \partial_{v_j} g^n$.

	With this bound established, we return to proving the compactness of the velocity averages.
	Equation \eqref{eq-Gn-distributions-integral} implies that, in the sense of distributions $\mathcal{D}'((0,\infty) \times \Rxv)$, the transport equation for $G^n$ can be written as
	\begin{equation*}
		\pdv{G^n}{t} + v_i \pdv{G^n}{x_i}
		= \mpdv{H^n_{ij}}{v_i}{v_j}
		+ \pdv{H^n_{i}}{v_i}
		+ H^n
	\end{equation*}
	where
	\begin{align*}
		H^n_{ij}
			&= \Abar^n_{ij} g^n (1-g^n) \\
		H^n_{i}
			&= - \Bbar^n_i  g^n(1-g^n)(1-2g^n)
			- \big(\Div \Abar^n \big)_i \, g^n(1-g^n)
			- 2 \Bbar^n_i \left(\frac{(g^n)^2}{2} - \frac{(g^n)^3}{3}\right) \\
		H^n
			& = 2 \Abar^n_{ij} \pdv{g^n}{v_i} \pdv{g^n}{v_j}
			+ 2 \big(\Div \Bbar^n \big) \left(\frac{(g^n)^2}{2} - \frac{(g^n)^3}{3}\right)
	\end{align*}

	Since $0 \leq g^n \leq 1$ for all $n$, it follows that
	\begin{align*}
		| H^n_{ij} |
			&\leq |\Abar^n_{ij}|\\
		|  H^n_{i} |
			&\leq 3 |\Bbar_i^n| + \left| (\Div \Abar^n)_i \right|\\
		| H^n |
			&\leq 2 \Abar^n_{ij} \pdv{g^n}{v_i} \pdv{g^n}{v_j} + 2 \left| \Div \Bbar^n \right|,
	\end{align*}
	implying each of these sequences of functions is uniformly bounded in $L^1_{loc}((0,\infty) \times \Rxv)$.
	Thus, Theorem \ref{bouchut-vel-avg} applies and we have that the velocity averages of $g^n (1-g^n)$ are compact.
\end{proof}

\subsection{Quasi-ellipticity of the diffusion matrix: compactness in velocity}
\label{subsec:quasi-ellipticity}

In this section, we explore the parabolic part of the equation in order to establish compactness for the velocity gradients $\nabla_v g^n$.

Our approach is largely inspired by classical parabolic theory.
If the equations \eqref{eq-landau-like} were uniformly parabolic, the matrices $\Abar^n$ would be uniformly elliptic, meaning there would exist a constant $\theta > 0$ such that $\Abar^n_{ij} \eta_i \eta_j \geq \theta |\eta|^2$ for all $\eta \in \mathbb{R}^N$.
This property, together with the inequality \eqref{ineq-aij-deriv}, would directly yield a bound on the velocity derivatives of $g^n$, which would imply compactness in the $v$ variable.
Naturally, we seek to show an analogous ellipticity estimate uniform in $n$ for the sequence $\abar^n = (\abar^n_{ij})_{ij} = (a_{ij} *_v f(1-f))_{ij}$ of the LFD equation.

For the homogeneous case, Bagland \cite{bagland_2004} demonstrated that the a priori bounds on entropy, mass and kinetic energy control the size of $g(1-g)$ from below, allowing us to prove an ellipticity estimate using the techniques from Desvillettes and Villani \cite{desvillettes-villani_2000}.
However, this approach does not readily translate to the non-homogeneous case, where a priori estimates alone are insufficient to guarantee such a bound.
Following Lions \cite{lions_1994}, we instead seek a weaker, more specific notion of ellipticity that suffices to establish compactness.

To this end, consider the estimate \eqref{ineq-uniform-ellip}.
This inequality provides a lower bound for the quadratic form $\Abar^n_{ij} \eta_i \eta_j$.
Inspecting the structure of the integrand in \eqref{ineq-uniform-ellip}, we can identify essentially two obstacles preventing the right-hand side from being elliptic.
The first is an angular problem, occurring when $v-v_*$ and $\eta$ are nearly parallel.
This is a purely geometric constraint and can be addressed with a truncation argument.

The second is a statistical vacuum problem, occurring when the functions $G^n = g^n(1-g^n)$ become too small.
Unlike the angular factor, the behavior of $G^n$ depends entirely on the evolution of the solutions $g^n$.
Intuitively, we expect $\Abar_{ij}^n$ to be elliptic in the regions where $G^n$ possesses sufficient mass.
Asymptotically, one then expects the uniform ellipticity of the sequence $\Abar_{ij}^n$ to be described by the behavior of the limit $G = \lim G^n$.

The following result formalizes this intuition by establishing that, in the regions of space-time where $G$ is not a vacuum (i.e., the density $\rho = \int G \, dv$ is strictly greater than some $\alpha > 0$), we have asymptotically uniform ellipticity outside of a set of arbitrarily small measure.

\begin{lemma}
	\label{lemma-ellipticity}
	Fix $T, R > 0$.
	Let $(g^n)_n$ be a sequence of functions such that $0 \leq g^n \leq 1$ for every $n \in \mathbb{N}$.
	Consider $G^n \equiv g^n(1-g^n)$ and let $G$ be a function of $L^\infty((0,T) \times \Rxv)$ such that
	\[
		\int G^n \varphi\, dv \to \int G \varphi\, dv
		\quad\text{a.e. in }
		[0,T] \times B_R
	\]
	for every $\varphi \in L^1(\Rv)$.
	Moreover, let
	\[
		K_{\alpha} \equiv \{ (t,x) \in [0,T] \times B_{R} : \rho(t,x) > \alpha \},
	\]
	where $\rho(t,x) = \int G \,dv$ and let $\Abar^n$ be a sequence satisfying \eqref{ineq-uniform-ellip}.
	Then, for every $\alpha, \varepsilon > 0$ there exists a measurable set $E$ with $|E| < \varepsilon$ and an integer $n_0$ such that,
	\[
		\Abar^n_{ij}(t,x,v) \eta_i \eta_j \geq C(\alpha, \varepsilon) |\eta|^2,
	\]
	for $(t,x,v,\eta) \in (K_{\alpha} \cap E^c) \times B_R \times \mathbb{R}^N$ and $n \geq n_0$.
\end{lemma}
\begin{proof}
	By homogeneity, we may assume without loss of generality that $\eta \in \mathbb{S}^{N-1}$.
	To address the angular singularity impeding uniform ellipticity, we introduce the parameter $\mu \in [0,1)$.
	We define the following set, which enforces the $R'$ constraint from \eqref{ineq-uniform-ellip} and truncates the problematic angles,
	\[
		V(v,\mu,\eta) = \left\{
			v_* \in \mathbb{R}^N
			\text{ s.t. }
			\left|\frac{v-v_*}{|v-v_*|} \cdot \eta\right| \leq \mu
			\text{ and }
			|v_*| \leq (1-\mu)^{-1}
			\right\}.
	\]
	Next, we define the corresponding truncation of the limit function $\rho$,
	\[
		\rho_{\mu}(t,x,v,\eta) = \int \mathbbm{1}_{V(v,\mu,\eta)}(v_*) G(t,x,v_*) \,dv_*.
	\]
	As $\mu \nearrow 1$, the sets expand as $V(v,\mu,\eta) \nearrow \mathbb{R}^N \setminus (v + \mathbb{R} \eta)$.
	Thus, by Beppo Levi's lemma we have that
	\[
		\rho_{\mu}(t,x,v,\eta) \to \rho(t,x)
		\text{ a.e. } (t,x) \in (0,T)\times B_R,\, \forall (v,\eta) \in B_R \times \mathbb{S}^{N-1}.
	\]

	For every fixed $(t,x)$, the dominated convergence theorem ensures that $\rho_\mu$ is continuous with respect to $(v,\eta)$.
	Since the limit $\rho$ is independent of these variables (and thus also continuous), we can apply Dini's theorem on the compact domain $B_R \times \mathbb{S}^{N-1}$.
	This guarantees that for almost every $(t,x)$, we have $\rho_{\mu}(t,x,v,\eta) \to \rho(t,x)$ uniformly in $(v,\eta)$.

	Applying Egorov's theorem we conclude that there exists a measurable set $E_1 \subset [0,T] \times B_R$, with $|E_1| < \varepsilon / 3$, such that $\rho_{\mu} \to \rho$ uniformly on $E_1^c \times B_R \times \mathbb{S}^{N-1}$.
	Therefore, there exists a $\mu_0 = \mu_0(\varepsilon, \alpha)$ such that
	\[
		\rho_{\mu_0} \geq \frac{\alpha}{2}
		\quad\text{ in }
		(K_\alpha \cap E_1^c) \times B_R \times \mathbb{S}^{N-1}.
	\]

	We define the approximated quantum density and its $V$ truncation by
	\[
		\rho^n(t,x)
		= \int G^n_* \,dv_*,
		\quad
		\rho^n_{\mu_0}(t,x,v,\eta)
		= \int \mathbbm{1}_{V(v,\mu_0,\eta)}(v_*) G^n_* \,dv_*.
	\]
	For any fixed $(t,x) \in [0,T] \times B_R$, the sequence of functions $(\rho^n_{\mu_0}(t,x,\cdot, \cdot))_n$ is uniformly equicontinuous.
	Indeed, observe that the mapping
	\begin{align*}
		B_R \times S^{N-1} &\to L^1(\Rv)\\
		(v,\eta) &\mapsto \mathbbm{1}_{V(v,\mu_0,\eta)}
	\end{align*}
	is continuous by the dominated convergence theorem.
	Since it has a compact domain, this mapping is uniformly continuous.
	Therefore, given any $\varepsilon' > 0$, there exists a $\delta > 0$ such that if $|(v_2, \eta_2) - (v_1, \eta_1)| \leq \delta$, we have $\| \mathbbm{1}_{V(v_2,\mu_0,\eta_2)} - \mathbbm{1}_{V(v_1,\mu_0,\eta_1)} \|_{L^1} \leq \varepsilon'$.
	Because $0 \leq G^n_* \leq 1$, it follows that
	\begin{equation*}
		|\rho^n_{\mu_0}(t,x,v_2,\eta_2) - \rho^n_{\mu_0}(t,x,v_1,\eta_1)|
		\leq \int \left| \mathbbm{1}_{V(v_2,\mu_0,\eta_2)} - \mathbbm{1}_{V(v_1,\mu_0,\eta_1)} \right| G^n_* \,dv_*
		\leq \varepsilon'.
	\end{equation*}

	We use this equicontinuity to improve the convergence on the variables $v$ and $\eta$.
	Let $D$ be a countable dense subset of $B_R \times \mathbb{S}^{N-1}$.
	By the hypothesis on $G^n$, for each fixed $(v, \eta) \in D$, we have $\rho^n_{\mu_0}(\cdot, \cdot, v, \eta) \to \rho_{\mu_0}(\cdot, \cdot, v, \eta)$ almost everywhere in $[0,T] \times B_R$.
	Because $D$ is countable, there exists a single set of full measure in $[0,T] \times B_R$ where this pointwise convergence holds for all $(v, \eta) \in D$.
	Since the sequence $\rho^n_{\mu_0}(t,x, \cdot, \cdot)$ is uniformly equicontinuous, the pointwise convergence on this dense set implies that $\rho^n_{\mu_0} \to \rho_{\mu_0}$ uniformly in the $(v, \eta)$ variables.

	By Egorov's theorem, there exists a measurable set $E_2 \subset [0,T] \times B_R$, with $|E_2| < \varepsilon / 3$, such that $\rho^n_{\mu_0} \to \rho_{\mu_0}$ uniformly in $E_2^c \times B_R \times \mathbb{S}^{N-1}$.
	Therefore, there exists an integer $N_1$ such that for every $n \geq N_1$,
	\[
		\rho^n_{\mu_0} \geq \frac{\alpha}{4} \text{ in } (K_\alpha \cap (E_1 \cup E_2)^c) \times B_R \times \mathbb{S}^{N-1}.
	\]

	The uniform quasi-ellipticity estimate \eqref{ineq-uniform-ellip} implies that for every $\varepsilon > 0$ there exists a $\nu(\mu_0, \varepsilon) > 0$ and a measurable set $E_3$ with $|E_3| < \varepsilon/3$ and an integer $N_2$ such that for every $(t,x) \in E_3^c$, $|v| \leq R$, $\eta \in \mathbb{S}^{N-1}$ and $n \geq N_2$, we have
	\begin{equation*}
		\Abar^n_{ij}(t,x,v) \eta_i \eta_j \geq \nu(\mu_0, \varepsilon) \rho^n_{\mu_0}(t,x,v,\eta).
	\end{equation*}

	Define $n_0 = \max(N_1, N_2)$.
	In light of the preceding estimates, it follows that for every $(t,x) \in E^c$, $(v, \eta) \in B_R \times \mathbb{S}^{N-1}$ and $n \geq n_0$, we have
	\begin{equation*}
		\Abar^n_{ij}(t,x,v) \eta_i \eta_j \geq \nu(\mu_0, \varepsilon) \frac{\alpha}{4},
	\end{equation*}
	where $E = E_1 \cup E_2 \cup E_3$ has measure $|E| < \varepsilon$.
\end{proof}

\subsection{Almost everywhere compactness}

Intuitively, the velocity averaging result of Lemma \ref{lemma-vel-avg} provides compactness in the variables $(t,x)$ through the regularizing effect of the transport operator.
It is therefore natural to expect that combining this with a compactness result in the velocity variable $v$ will yield compactness for the sequence of functions as a whole.

In this section, we make this intuition rigorous, showing that velocity averaging, combined with compactness in $v$, implies almost everywhere convergence for a subsequence.
Furthermore, we are able to relax the standard requirement for global uniform bounds on the velocity gradients.
Instead, we rely on a localized condition that, while technical, yields a more general result.

This extra generality, however, is necessary for our application since the diffusion matrix of the equation is not uniformly elliptic, as seen in Section \ref{subsec:quasi-ellipticity}.

\begin{proposition}
	\label{prop-ae-compactness}
	Let $(\Phi^n)_n$ be a sequence of nonnegative functions such that $\Phi^n \weakstarto \Phi$ in $L^\infty((0,\infty) \times \Rxv)$.
	Suppose that the velocity averages $(\int \Phi^n \varphi \,dv)_n$ are compact in $L^1_{loc}((0,\infty) \times \mathbb{R}^N_x)$.

	Let $\Psi \in L^1_{loc}((0,\infty) \times \Rxv)$ be a nonnegative function.
	Suppose that for every $T, R, \varepsilon, \alpha > 0$, there exists $C = C(\varepsilon, \alpha) > 0$ (which may also depend on $T$ and $R$), an integer $n_0 > 0$ and a measurable set $E \subset [0,T] \times B_R$ with $|E| \leq \varepsilon$ such that, for every $n \geq n_0$,
	\[
		\iiint_{(K_\alpha \cap E^c) \times B_R} |\nabla_v \Phi^n|^2 \,dtdxdv \leq C
	\]
	where
	$
		K_{\alpha} = \left\{ (t,x) \in [0,T] \times B_R : \int \Psi \, dv > \alpha \right\}
	$.
	Then, passing to a subsequence, $\Phi^n \to \Phi$ almost everywhere in $\{(t,x) \in (0,\infty) \times \mathbb{R}^N_x : \int \Psi \, dv > 0 \} \times \mathbb{R}^N_v$.
\end{proposition}
\begin{proof}

	Let us establish the convergence of the velocity averages.
	Since $C^\infty_c(\mathbb{R}^N_v)$ is dense in the separable space $L^1(\mathbb{R}^N_v)$, we can fix a countable dense subset $\{\varphi^k\}_k$ of $L^1(\mathbb{R}^N_v)$ such that $\varphi^k \in C^\infty_c(\mathbb{R}^N_v)$ for all $k$.
	Applying a standard Cantor diagonal extraction, we pass to a subsequence such that the velocity averages $\int \Phi^n \varphi^k \, dv$ converge in $L^1_{loc}((0,\infty) \times \mathbb{R}^N_x)$ for every $k$.
	Because the sequence $\Phi^n$ is uniformly bounded in $L^\infty((0,\infty) \times \Rxv)$, we can use a standard density argument to extend this convergence to all test functions $\varphi \in L^1(\mathbb{R}^N_v)$.

	Let $\eta \in C^\infty_c (\mathbb{R}^N_v)$ be a positive mollifier with $\int \eta \,dv$ = 1, supported in the unit ball, and define $\eta_{\delta}(v) = \frac{1}{\delta^N} \eta\left(\frac{v}{\delta}\right)$.
	We denote $\Phi^n_\delta := \Phi^n *_v \eta_{\delta}$ and $\Phi_\delta := \Phi *_v \eta_{\delta}$.
	Moreover, since $\Phi^n \weakstarto \Phi$ in $L^\infty$, the sequence is uniformly bounded by some $M>0$ almost everywhere.

	Fix $T, R_0 > 0$ and let $R_1 = R_0 + 1$.
	From the convergence of the velocity averages and the dominated convergence theorem, for every $\delta > 0$, $(\Phi^n_\delta)_n$ converges to $\Phi_\delta$ in $L^1((0,T) \times B_{R_1} \times B_{R_1})$.

	By the definition of $\Phi^n_\delta$ and the mean value theorem, we have
	\begin{equation*}
	\begin{split}
		\int_{B_{R_0}} |\Phi^n - \Phi^n_\delta| \, dv
		&\leq \int_{B_{R_0}} \int_{B_\delta} \frac{1}{\delta^N} \eta \left( \frac{v_*}{\delta} \right)
			|\Phi^n(v-v_*) -  \Phi^n(v)| \, dv_* dv \\
		&\leq \int_{B_{R_0}} \int_{B_\delta} \int_0^1
				\frac{1}{\delta^N} \eta \left( \frac{v_*}{\delta} \right)
				|\nabla_v \Phi^n(v-\tau v_*)| |v_*| \, d \tau dv_* dv.
	\end{split}
	\end{equation*}
	If $\delta$ is sufficiently small, then $v - \tau v_* \in B_{R_1}$, for every $v \in B_{R_0}$ and $\tau \in [0,1]$.
	The above integral can then be bounded as
	\begin{equation*}
	\begin{split}
		\int_{B_{R_1}} \int_{B_\delta} \int_0^1
			\frac{1}{\delta^N} \eta \left( \frac{v_*}{\delta} \right)
			|\nabla_v \Phi^n(v)| |v_*| \, d \tau dv_* dv
		&\leq \delta \int_{B_{R_1}} |\nabla_v \Phi^n(v)|
			\left( \int_{B_\delta} \frac{1}{\delta^N} \eta \left( \frac{v_*}{\delta} \right)\, dv_* \right) \, dv\\
		&= \delta \int_{B_{R_1}} |\nabla_v \Phi^n(v)| \, dv.
	\end{split}
	\end{equation*}
	Integrating over $(t,x) \in K_\alpha \cap E^c$ and applying the Cauchy-Schwarz inequality, we obtain
	\begin{equation}
	\label{eq:ae-convergence-delta-approx}
	\begin{split}
		\iiint_{(K_\alpha \cap E^c) \times B_{R_0}} |\Phi^n - \Phi^n_\delta| \, dtdxdv
		&\leq \delta \iiint_{(K_\alpha \cap E^c) \times B_{R_1}} |\nabla_v \Phi^n| \, dtdxdv\\
		&\leq \delta (C(\varepsilon, \alpha))^{1/2} (|K_\alpha \cap E^c| |B_{R_1}|)^{1/2}\\
		&\leq C_1(\varepsilon, \alpha) \delta,
	\end{split}
	\end{equation}
	for $\delta > 0$ sufficiently small and $n \geq n_0$.

	We show that $(\Phi^n)_n$ is Cauchy in $L^1(K_\alpha \times B_{R_0})$.
	Splitting the space-time integration domain into $K_\alpha \cap E$ and $K_\alpha \cap E^c$, we obtain that
	\begin{equation}
	\label{eq:ae-convergence-local-convergence-decomposition}
	\begin{split}
		\iiint_{K_\alpha \times B_{R_0}} |\Phi^n - \Phi^m| \, dtdxdv
		&= \iint_{K_\alpha \cap E} \int_{B_{R_0}} |\Phi^n - \Phi^m| \, dv \, dtdx \\
		&\quad + \iint_{K_\alpha \cap E^c} \int_{B_{R_0}} |\Phi^n - \Phi^n_\delta| + |\Phi^m - \Phi^m_\delta| \, dv \, dtdx\\
		&\quad + \iint_{K_\alpha \cap E^c} \int_{B_{R_0}} |\Phi^n_\delta - \Phi^m_\delta| \, dv \, dtdx\\
		&= I_1 + I_2 + I_3.
	\end{split}
	\end{equation}
	From the uniform bound on $\Phi^n$, we bound the first integral as $I_1 \leq 2 M |B_{R_0}| \varepsilon$.
	Meanwhile, the second integral can be estimated using \eqref{eq:ae-convergence-delta-approx}, yielding $I_2 \leq 2 C_1(\varepsilon, \alpha) \delta$.
	Finally, since for every fixed $\delta > 0$ the sequence $(\Phi^n_\delta)_n$ is Cauchy, we have that
	\[
		\limsup_{n,m \to \infty}\; I_3
		\leq \limsup_{n,m \to \infty} \|\Phi^n_\delta - \Phi^m_\delta\|_{L^1((0,T) \times B_{R_1} \times B_{R_1})}
		= 0.
	\]
	Combining the above estimates we obtain, after taking the $\limsup$ as $n, m \to \infty$,
	\[
		\limsup_{n,m \to \infty} \iiint_{K_\alpha \times B_{R_0}} |\Phi^n - \Phi^m| \, dtdxdv
		\leq 2 M |B_{R_0}| \varepsilon + 2 C_1(\varepsilon, \alpha) \delta
	\]
	Taking $\delta \to 0$ and then $\varepsilon \to 0$, we find that $\Phi^n \to \Phi$ in $L^1(K_\alpha \times B_{R_0})$.
	Since $\{(t,x) \in [0,T] \times B_{R_0} : \int \Psi \, dv > 0 \} = \bigcup_{k=1}^\infty K_{1/k}$ and $T, R_0$ are arbitrary, a diagonal argument allows us to extract a subsequence converging almost everywhere on the positive support of $\int \Psi \, dv$.
\end{proof}

Now we bring everything together for a proof of the main proposition of this section.
\begin{proof}[Proof of Proposition \ref{prop-compactness}]
	By the Dunford-Pettis theorem, the estimates \eqref{eq:uniform-a-priori} imply that, passing to a subsequence, we obtain $g^n \weakto g$ in $L^1_{loc}((0,\infty); L^1(\Rxv))$.
	Since the sequence $G^n \equiv g^n(1-g^n)$ is uniformly bounded in $L^\infty((0,\infty) \times \Rxv)$, passing to a further subsequence we obtain that this sequence converges weakly-* to a function $G$.
	Moreover, using that $g^n(1-g^n) \leq g^n$, it follows that $G^n$ also satisfies the estimates \eqref{eq:uniform-a-priori}.
	Therefore, once again by Dunford-Pettis, we can pass to a subsequence to obtain that $G^n \weakto G$ in $L^1_{loc}((0,\infty); L^1(\Rxv))$.

	Lemma \ref{lemma-vel-avg} gives us directly that the velocity averages $\int G^n \varphi \, dv$ are compact in $L^1_{loc}((0,\infty) \times \mathbb{R}^N_x)$.
	As in the proof of Proposition \ref{prop-ae-compactness}, we may pass to a subsequence to obtain convergence in this space for every $\varphi \in L^1(\mathbb{R}^N_v)$.
	Passing to a further subsequence, we also obtain convergence almost everywhere in $(0,\infty) \times \mathbb{R}^N$ for these averages.

	This allows us to apply Lemma \ref{lemma-ellipticity}, which together with inequality \eqref{ineq-aij-deriv} gives us that for each $R, \varepsilon, \alpha > 0$, there exist constants $C, n_0 > 0$ and a measurable set $E$, with $|E| \leq \varepsilon$, such that
	\begin{equation}
		\label{eq:compactness-nabla-gn-estimate}
		\iiint_{(K_\alpha \cap E^c) \times B_R} |\nabla_v g^n|^2 \,dtdxdv \leq C
	\end{equation}
	for every $n \geq n_0$, where $K_\alpha$ is defined by
	\[
		K_{\alpha} = \left\{
			(t,x) \in [0,T] \times B_{R} : \int G \, dv > \alpha
		\right\}.
	\]

	For $T > 0$, define $\mathcal{R}_T = \{(t,x) \in (0,T) \times \mathbb{R}^N_x : \int G \, dv > 0 \;\text{or}\, \int g \, dv = 0 \}$ and let $\mathcal{R}_\infty$ denote the corresponding set when $T = \infty$.
	To prove that $g^n \to g$ almost everywhere on $\mathcal{R}_\infty \times \mathbb{R}^N_v$, we partition the domain as $\mathcal{R}_\infty = \mathcal{R}^+ \cup \mathcal{R}^0$, where $\mathcal{R}^+$ is the subset where $\int G \, dv > 0$ and $\mathcal{R}^0$ is the subset where $\int g \, dv = 0$.

	Applying Proposition \ref{prop-ae-compactness} with $\Phi^n = g^n$ and $\Psi = G$, we can pass to a subsequence such that $g^n \to g$ almost everywhere on $\mathcal{R}^+ \times \mathbb{R}^N_v$.
	On the subset $\mathcal{R}^0$, the nonnegativity of $g$ implies that $g = 0$ almost everywhere.
	Since $g^n \geq 0$, the weak convergence $g^n \weakto 0$ in $L^1_{loc}$ naturally upgrades to strong convergence in $L^1_{loc}(\mathcal{R}^0 \times \mathbb{R}^N_v)$.
	By extracting a further subsequence, we obtain convergence almost everywhere on $\mathcal{R}^0 \times \mathbb{R}^N_v$.
	Therefore, $g^n \to g$ almost everywhere on $\mathcal{R}_\infty \times \mathbb{R}^N_v$.
	By Scheffé's lemma, this implies that $g^n \to g$ in $L^1(\mathcal{R}_T \times \mathbb{R}^N_v)$, for every finite $T > 0$.

	Finally, since $|\nabla_v G^n| \leq |1-2g^n| |\nabla_v g^n| \leq |\nabla_v g^n|$, the bound \eqref{eq:compactness-nabla-gn-estimate} applies to $G^n$ as well.
	We partition the space $(0,\infty) \times \mathbb{R}^N_x$ into $\mathcal{X}^+ \cup \mathcal{X}^0$, where $\mathcal{X}^+$ is the subset where $\int G \, dv > 0$ and $\mathcal{X}^0$ is the subset where $\int G \, dv = 0$.

	Using Proposition \ref{prop-ae-compactness} with $\Phi^n = G^n$ and $\Psi = G$, we extract a subsequence such that $G^n$ converges almost everywhere in $\mathcal{X}^+ \times \mathbb{R}^N_v$.
	On $\mathcal{X}^0$, the nonnegativity of $G$ implies $G = 0$ almost everywhere.
	Since $G^n \geq 0$, we upgrade the weak convergence $G^n \weakto 0$ to strong convergence in $L^1_{loc}(\mathcal{X}^0 \times \mathbb{R}^N_v)$.
	By passing to a further subsequence, this yields convergence almost everywhere on this set.

	Therefore, $G^n$ converges almost everywhere on $(0,\infty) \times \Rxv$, and a new application of Scheffé's lemma yields convergence in $L^1_{loc}((0,\infty); L^1(\Rxv))$.
\end{proof}

\section{Cauchy theory for the approximated equation}
\label{sec:cauchy-approximated-equation}

In this section, we construct an approximation scheme for the LFD equation.
Our primary focus is to develop a sequence of approximate solutions that satisfies the hypotheses from Proposition \ref{prop-compactness}, allowing us to recover the LFD equation in the limit, a step we will detail in Section \ref{sec:existence-theorem-indicator}.

To construct this scheme, we regularize the equation by introducing a smooth collision kernel $a$ and smooth initial data $f_0$.
Additionally, we include a vanishing viscosity term $\varepsilon \Delta_v f$, which will guarantee sufficient regularity for the resulting solutions.

To ensure our approximations remain consistent with conservation laws, thus satisfying the bounds in \eqref{eq:uniform-a-priori}, we impose the following minimal hypotheses on the collision kernel $a$.
Specifically, for every $z \in \mathbb{R}^N$, we assume that
\begin{equation}
	\label{property:approx-a}
	a(z) \geq 0, \quad
	a(-z) = a(z), \quad
	a(z) z = 0, \quad \text{and} \quad
	\sum_{i,j=1}^N \partial^2_{z_i z_j} a_{ij}(z) \leq 0.
\end{equation}
where the notation $a(z)\geq 0$ denotes that the matrix $a(z)$ is positive semidefinite.

The first three hypotheses will guarantee (see Proposition \ref{prop-approx-a-priori-estimates} below) that the approximate equation satisfies the approximate conservation laws, as well as the entropy inequality.
The final condition will be important to ensure that the approximate solutions also satisfy the Pauli bound $0 \leq f \leq 1$.

Our primary objective in this section is to establish the following existence result for the regularized problem.
\begin{proposition}
	\label{prop-existence-approx}
	Suppose $a \in C^\infty(\mathbb{R}^N) \cap W^{3,\infty}(\mathbb{R}^N)$ satisfies \eqref{property:approx-a}, and the initial data $f_0 \in \mathcal{S}(\Rxv)$ satisfies
	\begin{equation}
		\label{eq:estimate-f_0-reglarized}
		0 \leq f_0(x,v) \leq
		\frac{C'_1 e^{-\alpha_0 \weightxv}}{1 + C'_1 e^{-\alpha_0 \weightxv}}
	\end{equation}
	for some constants $C'_1, \alpha_0 > 0$.
	Then, there exists a weak solution $f \in C([0,\infty); \mathcal{D}'(\Rxv)) \cap L^2((0,\infty) \times \Rx; H^1(\Rv))$ to the approximated problem
	\begin{equation}
		\label{regularized}
		\begin{cases}
			\displaystyle
			\pdv{f}{t} + v_i \pdv{f}{x_i} =
			\pdv{}{v_j} \left\{ \abar^f_{ij} \pdv{f}{v_i} - \bbar^f_j f (1-f) \right\} + \varepsilon \Delta_v f\\
			f|_{t=0} = f_0,
		\end{cases}
	\end{equation}
	where the coefficients are defined by \eqref{eq:definition-aij-g-bj-g}.

	The solution $f$ is continuous in $[0,\infty) \times \Rxv$ and is $C^2$ in the velocity variable $v$, with the second derivatives $\nabla^2_v f$ being locally H\"{o}lder continuous in $(0,\infty) \times \Rxv$.
	Moreover, for every $ T > 0$ there exist constants $C, C_1, C_2, \alpha > 0$ such that $f$ satisfies the uniform derivative bounds
	\[
		\| \nabla_v f \|_{L^\infty((0,T) \times \Rxv)} + \| \nabla_v^2 f \|_{L^\infty((0,T) \times \Rxv)} \leq C_1,
	\]
	as well as the pointwise decay estimates
	\[
		0 \leq f(t,x,v) \leq
		\frac{C_2 e^{-\alpha \weightxv}}{1 + C_2 e^{-\alpha \weightxv}}
	\]
	for a.e. $(t,x,v) \in (0,T) \times \Rxv$.
\end{proposition}

Our approach will be to construct an iterative scheme in a linear equation.
Given $f_0 \in \mathcal{S}(\Rxv)$ and $g \in C^\infty((0,T) \times \Rxv)$, consider the linear approximated problem
\begin{equation}
	\label{regularized-linear}
	\begin{cases}
		\displaystyle
		\pdv{f}{t} + v_i \pdv{f}{x_i} =
		\pdv{}{v_j} \left\{ \abar^g_{ij} \pdv{f}{v_i} \right\}
		- \bbar^g_j (1-2g) \pdv{f}{v_j}
		- \pdv{\bbar^g_j}{v_j} g (1-f)
		+ \varepsilon \Delta_v f,\\
		f|_{t=0} = f_0,
	\end{cases}
\end{equation}
where
\begin{equation}
	\label{eq:definition-aij-g-bj-g}
	\abar^g_{ij} \coloneq \frac{a_{ij} *_v (g(1-g))}{1 + \varepsilon \| g(t,x) \|_{L^1}}
	\quad\text{and}\quad
	\bbar^g_j \coloneq \frac{(\partial_{v_i}  a_{ij}) *_v g}{1 + \varepsilon \| g(t,x) \|_{L^1}}.
\end{equation}
The norm in the denominators is taken with respect to the $v$ variable, meaning $\| g(t,x) \|_{L^1} = \int_{\mathbb{R}^N} |g(t,x,v)| \, dv$.

For the remainder of this section, we assume the regularity hypotheses of Proposition \ref{prop-existence-approx} hold, namely that $f_0 \in \mathcal{S}(\Rxv)$ and $a \in C^\infty(\mathbb{R}^N) \cap W^{3,\infty}(\mathbb{R}^N)$.

The solution to the nonlinear problem \eqref{regularized} is obtained by passing to the limit $n\to\infty$.
To justify this convergence, we first establish several properties of the linear equation \eqref{regularized-linear} that can be controlled uniformly.

It has long been observed that although the diffusion acts only in the $v$ variable, the solutions to this type of problem are regular in all variables $(t,x,v)$.
This behavior, called hypoellipticity, was first identified by Kolmogorov \cite{kolmogorov_1934} for the equation $\partial_t f + v \cdot \nabla_x f - \Delta_v f = 0$, while studying a stochastic process, which he was then able to solve using an explicit fundamental solution.

The pursuit of fundamental solutions then expanded to a wider class of operators, such as
\begin{equation}
	\label{example-hypoelliptic}
	\mathcal{L}(f) = - \pdv{f}{t}
		- v_i \pdv{f}{x_i}
		+ a_{ij} \mpdv{f}{v_i}{v_j}
		+ a_i \pdv{f}{v_i}
		+ a f,
\end{equation}
where the diffusion matrix $(a_{ij})_{ij}$ is uniformly elliptic.
Equations of the type $\mathcal{L}(f) = 0$ were extensively studied in 1951 by Weber \cite{weber_1951}, then in 1964 by Il'in \cite{ilin_1964}, who constructed fundamental solutions for these equations that guarantee the existence of classical solutions.

In H\"{o}rmander's work, we find a more precise definition of hypoellipticity: the operator $\mathcal{L}$ is called hypoelliptic if $f$ is smooth wherever $\mathcal{L}(f)$ is smooth.
It was proven in \cite{hormander-1967} that operators of the form \eqref{example-hypoelliptic} with $C^\infty$ coefficients are indeed hypoelliptic.

These classical existence and regularity results imply that if $g \in C^\infty((0,\infty) \times \Rxv)$ and $f_0 \in C^\infty(\Rxv)$, then there exists a solution $f \in C^\infty((0,\infty) \times \Rxv)$ to the linear problem \eqref{regularized-linear} such that $f(t) \to f_0$ when $t \to 0^+$.

To prove Proposition \ref{prop-existence-approx}, we must first establish a sequence of uniform a priori estimates for these smooth solutions.
The crucial first step is to obtain uniform pointwise control over these smooth solutions.
Using a degenerate parabolic comparison principle, the following lemma establishes a stability property for our approximation scheme: it defines a specific exponential decay profile such that, if the background function $g$ respects this profile, the solution $f$ is guaranteed to respect the exact same bounds.

\begin{lemma}
	\label{f-est}
	Let $T > 0$ and $\varepsilon \leq 1$.
	Suppose the initial data $f_0$ satisfies the bound
	\begin{equation}
		\label{eq:comparison-initial-data}
		0
		\leq f_0
		\leq \frac{\exp(k_0 - k_2 \weightxv)}{1 + \exp(k_0 - k_2 \weightxv)}
		\quad \text{on }
		\mathbb{R}^{2N}_{x,v},
	\end{equation}
	for some constants $k_0, k_2 > 0$.

	There exists a constant $k_1 > 0$ (depending only on $T$, $\varepsilon$, $k_2$, and the $L^\infty$ norms of $a$, $\nabla_v a$, and $\nabla_v^2 a$) such that, defining the envelope
	\begin{equation*}
		\varphi(t,x,v) = k_0 + k_1 t - k_2 \exp(-t) \weightxv,
	\end{equation*}
	the following stability property holds:
	if the smooth function $g$ satisfies
	\begin{equation}
		\label{eq:comparison-g-hypothesis}
		0 \leq g \leq \frac{\exp\varphi}{1 + \exp\varphi}
		\quad\text{on }
		[0,T] \times \mathbb{R}^{2N}_{x,v},
	\end{equation}
	then the corresponding solution $f$ to \eqref{regularized-linear} satisfies
	\begin{equation}
		\label{eq:comparison-f-objective}
		0 \leq f \leq \frac{\exp\varphi}{1 + \exp\varphi}
		\quad\text{on }
		[0,T] \times \mathbb{R}^{2N}_{x,v}.
	\end{equation}
\end{lemma}
\begin{proof}

	To establish the bound \eqref{eq:comparison-f-objective}, we apply a comparison principle for the linear operator $\mathcal{L}$ associated with \eqref{regularized-linear}.
	We define $\mathcal{L}$ in its non-divergence form as
	\[
		\mathcal{L}(u) = (\abar_{ij} + \varepsilon \delta_{ij}) \mpdv{u}{v_i}{v_j}
		+ \left(
			\pdv{\abar_{ij}}{v_i} - \bbar_j (1-2g)
		\right) \pdv{u}{v_j}
		+ \pdv{\bbar_j}{v_j} g u
		- \pdv{u}{t}
		- v_i \pdv{u}{x_i},
	\]
	where, for brevity, we have dropped the reference to $g$ in $\abar$ and $\bbar$.

	We then pick a constant $k_1$ such that
	\begin{equation*}
		k_1 \geq k_2 (k_2 + 2) N^2 (\|\abar\|_\infty + 1)
				+ k_2 N (N \| \nabla_v \abar \|_\infty
				+ \| \bbar \|_\infty)
				+ N \| \nabla_v \bbar \|_\infty \exp(T),
	\end{equation*}
	where $\| \cdot \|_\infty$ denotes the $L^\infty$ norm over $(0,T) \times \Rxv$.

	We first establish the upper bound of the inequality \eqref{eq:comparison-f-objective}.
	For $\varphi \in C^\infty((0,\infty) \times \Rxv)$ we have, by direct computation,
	\begin{align*}
		\pdv{}{v_j} \left(\frac{\exp \varphi}{1 + \exp \varphi}\right)
		&= \frac{\exp \varphi}{(1 + \exp \varphi)^2} \pdv{\varphi}{v_j}, \\
		\mpdv{}{v_i}{v_j} \left(\frac{\exp \varphi}{1 + \exp \varphi}\right)
		&= \frac{\exp\varphi}{(1 + \exp\varphi)^2}
		\left(
			\mpdv{\varphi}{v_i}{v_j}
			+ \frac{1 - \exp\varphi}{1 + \exp\varphi}
			\pdv{\varphi}{v_i} \pdv{\varphi}{v_j}
		\right)
	\end{align*}
	with analogous expressions holding for the first derivatives in $t$ and $x_i$.
	Substituting these expressions in $\mathcal{L}$ and using that $f$ is a solution to \eqref{regularized-linear}, we collect terms to find that
	\begin{equation}
	    \label{eq:Lexp/exp+1-f}
	    \begin{split}
	        \mathcal{L}\left(\frac{\exp\varphi}{1 + \exp\varphi} - f\right)
	        = \frac{\exp\varphi}{(1 + \exp\varphi)^2} \Bigg[
	            &(\abar_{ij} + \varepsilon \delta_{ij}) \left( \mpdv{\varphi}{v_i}{v_j} + \frac{1 - \exp\varphi}{1 + \exp\varphi} \pdv{\varphi}{v_i} \pdv{\varphi}{v_j} \right) \\
	        & + \left( \pdv{\abar_{ij}}{v_i} - \bbar_j (1-2g) \right) \pdv{\varphi}{v_j} \\
	        & - \pdv{\bbar_j}{v_j} g \frac{1 + \exp\varphi}{\exp\varphi}
	        - \pdv{\varphi}{t}
	        - v_i \pdv{\varphi}{x_i}
	        \Bigg].
	    \end{split}
	\end{equation}
	From the explicit formula for $\varphi$, the first derivatives are $\partial_{t} \varphi = k_1 + k_2 \exp(-t) \weightxv$ and $\partial_{z_i} \varphi = - k_2 \exp(-t) z_i \weightxv^{-1}$ for $z \in \{x, v\}$.
	The second derivative in $v$ is
	\begin{equation*}
	    \mpdv{\varphi}{v_i}{v_j}
	    = - k_2 \exp(-t) \left( \frac{\delta_{ij}}{\langle x, v \rangle} - \frac{v_i v_j}{\langle x, v \rangle^{3}} \right).
	\end{equation*}
	Substituting these expressions into \eqref{eq:Lexp/exp+1-f}, it follows that
	\[
		\mathcal{L}\left(\frac{\exp \varphi}{1 + \exp \varphi} - f\right) =
		\frac{\exp\varphi}{(1 + \exp\varphi)^2} \exp(-t)
		(A_1 + A_2 + A_3),
	\]
	where the terms in the right-hand side are defined by
	\begin{align*}
		A_1 &= \sum_{i,j = 1}^N (\abar_{ij} + \varepsilon \delta_{ij})
		\bigg(
			- k_2 \frac{\delta_{ij}}{\langle x, v \rangle}
			+ k_2 \frac{v_i v_j}{\langle x, v \rangle^3}
			+ k_2^2 \exp(-t) \frac{1 - \exp\varphi}{1 + \exp\varphi}
			\frac{v_i v_j}{\langle x, v \rangle^2}
		\bigg),\\
		A_2 &= - \sum_{j=1}^N \left(
					\sum_{i=1}^N \pdv{\abar_{ij}}{v_i} - \bbar_j (1-2g)
				\right) k_2 \frac{v_j}{\langle x, v \rangle}
			- \sum_{j=1}^N \pdv{\bbar_j}{v_j} g \exp(t) \frac{1 + \exp\varphi}{\exp\varphi},\\
		A_3 &=
		- k_1 \exp(t)
		- k_2 \langle x, v \rangle
		+ k_2 \frac{|x \cdot v|}{\langle x, v \rangle}.
	\end{align*}
	For clarity in the subsequent estimates, here we have explicitly written out the sums implied by the repeated index convention.
	Let us bound each of the terms $A_1$, $A_2$ and $A_3$.

		For $A_1$, we apply an $L^\infty$ bound in $\abar$ and we use that the functions $\langle x, v\rangle^{-1}$, $|v_i v_j| \langle x, v\rangle^{-3}$ and $|v_i v_j|\langle x, v\rangle^{-2}$ are bounded by $1$.
	Combining these estimates with $\varepsilon \leq 1$ yields
	\[
		A_1 \leq 2 k_2 N^2 (\|\abar\|_\infty + 1)
		+ k_2^2 \exp(-t) N^2 \frac{|1 - \exp\varphi|}{1 + \exp\varphi} (\|\abar\|_\infty + 1).
	\]
	Then, using that $|1 - \exp\varphi| \leq 1 + \exp\varphi$ and that $\exp(-t) \leq 1$, the above estimate implies
	\begin{equation}
		\label{eq:comparison-bound-A1}
		A_1 \leq k_2 (k_2+2) N^2 (\| \abar \|_\infty + 1).
	\end{equation}

	Next, we estimate $A_2$.
	We bound the first term by applying $L^\infty$ estimates on $\bbar$ and on $\nabla_v \abar$, alongside the bounds $|v_j| \langle x, v \rangle^{-1} \leq 1$ and $|1 - 2g| \leq 1$.
	For the second term, we apply an $L^\infty$ estimate on $\nabla_v \bbar$ and we use the bound \eqref{eq:comparison-g-hypothesis} on $g$.
	Using that $t \leq T$, we obtain
	\begin{equation}
		\label{eq:comparison-bound-A2}
		A_2 \leq k_2 N (N \| \nabla_v \abar \|_\infty + \| \bbar \|_\infty)
			+ N \| \nabla_v \bbar \|_\infty \exp(T).
	\end{equation}

	Finally, for $A_3$, we use Young's inequality on the last term, together with the fact that $\exp(t) \geq 1$ to find that
	\begin{equation}
		\label{eq:comparison-bound-A3}
		A_3 \leq - k_1 \exp(t) - k_2 \langle x, v \rangle
		+ \frac{k_2}{2} \frac{|x|^2 + |v|^2}{\langle x, v \rangle}
		\leq - k_1.
	\end{equation}
	Combining the estimates \eqref{eq:comparison-bound-A1}, \eqref{eq:comparison-bound-A2} and \eqref{eq:comparison-bound-A3}, we conclude that
	\begin{equation*}
	    \begin{split}
	        \mathcal{L}\left(\frac{\exp \varphi}{1 + \exp \varphi} - f\right)
	        &\leq \frac{\exp\varphi}{(1 + \exp\varphi)^2} \exp(-t) \bigg[
	            k_2 (k_2 + 2) N^2 (\|\abar\|_\infty + 1) \\
			&\qquad + k_2 N (N \| \nabla_v \abar \|_\infty + \| \bbar \|_\infty)
			+ N \| \nabla_v \bbar \|_\infty \exp(T) - k_1
			\bigg].
	    \end{split}
	\end{equation*}
	The choice of the constant $k_1$ guarantees that the quantity inside the brackets is non-positive.
	Hence, we have
	\[
		\mathcal{L}\left(\frac{\exp \varphi}{1 + \exp \varphi} - f\right) \leq 0.
	\]
	The condition for the initial data implies that
	\[
		f_0(x,v) \leq \frac{\exp\varphi (0,x,v)}{1 + \exp\varphi (0,x,v)}.
	\]
	Applying the maximum principle on $[0,T] \times \Rxv$ (see, for example, Chapter 2, Theorems 2 and 3 of \cite{friedman_1964}), we have the upper bound of the inequality \eqref{eq:comparison-f-objective}.

	For the lower bound, notice that by the definition of $\bbar$ and the hypothesis \eqref{property:approx-a}, we have
	\[
		\mathcal{L}(f) = \pdv{\bbar_j}{v_j} g \leq 0,
	\]
	and since $f_0 \geq 0$, the result follows again from the maximum principle.

	Finally, using the definitions of $\abar_{ij}$ and $\bbar_j$, together with Young's inequality for convolutions and the fact that $x/(1+\varepsilon x) \leq \varepsilon^{-1}$, we have the uniform bounds
	\begin{equation}
	\label{eq:maximum-principle-bounds-abar-bbar}
	\begin{aligned}
		&\| \abar_{ij} \|_\infty
		\leq \varepsilon^{-1} \| a_{ij} \|_\infty
		&\| \bbar_j \|_\infty
		\leq \varepsilon^{-1} N \| \nabla_v a \|_\infty \\
		&\| \nabla_v \abar_{ij} \|_\infty
		\leq \varepsilon^{-1} \| \nabla_v a_{ij} \|_\infty
		&\| \nabla_v \bbar_j \|_\infty
		\leq \varepsilon^{-1} N \| \nabla^2_v a \|_\infty
	\end{aligned}
	\end{equation}
	These imply that $k_1$ can be chosen independently of $g$.
\end{proof}

We next establish uniform $L^\infty$ bounds for the derivatives of $f$ with respect to the velocity variable $v$.
Our strategy relies on a two-step bootstrap argument.
First, we employ hypoelliptic De Giorgi--Nash--Moser theory to obtain Hölder continuity for the solution $f$.
Treating the governing equation as a linear operator with H\"{o}lder continuous coefficients, we subsequently apply Schauder estimates for hypoelliptic equations in order to achieve extra regularity in $v$.
This yields the following result.
\begin{lemma}
	\label{derivative-infty-est}
	Let $T > 0$ and suppose $0 \leq g \leq 1$.
	Let $f$ be a solution to \eqref{regularized-linear} satisfying \eqref{eq:comparison-f-objective}, with initial data $f_0 \in \mathcal{S}(\Rxv)$.
	Then, there exist an exponent $\alpha \in (0,1)$ and a constant $C_1 > 0$ such that
	\[
		\| f \|_{L^\infty([0,T);\, C^\alpha(\Rxv))} \leq C_{1}.
	\]
	Furthermore, if $0 \leq g \leq 1$ satisfies the additional regularity estimate
	\[
		\| g \|_{L^\infty((0,T);\, C^\alpha(\Rxv))} \leq C_2,
	\]
	and has the decay in $v$
	\[
		g \leq e^{-C_3 |v|}
		\quad\text{in }
		(0,T) \times \Rxv
	\]
	for some constants $C_2, C_3 > 0$, then there exists a constant $C_4 > 0$ such that
	\begin{equation}
		\label{eq:derivative-infty-bound}
		\| \nabla_v f \|_{L^\infty((0,T) \times \Rxv)} + \| \nabla_v^2 f \|_{L^\infty((0,T) \times \Rxv)} \leq C_4.
	\end{equation}

	Finally, there exists an exponent $\gamma \in (0,1)$ such that for each compact set $K \subset (0,T) \times \Rxv$, there exists a constant $C_K > 0$ such that
	\[
		\| \nabla^2_v f \|_{C^{\gamma}(K)} \leq C_K.
	\]

	The constants $\alpha$ and $C_1$ depend only on $N$, $\varepsilon$, $f_0$, and the $L^\infty$ norms of $a$,  $\nabla a$ and $\nabla^2 a$, with $C_1$ additionally depending on $T$.
	The constants $C_4$ and $C_K$ further depend on $C_2$, $C_3$ and the $C^\alpha$ norms of $a$, $\nabla a$ and $\nabla^2 a$.
\end{lemma}
\begin{proof}
	Consider the linear operator
	\[
		\mathcal{L}(f)
		= \partial_t f + v \cdot \nabla_x f
		- \Div_v (A^g \nabla_v f) - B^g \cdot \nabla_v f.
	\]
	where $A^g$ is the matrix given by $A^g_{ij} = \abar^g_{ij} + \varepsilon \delta_{ij}$ and $B^g_j = - \bbar^g_j (1-2g)$.
	Given the inequalities \eqref{eq:maximum-principle-bounds-abar-bbar}, we have that, for all $(t,x,v) \in (0,\infty) \times \Rxv$,
	\begin{equation}
	\label{eq:hypoelliptic-A-B-estimates}
	\begin{cases}
		\varepsilon I \leq A^g(t,x,v) \leq \Lambda I \\
		|B^g(t,x,v)| \leq \Lambda,
	\end{cases}
	\end{equation}
	where $\Lambda = \max\left\{ \varepsilon + N \max_{ij} \| a_{ij} \|_{L^
	\infty}, N^2 \max_{ij} \| \nabla a_{ij} \|_{L^\infty}  \right\}$.

	Let $f$ be a solution to \eqref{regularized-linear} in $(0, \infty) \times \Rxv$.
	It will be useful to work with a function defined for all $t \in \mathbb{R}$, so let us define the extension $\tilde{f}$ as
	\[
		\tilde{f}(t) =
		\begin{cases}
			f_0 \text{ if } t \leq 0,\\
			f(t) \text{ if } t > 0,
		\end{cases}
	\]
	and we also extend $\abar^g$ and $\bbar^g$ by zero for $t \leq 0$, which gives rise to the extensions $\tilde{A^g}$ and $\tilde{B^g}$ of $A^g$ and $B^g$, respectively.
	Crucially, for $t \leq 0$, this implies $\tilde{A^g} = \varepsilon I$ and $\tilde{B^g} = 0$, which guarantee that the bounds \eqref{eq:hypoelliptic-A-B-estimates} remain valid globally with the same constants.

	Let us still denote by $\mathcal{L}$ the operator where the coefficients $A^g$ and $B^g$ are replaced by their extensions $\tilde{A^g}$ and $\tilde{B^g}$.
	Since $f$ reaches its initial data $f_0$ continuously at $t = 0$, the extension $\tilde{f}$ is continuous at $t = 0$.
	Therefore, the weak time derivative $\partial_t \tilde{f}$ is defined piecewise by
	\begin{equation}
		\label{eq:L-of-extension}
		\mathcal{L}(\tilde{f})
		=
		\begin{cases}
			- \pdv{\bbar^g_j}{v_j} g (1-f)
					&\text{ in }
					(0,T) \times \Rxv,\\
			v \cdot \nabla_x f_0 - \varepsilon \Delta_v f_0
				&\text{ in } (-\infty,0) \times \Rxv.
		\end{cases}
	\end{equation}

	We derive the Hölder regularity for solutions $f$ of \eqref{regularized-linear} by applying results from the hypoelliptic De Giorgi--Nash--Moser theory.
	For a point $z_0 = (t_0, x_0, v_0)$, we define the kinetic cylinder centered at $z_0$ with radius $r>0$ as
	\[
		Q_r(z_0)
		= \{
			(t,x,v):
			t \in (t_0 - r^2, t_0],
			|x - x_0 - (t-t_0) v_0| < r^3,
			|v - v_0| < r
		\}.
	\]
	Following the results established in \cite{golse-2019} (in particular, Theorem 1.3), there exists a constant $C = C(N, \varepsilon, \Lambda, R, r)$ and an exponent $\alpha = \alpha(N, \varepsilon, \Lambda) \in (0,1)$ such that if $u$ is bounded in $L^2(Q_R(0))$ and $\mathcal{L}(u)$ is bounded in $L^\infty(Q_r(0))$, then
	\[
		\| u \|_{C^\alpha(Q_r(0))} \leq C
			\left(
			\| u \|_{L^2(Q_R(0))} + \| \mathcal{L}(u) \|_{L^\infty(Q_R(0))}
			\right).
	\]
	Here, $\| \cdot \|_{C^\alpha(Q)}$ denotes the standard H\"{o}lder norm on a cylinder $Q$, defined as
	\[
		\| \varphi \|_{C^\alpha(Q)}
		= \| \varphi \|_{L^\infty(Q)}
		+ \sup_{z_1, z_2 \in Q, z_1 \neq z_2}
			\frac{|\varphi(z_1) - \varphi(z_2)|}{d(z_1, z_2)^\alpha},
	\]
	where $z_i = (t_i, x_i, v_i)$ and $d(z_1, z_2) = |t_1 - t_2| + |x_1 - x_2| + |v_1 - v_2|$.

	We begin by extending these local estimates to a global result.
	For any function $\phi = \phi(t,x,v)$ let $\phi_{z_0}$ denote its Galilean translation by $z_0 = (t_0,x_0,v_0)$, that is,
	\[
		\phi_{z_0} (t,x,v) = \phi(t + t_0, x + x_0 + t v_0, v + v_0).
	\]
	The translated operator
	\[
		\mathcal{L}_{z_0}(f)
		\equiv \partial_t f + v \cdot \nabla_x f
		- \Div_v \left([\tilde{A^g}]_{z_0} \nabla_v f \right) - [\tilde{B^g}]_{z_0} \cdot \nabla_v f
	\]
	satisfies the same estimates \eqref{eq:hypoelliptic-A-B-estimates} as the original operator $\mathcal{L}$.
	Since the operator and the transformation commute in the sense of
	\[
		\mathcal{L}_{z_0}(u_{z_0}) = [\mathcal{L}(u)]_{z_0},
	\]
	it follows that for any $0 < r < R$, there exists a constant $C = C(N, \varepsilon, \Lambda, R, r)$ such that for every $z_0 = (t_0,x_0,v_0)$, we have
	\[
		\| u \|_{C^\alpha(Q_r(z_0))}
		\leq C
			\left(
				\| u \|_{L^2(Q_R(z_0))}
				+ \| \mathcal{L}(u) \|_{L^\infty(Q_R(z_0))}
			\right).
	\]

	Fix $t_0 \in [0, T)$ and observe that for all $0 < r < R$ we have
	\[
		\left\{
			t = t_0,
			|x - x_0| < r^3,
			|v - v_0| < r
		\right\}
		\subset Q_r(z_0)
	\]
	and
	\[
		Q_R(z_0)
		\subset (t_0 - R^2, t_0] \times \Rxv.
	\]
	Applying the local regularity estimate to $u = \tilde{f}$ with radii $r = 1/2$ and $R = 1$, it follows that for every $0 \leq t < T$,
	\[
		\| f(t) \|_{C^\alpha(B_{1/8}(x_0) \times B_{1/2}(v_0))}
		\leq C
			\left(
				\| \tilde{f} \|_{L^2((- 1, T] \times \Rxv)}
				+ \| \mathcal{L}(\tilde{f}) \|_{L^\infty((- 1, T] \times \Rxv)}
			\right).
	\]

	We now estimate the terms on the right-hand side.
	The $L^2$-norm of $\tilde{f}$ is uniformly bounded by a constant depending only on $f_0$, from interpolation between its $L^1$ and $L^\infty$ norms.
	For the $L^\infty$-norm of $\mathcal{L}(\tilde{f})$, we estimate the terms in \eqref{eq:L-of-extension}.
	Splitting the analysis into positive and negative values of $t$ we have, for $t < 0$, that
	\begin{equation*}
		\| \mathcal{L}(\tilde{f}) \|_{L^\infty((- 1, 0) \times \Rxv)}
		\leq \| v \cdot \nabla_x f_0  \|_{L^\infty(\Rxv)}
			+ \varepsilon \| \Delta_v f_0 \|_{L^\infty(\Rxv)}.
	\end{equation*}
	For $t > 0$, using Young's inequality for convolutions together with the bounds \eqref{eq:maximum-principle-bounds-abar-bbar} and the fact that $0 \leq f, g \leq 1$, we obtain that
	\begin{align*}
		\| \mathcal{L}(\tilde{f}) \|_{L^\infty((0, T) \times \Rxv)}
		&\leq N \varepsilon^{-1}
			   \max_{ij} \| \nabla^2 a_{ij} \|_{L^\infty}.
	\end{align*}

	Combining these estimates, for every $(x_0, v_0) \in \Rxv$ there exists a constant $C'_1$ (depending only on $T$, $N$, $\varepsilon$, $f_0$ and the $L^\infty$ norms of $a$, $\nabla a$ and $\nabla^2 a$) such that
	\[
		\| f \|_{L^\infty([0,T);\, C^\alpha(B_{1/8}(x_0) \times B_{1/2}(v_0)))} \leq C'_1.
	\]
	The global inequality in $\Rxv$ from the statement then follows from the uniformity of this bound with respect to $x_0$ and $v_0$.
	Indeed, fix $t \in [0,T)$ and let $(x_1,v_1), (x_2,v_2) \in \Rxv$.

	If $|x_1 - x_2| \geq 1/8$ or $|v_1 - v_2| \geq 1/2$ then the denominator of the H\"{o}lder quotient is larger than $(1/8)^\alpha$.
	Estimating the numerator with an $L^\infty$ bound on $f$ we obtain that
	\[
		\frac{|f(t,x_1,v_1) - f(t,x_2,v_2)|}{(|x_1 - x_2| + |v_1 - v_2|)^\alpha}
		\leq 2^{3\alpha + 1} \| f \|_{L^\infty((0,T) \times \Rxv)}.
	\]
	On the other hand, when $|x_1 - x_2| < 1/8$ and $|v_1 - v_2| < 1/2$, the bound on this quotient follows trivially from the preceding local estimate applied at $(x_0, v_0) = (x_1, v_1)$.
	Combining these two cases yields the desired global bound.

	The second inequality from the statement is derived from Schauder estimates for hypoelliptic equations.
	Specifically, if we supplement our previous assumptions with an additional $C^\alpha$ bound on $g$, this regularity propagates to the coefficients $A^g_{ij}$ and $B^g_j$.

	Indeed, for $A^g_{ij}$, notice that expanding the convolution then using that $1 - g \leq 1$, we obtain
	\begin{equation*}
		\begin{split}
			| \abar^g_{ij}(t,x_1, v_1) - \abar^g_{ij}(t,x_1,v_2)|
			&\leq \frac{1}{1 + \varepsilon \| g(t,x_1) \|_{L^1}}
			\| a_{ij} \|_{C^\alpha(\mathbb{R}^N)} |v_1 - v_2|^\alpha \|[g(1-g)](t,x_1)\|_{L^1_v} \\
			&\leq \varepsilon^{-1} \| a_{ij} \|_{C^\alpha(\mathbb{R}^N)} |v_1 - v_2|^\alpha.
		\end{split}
	\end{equation*}
	On the other hand, denoting $\tilde{a} \coloneq a_{ij} * (g(1-g))$ and using that the denominator $1 + \varepsilon \| g(t,x) \|_{L^1_v} \geq 1$, we have
	\begin{equation*}
	    \begin{split}
	    	| \abar^g_{ij}(t,x_1, v_2) - \abar^g_{ij}(t,x_2,v_2)|
	    	&\leq
	    	\left|
	    		\tilde{a}(t, x_1, v_2) - \tilde{a}(t, x_2, v_2)
	    	\right| \\
	    	&\quad + \frac{\varepsilon \tilde{a}(t, x_1, v_2)}{1 + \varepsilon \| g(t,x_1) \|_{L^1_v}}
	    	\left|
	    		\int g(t,x_2, v_*) - g(t,x_1, v_*) \, dv_*
	    	\right| \\
	    	&\leq 2 \| a_{ij} \|_{L^\infty} \int |g(t,x_2, v_*) - g(t,x_1, v_*)| \, dv_*.
	    \end{split}
	\end{equation*}
	We estimate this last integral by interpolation.
	Fix $\theta \in (0,1)$.
	We have, from the H\"{o}lder regularity and the exponential decay of $g$, that there exists a constant $C > 0$ such that
	\begin{equation*}
		\begin{split}
			|g(t,x_2, v_*) - g(t,x_1, v_*)|
			&= |g(t,x_2, v_*) - g(t,x_1, v_*)|^\theta |g(t,x_2, v_*) - g(t,x_1, v_*)|^{1-\theta} \\
			& \leq C \|g(t)\|_{C^\alpha(\Rxv)}^\theta |x_2 - x_1|^{\theta\alpha} e^{- (1-\theta) C_3 |v_*|} \\
			& \leq C C_2^\theta |x_2 - x_1|^{\theta\alpha} e^{- (1-\theta) C_3 |v_*|}
		\end{split}
	\end{equation*}
	Integrating the above inequality over $v_* \in \mathbb{R}^N$ and substituting in the former inequalities, it follows that $\abar^g_{ij}$ is a $L^\infty((0,T); C^{\theta \alpha}(\Rxv))$ function, whose norm depends on $\varepsilon, \theta, C_2, C_3$ and the $C^\alpha$ norm of $a_{ij}$, but not on $g$.

	Similarly, one can show the same result for $\bbar^g_j$, with the $L^\infty((0,T); C^{\theta \alpha}(\Rxv))$ norm depending on the same parameters, but additionally on the $C^\alpha$ norm of $\nabla a_{ij}$, and for $\partial_{v_j} b^g_j$, with the norm depending additionally on $C^\alpha$ norm of $\nabla^2 a_{ij}$.
	Let us now call this new regularity $\beta = \theta \alpha < \alpha$.

	Therefore, by applying Theorem 1.6(ii) from \cite{dong2025} (refer also to Theorem 1.1 in \cite{henderson2024}), there exist constants $C_0 > 0$ and $\lambda_0 \geq 1$ (depending only on $N$, $\varepsilon$, $\Lambda$, $C_2$, and the $C^\alpha$ norms of $a$ and $\nabla a$) such that
	\begin{equation*}
		\begin{split}
			\lambda_0^\beta \| \nabla^2_v u \|_{L^\infty(\mathcal{S}_T)}
			+ [ \nabla^2_v u ]_{C^\beta_{\mathrm{kin}}(\mathcal{S}_T)}
			\leq C_0
				\bigg(
					[ \mathcal{L}(u) + \lambda_0 u &]_{L^\infty((-\infty,T);\, C^{\beta/3, \beta}_{x,v}(\Rxv))} \\
					&+ \lambda_0^\beta \| \mathcal{L}(u) + \lambda_0 u \|_{L^\infty(\mathcal{S}_T)}
				\bigg),
		\end{split}
	\end{equation*}
	where $\mathcal{S}_T \coloneq (-\infty, T) \times \Rxv$.
	Because $\lambda_0$ is a strictly positive constant, we can absorb its powers into a new universal constant $C_1$, yielding the bound
	\begin{equation}
		\label{eq:schauder-hypoelliptic}
		\begin{split}
			\| \nabla^2_v u \|_{C^\beta_{\mathrm{kin}}(\mathcal{S}_T)}
			\leq C_1
				\bigg(
					[ \mathcal{L}(u) + \lambda_0 u ]_{L^\infty((-\infty,T);\, C^{\beta/3, \beta}_{x,v}(\Rxv))}
					+ \| \mathcal{L}(u) + \lambda_0 u \|_{L^\infty(\mathcal{S}_T)}
				\bigg),
		\end{split}
	\end{equation}
	Here, the left-hand side contains the kinetic H\"{o}lder norm
	\[
		\| \psi \|_{C^\beta_{\mathrm{kin}}(\mathcal{S}_T)}
		= \| \psi \|_{L^\infty(\mathcal{S}_T)}
		+ \sup_{z_1 \neq z_2 \in \mathcal{S}_T}
			\frac{| \psi(z_1) - \psi(z_2) |}{\rho(z_1, z_2)^\beta},
	\]
	where the supremum is taken over all distinct pairs $z_1 = (t_1, x_1, v_1)$ and $z_2 = (t_2, x_2, v_2)$ in $\mathcal{S}_T$ and $\rho(z_1, z_2)$ is the kinetic quasi-distance (as defined in \cite{dong2025}),
	\[
		\rho(z_1, z_2)
		= \max \Big\{
			|t_1 - t_2|^{1/2},
			|x_1 - x_2 - (t_1 - t_2) v_2|^{1/3},
			|v_1 - v_2|
		\Big\}.
	\]
	On the right-hand side of \eqref{eq:schauder-hypoelliptic}, we have the $L^\infty((-\infty,T);\, C^{\beta/3, \beta}_{x,v}(\Rxv))$ semi-norm, defined by
	\[
		[\psi]_{L^\infty((-\infty,T);\, C^{\beta/3, \beta}_{x,v}(\Rxv))}
		= \sup_{t < T} \sup_{(x_1, v_1) \neq (x_2, v_2)}
			\frac{|\psi(t,x_1,v_1) - \psi(t,x_2,v_2)|}{(|x_1 - x_2|^{1/3} + |v_1 - v_2|)^\beta}.
	\]

	Furthermore, distinguishing the cases  $|x_1 - x_2| \leq 1$ and $|x_1 - x_2| > 1$, we can show that
	\[
		[\psi]_{L^\infty((-\infty,T);\, C^{\beta/3, \beta}_{x,v}(\Rxv))}
		\leq 3 \| \psi \|_{L^\infty((-\infty,T);\, C^{\beta}(\Rxv))}.
	\]
	Therefore, applying the triangle inequality to the right-hand side of inequality \eqref{eq:schauder-hypoelliptic} for $u = \tilde{f}$ and bounding all terms with the $L^\infty((-\infty,T);\, C^{\beta}(\Rxv))$ norm, there exists a constant $C_2 > 0$ (depending on the same parameters as the previously stated constant $C_1$) such that
	\begin{equation}
		\label{eq:D2f-estimation-f}
		\| \nabla^2_v \tilde{f} \|_{C^\beta_{\mathrm{kin}}(\mathcal{S}_T)}
		\leq C_2
			\left(
				\| \mathcal{L}(\tilde{f}) \|_{L^\infty((-\infty,T);\, C^{\beta}(\Rxv))}
				+ \| \tilde{f} \|_{L^\infty((-\infty,T);\, C^{\beta}(\Rxv))}
			\right).
	\end{equation}

	Splitting the analysis between the intervals $t \leq 0$ and $0 < t < T$, we first observe that the extension $\tilde{f}$ satisfies the global bound
	\begin{align*}
	    \| \tilde{f} \|_{L^\infty((-\infty, T);\, C^\beta(\Rxv))}
	    &\leq \| f_0 \|_{C^\beta(\Rxv)}
	    + \| f \|_{L^\infty((0,T);\, C^\beta(\Rxv))}\\
	    &\leq \| f_0 \|_{C^\beta(\Rxv)}
	    + C_1.
	\end{align*}
	We now estimate the source term $\mathcal{L}(\tilde{f})$ on each interval.
	For $t > 0$, using \eqref{eq:L-of-extension} and the uniform bounds on $g$ and $f$, we obtain
	\begin{equation*}
	\begin{split}
	    \| \mathcal{L}(\tilde{f}) \|_{L^\infty((0, T);\, C^\beta(\Rxv))}
	    \leq \left\| \sum_j \pdv{\bbar^g_j}{v_j} \right\|_{L^\infty((0,T);\, C^\beta(\Rxv))}
	    	&\| g \|^2_{L^\infty((0,T);\, C^\beta(\Rxv))} \\
	        &\quad\times \left(
	            1 + \| f \|_{L^\infty((0,T);\, C^\beta(\Rxv))}
	        \right).
	\end{split}
	\end{equation*}
	The first term is bounded by a constant depending only on $\varepsilon, \theta, C_2, C_3$ and the $C^\beta$ norm of $\nabla^2 a$, while the other terms are bounded by hypothesis.
	On the other hand, for $t < 0$, the operator acts on the initial data, yielding the estimate
	\[
	    \| \mathcal{L}(\tilde{f}) \|_{L^\infty((-\infty, 0);\, C^\beta(\Rxv))}
	    \leq \| v \cdot \nabla_x f_0  \|_{C^\beta(\Rxv)}
	    + \varepsilon \| \Delta_v f_0 \|_{C^\beta(\Rxv)}.
	\]
	Substituting these bounds into \eqref{eq:D2f-estimation-f}, we conclude that there exists a constant $C_4$ (depending only on $T$, $C_2$, $N$, $\varepsilon$, $\theta$, $f_0$ and the $C^\beta$ norms of $a$, $\nabla a$ and $\nabla^2 a$) such that
	\[
	    \| \nabla^2_v f \|_{C^\beta_{\mathrm{kin}}((0, T) \times \Rxv)} \leq C_4,
	\]
	which implies the $L^\infty((0, T) \times \Rxv)$ bound on $\nabla^2_v f$.
	Since $f$ is uniformly bounded in $L^\infty$, the bound for the first derivative then follows by interpolation.

	Finally, let us prove the local H\"{o}lder bound for $\nabla^2_v f$ by comparing the kinetic quasi-distance $\rho$ to the standard  distance $d(z_1, z_2) = |t_1 - t_2| + |x_1 - x_2| + |v_1 - v_2|$.

	Let $K \subset (0,T) \times \Rxv$ be a compact set.
	There exists a constant $M > 0$ such that $|v| \leq M$ for all $(t,x,v) \in K$.
	We split the supremum of the standard H\"{o}lder quotient into two cases.
	If $d(z_1, z_2) > 1$, the denominator is strictly greater than $1$, and the quotient is trivially bounded by $2 \| \nabla^2_v f \|_{L^\infty(K)}$.

	If $d(z_1, z_2) \leq 1$, we have for the time component of $\rho$ that $|t_1 - t_2|^{1/2} \leq |t_1 - t_2|^{1/3} \leq d(z_1, z_2)^{1/3}$.
	For the velocity component, $|v_1 - v_2| \leq |v_1 - v_2|^{1/3} \leq d(z_1, z_2)^{1/3}$.
	Lastly, for the spatial component, we have by the triangle inequality that
	\[
		\begin{split}
			| x_1 - x_2 - (t_1 - t_2) v_2 |^{1/3}
			&\leq \big(| x_1 - x_2 | + |t_1 - t_2| |v_2| \big)^{1/3} \\
			&\leq \max\{1, M\}^{1/3} d(z_1, z_2)^{1/3}.
		\end{split}
	\]
	Therefore, taking the maximum of these components, we find that $\rho(z_1, z_2) \leq C_M d(z_1, z_2)^{1/3}$, for a constant $C_M > 0$ depending only on $M$.
	Combining this estimate with the estimate where $d(z_1, z_2) > 1$, we obtain the local embedding $C^\beta_{\mathrm{kin}}(
	(0,T) \times \Rxv) \hookrightarrow C^{\beta/3}(K)$.
	Setting $\gamma \coloneq \beta / 3$ yields the local H\"{o}lder estimate for $\nabla^2_v f$ and concludes the proof.
\end{proof}

Now, we can use the bounds shown above to build an iterative process, which we will show converges to a solution of the approximate problem \eqref{regularized}.

\begin{proof}[Proof of Proposition \ref{prop-existence-approx}]
	\proofpart{1}{The iterative process and uniform estimates}

	Let $f^0$ be the solution of the Kolmogorov equation $\partial_t f + v \cdot \nabla_x f - \varepsilon \Delta_v f = 0$ with initial data $f_0$.
	For $k \geq 1$, let $f^{k}$ be the solution of the problem \eqref{regularized-linear} with $g = f^{k-1}$.
	That is, we consider the iterative process
	\begin{equation}
		\label{iteration-eq-k}
		\begin{cases}
			\displaystyle
			\pdv{f^k}{t} + v_i \pdv{f^k}{x_i} =
			\pdv{}{v_j} \left\{ \abar^{k-1}_{ij} \pdv{f^k}{v_i} \right\}
			- \bbar^{k-1}_j (1-2f^{k-1}) \pdv{f^k}{v_j}
			- \pdv{\bbar^{k-1}_j}{v_j} f^{k-1} (1-f^k)
			+ \varepsilon \Delta_v f^k\\

			f^k|_{t=0} = f_0,
		\end{cases}
	\end{equation}
	where, to simplify notation, we write $\abar^{k-1}$ and $\bbar^{k-1}$ for $\abar^{f^{k-1}}$ and $\bbar^{f^{k-1}}$, respectively.
	As mentioned at the beginning of the section, for every $k \in \mathbb{N}$, this problem admits a smooth solution $f^k \in C^\infty((0,\infty) \times \Rxv)$.
	We can extend this function to $C([0,\infty) \times \Rxv)$ by continuity, which ensures $f^k(0) = f_0$.

	Let us establish the decay properties for $f^k$.
	Fix a final time $T > 0$.
	From \eqref{eq:estimate-f_0-reglarized}, we have that $f_0$ satisfies
	\[
		0 \leq f^0 \leq \frac{\exp(k_0 - k_2 \weightxv)}{1 + \exp(k_0 + k_1 t - k_2 \weightxv)}
	\]
	with $k_0 = \log C_1$ and $k_2 = \alpha_0$.
	We define the envelope $\varphi(t,x,v) = k_0 + k_1 t - k_2 \exp(-t) \weightxv$, where $k_1 > 0$ is the constant given by Lemma \ref{f-est} (which depends only on $T$, $k_2$, and the properties of $a$, but is entirely independent of $g$).

	We proceed by induction on $k$.
	For the base case $k=0$, the function $f^0$ solves the Kolmogorov equation, which corresponds to the problem \eqref{regularized-linear} with $g = 0$.
	Applying Lemma \ref{f-est}, it immediately follows that $f^0$ is bounded by
	\begin{equation*}
		0 \leq f^0(t,x,v) \leq \frac{\exp\varphi(t,x,v)}{1 + \exp\varphi(t,x,v)}
		\quad\text{on }
		[0,T] \times \mathbb{R}^{2N}_{x,v}.
	\end{equation*}
	For the inductive step, assume that this bound holds for $f^{k-1}$.
	We can therefore apply Lemma \ref{f-est} with $g = f^{k-1}$ to conclude that the solution $f^k$ satisfies the exact same bound.
	By induction, this bound holds uniformly for all $k \in \mathbb{N}$.
	Because $t \leq T$, we can bound $\exp\varphi$, implying there exist constants $\alpha = k_2 e^{-T}$ and $C_1 = k_0 + k_1 T$, which depend on $T$ but not on $k$ such that, for all $k \in \mathbb{N}$,
	\begin{equation}
		\label{ineq-iteration-f-est}
		0
		\leq f^k(t,x,v)
		\leq \frac{C_1 e^{-\alpha \weightxv}}{1 + C_1 e^{-\alpha \weightxv}}
		\quad\text{on }
		[0,T] \times \mathbb{R}^{2N}_{x,v}.
	\end{equation}

	Furthermore, by applying Lemma \ref{derivative-infty-est}, we establish the existence of a constant $\gamma, C_2 > 0$, depending on $T$ but uniform in $k$, such that
	\begin{equation}
		\label{eq:uniform-holder-fk}
	    \| f^k \|_{L^\infty([0,T);\, C^\gamma(\Rxv))} \leq C_2.
	\end{equation}
	From this, together with the decay estimate uniform on $f^k$ \eqref{ineq-iteration-f-est}, we have from the second conclusion of Lemma \ref{derivative-infty-est} that there exists a constant $C_3$, uniform in $k$, such that
	\begin{equation}
		\label{eq:uniform-deriv-fk-est}
	    \| \nabla_v f^k \|_{L^\infty((0,T) \times \Rxv)} + \| \nabla_v^2 f^k \|_{L^\infty((0,T) \times \Rxv)} \leq C_3.
	\end{equation}
	It then follows from Glaeser's inequality applied to the nonnegative functions $f^k$, together with the inequality \eqref{ineq-iteration-f-est}, that there exists a constant $C_4$ such that, for every $i = 1, \dots, N$ we have the pointwise inequality
	\begin{equation}
		\label{eq:iteration-pointwise-derivatives-bound}
		\left| \pdv{f^k}{v_i} \right|^2
		\leq 2 C_3 f^k
		\leq C_4 e^{-\alpha \weightxv}
		\quad\text{on }
		(0,T) \times \mathbb{R}^{2N}_{x,v}.
	\end{equation}

	\proofpart{2}{Convergence of $f^k$}

	We show that $(f^k)_k$ is a Cauchy sequence in $L^\infty_{loc}((0,\infty); L^2(\Rxv))$ using energy estimates.
	For the sake of notational brevity, let us denote the space $L^2(\Rxv)$ by $L^2_{xv}$.
	Furthermore, in this step, we suppress the explicit time dependence of the functions $w^k$, $f^k$ and the coefficients $\abar^k$ and $\bbar^k$.
	Subtracting equation \eqref{iteration-eq-k} for $f^{k+1}$ from equation \eqref{iteration-eq-k} for $f^k$, we have
	\begin{equation*}
	\begin{split}
		\pdv{w^{k+1}}{t} + v_i \pdv{w^{k+1}}{x_i}
		&= \pdv{}{v_j} \left\{
			\Abar^{k}_{ij} \pdv{w^{k+1}}{v_i}
			+ ( \abar^{k}_{ij} - \abar^{k-1}_{ij} ) \pdv{f^k}{v_i}
		\right\}
		- \bbar^{k}_j (1-2f^k) \pdv{w^{k+1}}{v_j}\\
		&\quad + \pdv{\bbar^k_j}{v_j} f^k w^{k+1}
		 + \left[
			2 \bbar^k_j \pdv{f^k}{v_j} - \pdv{\bbar^k_j}{v_j} (1 - f^k)
		\right]	w^k\\
		&\quad - (\bbar^k_j - \bbar^{k-1}_j) (1 - 2 f^{k-1}) \pdv{f^k}{v_j}
		- \pdv{(\bbar^k_j - \bbar^{k-1}_j)}{v_j} (1 - f^k) f^{k-1}
	\end{split}
	\end{equation*}
	where $w^{k+1} = f^{k+1} - f^k$ and $\Abar^{k}_{ij} = \abar^{k}_{ij} + \varepsilon \delta_{ij}$.
	Consider the function $\beta(t) = \beta_0 e^{-\lambda t}$, with $\beta_0 < \alpha$ and $\lambda \geq 1/2$.
	Multiplying the equation by $w^{k+1} e^{\beta(t) \weightxv}$ and integrating by parts for $(x, v) \in \Rxv$ yields
	\begin{equation}
	    \label{eq:cauchy-dt-estimate}
	    \frac{1}{2} \frac{d}{dt} \intxv (w^{k+1})^2 e^{\beta(t) \weightxv}
	    = I_1 + I_2 + I_3 + I_4
	\end{equation}
	where
	\begin{align*}
	    I_1 &= - \intxv \left[
	        \Abar^{k}_{ij} \pdv{w^{k+1}}{v_i}
	        + ( \abar^{k}_{ij} - \abar^{k-1}_{ij} ) \pdv{f^k}{v_i}
	    \right]
	    \left( \pdv{w^{k+1}}{v_j} + \beta w^{k+1} \frac{v_j}{\weightxv} \right) e^{\beta \weightxv}
	    \\
	    I_2 &= \intxv \left[
	        \frac{1}{2} \pdv{\bbar^k_j}{v_j} - \bbar^k_j \pdv{f^k}{v_j}
	        + \frac{\beta}{2} \bbar^k_j (1-2f^k) \frac{v_j}{\weightxv}
	        + \frac{\beta}{2} \frac{x_i v_i}{\weightxv}
	        - \frac{\lambda}{2} \beta \weightxv
	    \right] (w^{k+1})^2 e^{\beta \weightxv} \\
	    I_3 &=  \intxv \left[
	        2 \bbar^k_j \pdv{f^k}{v_j}
	        - \pdv{\bbar^k_j}{v_j} (1 - f^k)
	    \right] w^k w^{k+1} e^{\beta \weightxv}\\
	    I_4 &= - \intxv \left[
	        (\bbar^k_j - \bbar^{k-1}_j) (1 - 2 f^{k-1}) \pdv{f^k}{v_j}
	        + \pdv{(\bbar^k_j - \bbar^{k-1}_j)}{v_j} (1 - f^k) f^{k-1}
	    \right] w^{k+1} e^{\beta \weightxv}.
	\end{align*}
	For brevity, we have dropped the explicit time dependence and write $\beta$ for $\beta(t)$ throughout this estimate.

	Let us establish some bounds on $\abar^{k}_{ij} - \abar^{k-1}_{ij}$.
	Applying the triangle inequality, we obtain
	\begin{equation*}
		\begin{split}
			|\abar^{k}_{ij} - \abar^{k-1}_{ij}|
			&\leq \left|
				\frac{a_{ij} *_v [(f^k - f^{k-1}) (1 - f^k - f^{k-1})]}{1 + \varepsilon \| f^k \|_{L^1_v}}
			\right|\\
			&\quad + \left|
				\varepsilon \frac{
					a_{ij} *_v [(f^{k-1}(1 - f^{k-1}))]
				}{
					(1 + \varepsilon \| f^{k-1} \|_{L^1_v})
				} \right|
				\frac{\big| \| f^{k-1} \|_{L^1_v} - \| f^k \|_{L^1_v} \big|}{(1 + \varepsilon \| f^k \|_{L^1_v})},
		\end{split}
	\end{equation*}
	where we have used the abbreviation $\| g \|_{L^1_v}$ to denote the $(t,x)$ function $\| g(t,x) \|_{L^1(\mathbb{R}^N)}$.

	Since $0 \leq f^k \leq 1$, we have that  $| 1 - f^k - f^{k-1} | \leq 1$.
	Therefore, as the denominator is bounded below by $1$, we have that the first term is bounded by $\|a_{ij}\|_{L^\infty} \|w^k\|_{L^1_v}$.
	For the second term, we apply the reverse triangle inequality $\big| \|f^{k-1}\|_{L^1_v} - \|f^k\|_{L^1_v} \big| \leq \|w^k\|_{L^1_v}$ and we use that $\varepsilon x/(1 + \varepsilon x) \leq 1$ to show that this term is also bounded by $\|a_{ij}\|_{L^\infty} \|w^k\|_{L^1_v}$.

	We square the resulting inequality and we multiply by $e^{-(\alpha-\beta) \weightxv}$.
	Integrating over $(x,v) \in \Rxv$ yields, by the Cauchy-Schwarz inequality, that there exists a constant $C_a > 0$ such that
	\begin{equation}
		\label{eq:cauchy-bound-aij-k-difference}
		\begin{split}
			\int_{xv} | \abar^{k}_{ij} - \abar^{k-1}_{ij} |^2 e^{-(\alpha-\beta) \weightxv}
			&\leq 4 \|a_{ij}\|^2_{L^\infty} \int_{xv} \|w^k\|^2_{L^1_v} e^{-(\alpha-\beta) \weightxv} \\
			&\leq 4 \|a_{ij}\|^2_{L^\infty}
				\int_{x} \left\|w^k e^{\beta/2 \weightxv} \right\|^2_{L^2_v}
					\left\|e^{-\beta/2 \weightxv} \right\|^2_{L^2_v}
					\int_v e^{-(\alpha-\beta) \weightxv} \\
			&\leq C_a \int_{xv} |w^k|^2 e^{\beta \weightxv}.
		\end{split}
	\end{equation}

	By an identical algebraic splitting, analogous pointwise bounds hold for $\bbar^k_j$ and $\partial_{v_j} \bbar^k_j$.
	Since their numerators are linear in $f$, the same reverse triangle inequality and fractional bounds yield
	\begin{align*}
	    |\bbar^{k}_j - \bbar^{k-1}_j|
	    &\leq 2 \| \nabla_v a_{ij} \|_{L^\infty} \| w^k \|_{L^1_v}, \\
	    \left| \pdv{\bbar^{k}_j}{v_j} - \pdv{\bbar^{k-1}_j}{v_j} \right|
	    &\leq 2 \| \nabla^2_v a_{ij} \|_{L^\infty} \| w^k \|_{L^1_v}.
	\end{align*}
	Therefore, applying the exact same squaring and integrating procedure with respect to the weight $e^{-(\alpha-\beta) \weightxv}$, we obtain that
	\begin{equation}
		\label{eq:cauchy-bound-bbar-deriv-difference}
	    \int_{xv} \left( |\bbar^{k}_j - \bbar^{k-1}_j|^2 + \left| \pdv{\bbar^{k}_j}{v_j} - \pdv{\bbar^{k-1}_j}{v_j} \right|^2 \right) e^{-(\alpha-\beta) \weightxv}
	    \leq C_b \int_{xv} |w^k|^2 e^{\beta \weightxv}.
	\end{equation}

	We return to the estimation of the terms $I_i$ in \eqref{eq:cauchy-dt-estimate}.
	Restricting our analysis to $t \in (0,T)$, we observe that for $I_1$, the ellipticity of $\Abar^k$ yields
	\begin{equation*}
	\begin{split}
		I_1 \leq &- \varepsilon \intxv | \nabla_v w^{k+1} |^2 e^{\beta \weightxv}
		- \intxv ( \abar^{k}_{ij} - \abar^{k-1}_{ij} ) \pdv{f^k}{v_i} \pdv{w^{k+1}}{v_j} e^{\beta \weightxv} \\
		&- \intxv \Abar^{k}_{ij} \pdv{w^{k+1}}{v_i} \beta w^{k+1} \frac{v_j}{\weightxv} e^{\beta \weightxv}
		- \intxv ( \abar^{k}_{ij} - \abar^{k-1}_{ij} ) \pdv{f^k}{v_i} \beta w^{k+1} \frac{v_j}{\weightxv} e^{\beta \weightxv}.
	\end{split}
	\end{equation*}
	Applying Young's inequality to the remaining terms, we can absorb the resulting $| \nabla_v w^{k+1} |^2$ contributions into the first strictly negative term, which yields the bound
	\begin{equation*}
	\begin{split}
		I_1
		&\leq \sum_{ij} \intxv
			\left(
				\frac{N}{2\varepsilon}
				+ \frac{1}{2}
			\right)
			| \abar^{k}_{ij} - \abar^{k-1}_{ij} |^2 \left|\pdv{f^k}{v_i}\right|^2
			e^{\beta \weightxv} \\
		&\quad + \sum_{ij} \intxv
			\left(
				\frac{N}{2\varepsilon} |\Abar^k_{ij}|^2
				+ \frac{1}{2}
			\right)
			\beta^2 |w^{k+1}|^2
			\left|\frac{v_j}{\weightxv}\right|^2
		e^{\beta \weightxv}.
	\end{split}
	\end{equation*}
	For the first integral, we use the pointwise bounds \eqref{eq:iteration-pointwise-derivatives-bound} on the derivatives of $f^k$, together with the estimate \eqref{eq:cauchy-bound-aij-k-difference} on the differences $| \abar^{k}_{ij} - \abar^{k-1}_{ij} |$.
	On the second integral, we use the uniform bounds \eqref{eq:maximum-principle-bounds-abar-bbar} on $\abar^k$ to bound the diffusion matrix $|\Abar^k_{ij}| \leq \varepsilon + \varepsilon^{-1} \|a_{ij}\|_{L^\infty}$.
	Finally, noting that $|v_j / \weightxv| \leq 1$ and $\beta(t) \leq \beta_0$, we obtain the bound
	\begin{equation*}
		I_1 \leq
		\left(
			\frac{N}{2\varepsilon}
			+ \frac{1}{2}
		\right)
		C_4 C_a
		\intxv |w^{k}|^2 e^{\beta \weightxv}
		+ \beta^2_0
		N^2
		\left(
			\frac{N}{2\varepsilon}
			(\varepsilon
			+ \varepsilon^{-1} \|a_{ij}\|_{L^\infty})^2
			+\frac{1}{2}
		\right)
		\intxv |w^{k+1}|^2 e^{\beta \weightxv}.
	\end{equation*}

	Next, we estimate $I_2$.
	First we note that, by Young's inequality for products,
	\[
		\frac{\beta}{2} \frac{x \cdot v}{\weightxv}
		- \frac{\lambda}{2} \beta \weightxv
		\leq
		\frac{1}{2}
		\left(
			\frac{1}{2}
			- \lambda
		\right)
		\beta \weightxv.
	\]
	This quantity is always negative from the choice of $\lambda$.
	Furthermore, using the uniform boundedness \eqref{eq:maximum-principle-bounds-abar-bbar} on $\bbar_j$ and its derivatives, as well as the pointwise bound \eqref{eq:iteration-pointwise-derivatives-bound} on the derivatives of $f^k$, together with $\beta(t) \leq \beta_0$, the fact that $0 \leq f^k \leq 1$ and $|v_j / \weightxv| \leq 1$, we have that $I_2$ is bounded by
	\begin{equation*}
		I_2 \leq
		\frac{1}{2} \varepsilon^{-1} N^2
		\left[
			\left( 2\sqrt{C_4} + \beta_0 \right)
			\| \nabla_v a \|_\infty
			+ \| \nabla^2_v a \|_\infty
		\right]
		\int_{xv} (w^{k+1})^2 e^{\beta \weightxv}.
	\end{equation*}
	The same estimates applied to $I_3$, together with the Young's inequality on the term $w^k w^{k+1}$ yield
	\begin{equation*}
		I_3
		\leq \frac{1}{2} \varepsilon^{-1} N^2
			\left[
				2  \| \nabla_v a \|_\infty \sqrt{C_4}
				+ \| \nabla^2_v a \|_\infty
			\right]
		\intxv (|w^k|^2 + |w^{k+1}|^2) e^{\beta \weightxv}.
	\end{equation*}

	To estimate $I_4$, we pass to the absolute value and use that $|1 - 2 f^{k-1}| \leq 1$ and $|1 - f^k| \leq 1$.
	Using Young's inequality to decouple the products involving $w^{k+1}$, we apply the pointwise bound \eqref{eq:iteration-pointwise-derivatives-bound} on the derivatives of $f^k$ and the decay estimate \eqref{ineq-iteration-f-est}.
	Finally, summing over the $N$ components and applying the bound \eqref{eq:cauchy-bound-bbar-deriv-difference}, we obtain
	\begin{equation*}
	\begin{split}
		I_4
		&\leq \frac{C_4}{2} \sum_j \intxv
			|(\bbar^k_j - \bbar^{k-1}_j)| ^2
			e^{-(\alpha - \beta) \weightxv}
			+ N \intxv
			|w^{k+1}|^2 e^{\beta \weightxv} \\
			&\quad + \frac{C_1}{2}
			\sum_j \intxv
			\left| \pdv{(\bbar^k_j - \bbar^{k-1}_j)}{v_j} \right|^2 e^{-(\alpha - \beta) \weightxv} \\
		&\leq \frac{1}{2} N C_b (C_4 + C_1) \int_{xv} |w^k|^2 e^{\beta \weightxv}
		+ N \intxv
		|w^{k+1}|^2 e^{\beta \weightxv}
	\end{split}
	\end{equation*}

	We introduce the time-dependent weight function $\omega_t(x,v) = e^{\beta(t) \weightxv}$ and define the corresponding weighted $L^2$ norm as
	\[
		\| g \|^2_{L^2(\omega_t)}
		= \int_{xv} |g(x,v)|^2 \omega_t(x,v).
	\]
	Applying the estimates for $I_i$ to the equation \eqref{eq:cauchy-dt-estimate} and integrating over $s \in (0,t)$, we have that there exist constants $C_T, C'_T > 0$ such that, for every $t \in (0,T)$,
	\begin{equation}
		\label{eq:cauchy-induction}
		\| w^{k+1} (t) \|^2_{L^2(\omega_t)}
		\leq C_T \int_0^t \| w^k (s) \|^2_{L^2(\omega_s)} \, ds
		+ C'_T \int_0^t \| w^{k+1} (s) \|^2_{L^2(\omega_s)} \, ds.
	\end{equation}

	Let us prove by induction that
	\begin{equation}
		\label{eq:cauchy-w-n+1}
		\| w^{k} (t) \|_{L^2(\omega_t)}^2
		\leq A_T \frac{e^{C'_T (k-1) T} (C_T t)^{k-1}}{(k-1)!}
		\quad\forall t \in (0, T),
	\end{equation}
	for every $k \in \mathbb{N}$, where $A_T = \| w^1 \|^2_{L^\infty((0,T); L^2(\omega_t))}$.
	The $k = 1$ case is trivial by the definition of $A_T$.
	If \eqref{eq:cauchy-w-n+1} holds for $k = n$, then from \eqref{eq:cauchy-induction} we have
	\[
		\| w^{n+1} (t) \|^2_{L^2(\omega_t)}
		\leq A_T e^{C'_T (n-1) T} \frac{(C_T t)^n}{n!}
		+ C'_T \int_0^t \| w^{n+1} (s) \|^2_{L^2(\omega_s)} \, ds,
	\]
	which, by Gr\"{o}wall's inequality, implies that
	\begin{align*}
		\| w^{n+1} (t) \|^2_{L^2(\omega_t)}
		&\leq A_T e^{C'_T (n-1) T} \frac{(C_T t)^n}{n!}
		+ A_T e^{C'_T (n-1) T} \int_0^t \frac{(C_T s)^n}{n!} C'_T e^{C'_T (t-s)} \, ds\\
		& \leq A_T e^{C'_T (n-1) T} \frac{(C_T t)^n}{n!}
		+ A_T e^{C'_T (n-1) T}  \frac{(C_T t)^n}{n!} (e^{C'_T t} - 1)\\
		&\leq A_T e^{C'_T n T} \frac{(C_T t)^n}{n!}.
	\end{align*}
	Recalling $w^{k+1} = f^{k+1} - f^k$ and the definition of $\omega_t$, we apply the supremum over $t \in (0,T)$ to obtain that
	\[
		\sum_{k=1}^{\infty}
		\left\| (f^k - f^{k-1}) e^{\beta(t) \weightxv} \right\|_{L^\infty((0,T); L^2(\Rxv))}
		\leq \sqrt{A_T} \sum_{k=1}^{\infty} \sqrt{\frac{e^{C'_T (k-1) T} (C_T T)^{k-1}}{(k-1)!}}.
	\]
	The series on the right-hand side converges by the ratio test.
	Because $\beta(t)$ is strictly positive on $[0,T]$, the weight $e^{\beta(t) \weightxv}$ is uniformly bounded from below by $1$, which implies the sequence $(f^k)_k$ is Cauchy in $L^\infty((0,T); L^2(\Rxv))$.
	Since $T$ is arbitrary, this convergence holds in $L^\infty_{loc}((0,\infty); L^2(\Rxv))$.

	Using the exponential decay estimate \eqref{ineq-iteration-f-est}, we may extend this convergence to any $p \in [1, \infty)$, yielding
	\begin{equation*}
		f^{k} \to f
		\quad \text{in }
		L^\infty_{loc}((0,\infty); L^p(\Rxv)).
	\end{equation*}
	Passing to a subsequence, we also have that $f^k \to f$ almost everywhere in $(0,\infty) \times \Rxv$.
	Moreover, \eqref{ineq-iteration-f-est} implies that $f^k e^{\alpha \weightxv}$ is uniformly bounded in $L^\infty_{loc}((0,\infty); L^\infty(\Rxv))$.
	Therefore, by interpolation, we find that for every $\eta \in [0,\alpha)$ and any $p \in [1, \infty)$,
	\begin{equation}
		\label{eq:linearized-existence-fk-convergence}
		e^{\eta \weightxv} f^{k} \to e^{\eta \weightxv} f
		\quad \text{in }
		L^\infty_{loc}((0,\infty); L^p(\Rxv)).
	\end{equation}

	Furthermore, the bounds \eqref{eq:uniform-deriv-fk-est} control the second derivatives $\nabla_v^2 f^k$ uniformly in $L^\infty$.
	These bounds are inherited by the limit $f$ by the Banach-Alaoglu theorem and the weak-* lower semicontinuity of the $L^\infty$ norm.
	Applying the standard Gagliardo-Nirenberg inequality in the velocity variable and integrating over $x$, we obtain for any $p \in [1,\infty)$, that
	\[
		    \| \nabla_v f^k(t) - \nabla_v f(t) \|_{L^{2p}(\Rxv)}
		    \leq C \| f^k(t) - f(t) \|_{L^p(\Rxv)}^{1/2}
		    	\| \nabla_v^2 (f^k - f) (t) \|_{L^\infty(\Rxv)}^{1/2}.
	\]
	From \eqref{eq:linearized-existence-fk-convergence}, we immediately deduce that for any $q\in [2,\infty)$,
	\begin{equation}
		\label{eq:linearized-existence-nabla-fk-convergence}
		\nabla_v f^{k} \to \nabla_v f
		\quad \text{in }
		L^\infty_{loc}((0,\infty);L^q(\Rxv)).
	\end{equation}
	Up to passing to a subsequence, we also obtain convergence almost everywhere in $(0,\infty) \times \Rxv$.

	By combining the uniform pointwise bounds \eqref{eq:iteration-pointwise-derivatives-bound} on the first derivatives of $f^k$ with the unweighted convergence \eqref{eq:linearized-existence-nabla-fk-convergence}, we have that for any $\eta \in [0,\alpha/2)$ and any $q \in [2, \infty)$,
	\begin{equation}
		\label{eq:linearized-existence-nabla-fk-convergence-weight}
		e^{\eta \weightxv} \nabla_v f^{k}
		\to e^{\eta \weightxv} \nabla_v f
		\quad \text{in }
		L^\infty_{loc}((0,\infty); L^q(\Rxv)).
	\end{equation}

	Let us now study the convergence of the coefficients.
	Recall the definition
	\[
		\abar^k_{ij}
		= \frac{a_{ij} *_v (f^k(1-f^k))}{1 + \varepsilon \| f^k(t,x) \|_{L^1}}.
	\]
	The convergence $f^k \to f$ almost everywhere, combined with the uniform exponential decay \eqref{ineq-iteration-f-est} allows us to pass to the limit in the denominator for almost every $(t,x) \in (0,\infty) \times \mathbb{R}^N_x$.
	Similarly, since $a_{ij}$ is bounded, the numerator converges almost everywhere to $a_{ij} *_v (f(1-f))$ in $(0,\infty) \times \Rxv$.
	Finally, from \eqref{eq:maximum-principle-bounds-abar-bbar} the sequence $\abar^k_{ij}$ is uniformly bounded in $L^\infty((0,\infty) \times \Rxv)$.
	Applying dominated convergence on compact subsets, we conclude that for all $p \in [1,\infty)$,
	\begin{equation}
		\label{eq:linearized-existence-abar-convergence}
		\abar^{k}_{ij} \to \abar_{ij}
		\quad\text{in }
		L^p_{loc}((0,\infty) \times \Rxv)
		\text{ and a.e.}
	\end{equation}

	Analogous arguments applied to the definitions of $\bbar^k_j$ and $\partial_{v_j} \bbar^k_j$, using their respective uniform $L^\infty$ bounds \eqref{eq:maximum-principle-bounds-abar-bbar}, yield that for all $p \in [1, \infty)$,
	\begin{equation}
		\label{eq:linearized-existence-bbar-convergence}
		\begin{aligned}
			\bbar^{k}_j &\to \bbar_j \\
			\pdv{\bbar^{k}_j}{v_j} &\to \pdv{\bbar_j}{v_j}
		\end{aligned}
		\qquad \text{in } L^p_{loc}((0,\infty) \times \Rxv) \text{ and a.e.}
	\end{equation}

	\proofpart{3}{Passing equation \eqref{iteration-eq-k} to the limit}

		Integrating \eqref{iteration-eq-k} by parts against a test function $\varphi \in \mathcal{D}([0,\infty) \times \Rxv)$ we obtain that, for every $t > 0$,
	\begin{equation}
		\label{iteration-eq-k-distributions}
		\begin{split}
		& \intxv f^{k + 1}(t) \varphi(t)
		- \intxv f_0 \varphi(0)
		- \int_0^t \intxv f^{k + 1} \pdv{\varphi}{t}
		- \int_0^t \intxv f^{k + 1} v_i \pdv{\varphi}{x_i}\\
		&\quad = \int_0^t \intxv ( \abar^{k}_{ij} + \varepsilon \delta_{ij} ) f^{k + 1} \mpdv{\varphi}{v_i}{v_j}
		+ \int_0^t \intxv \left( \pdv{\abar^{k}_{ij}}{v_i} + \bbar_j^{k} (1-2f^{k}) \right) f^{k + 1} \pdv{\varphi}{v_j}\\
		&\qquad + \int_0^t \intxv \Bigg[ \bigg(
							\pdv{\bbar^{k}_j}{v_j} (1-2f^{k})
							- 2 \bbar^{k}_j \pdv{f^{k}}{v_j}
						\bigg) f^{k + 1}
		- \pdv{\bbar^{k}_j}{v_j} f^{k} (1-f^{k + 1}) \Bigg]\varphi.
		\end{split}
	\end{equation}
	We want to pass to the limit in equation \eqref{iteration-eq-k-distributions}.
	From the convergence of $(f^{k + 1})_k$, the left hand side converges for almost every $t > 0$ to
	\[
		\intxv f(t) \varphi(t)
		- \intxv f_0 \,\varphi(0)
		- \int_0^t \intxv f \pdv{\varphi}{t}
		- \int_0^t \intxv f v_i \pdv{\varphi}{x_i}.
	\]

	Given the convergences \eqref{eq:linearized-existence-fk-convergence} and \eqref{eq:linearized-existence-nabla-fk-convergence}, the strong $L^p_{loc}((0,\infty) \times \Rxv)$ convergences \eqref{eq:linearized-existence-abar-convergence} and \eqref{eq:linearized-existence-bbar-convergence} of the coefficients ensure that all nonlinear products in the right-hand side of \eqref{iteration-eq-k-distributions} converge in $L^1_{loc}((0,\infty) \times \Rxv)$.
	Taking the limit $k \to \infty$, using the identity $( \partial_{v_j} \bbar_j (1-2f) - 2 \bbar_j \partial_{v_j} f ) f - \partial_{v_j} \bbar_j f (1-f) = - \partial_{v_j} (\bbar_j f^2)$ and integrating by parts on this derivative yields
	\begin{equation}
		\label{eq-approx-almost-distrib}
		\begin{split}
			&\intxv f(t) \varphi(t)
			- \intxv f_0 \varphi(0)
			- \int_0^t \intxv f \pdv{\varphi}{t}
			- \int_0^t  \intxv f v_i \pdv{\varphi}{x_i} \\
			&\quad = \int_0^t \intxv ( \abar_{ij} + \varepsilon \delta_{ij} ) f \mpdv{\varphi}{v_i}{v_j}
			+ \int_0^t \intxv \left( \pdv{\abar_{ij}}{v_i} + \bbar_j (1-f) \right) f \pdv{\varphi}{v_j}
		\end{split}
	\end{equation}
	for almost every $t > 0$.
	The integrals in $(0,t)$ are evidently continuous functions of $t$.
	This implies that for every $\varphi \in \mathcal{D}([0,\infty) \times \Rxv)$, the mapping
	\[\begin{cases}
		[0,\infty) &\to \mathbb{R}\\
		t & \mapsto \displaystyle\intxv f(t) \varphi(t)
	\end{cases}\]
	is a continuous function.
	The equality \eqref{eq-approx-almost-distrib} is then valid for every $t \geq 0$.

	We prove that $f \in C([0,\infty); \mathcal{D}'(\Rxv))$.
	Let $[0,T]$ be any compact time interval and let $\psi \in \mathcal{D}(\Rxv)$.
	Choose a test function $\varphi \in \mathcal{D}([0,\infty) \times \Rxv)$ such that $\varphi(t,x,v) = \psi(x,v)$ for every $(t,x,v) \in [0,T] \times \Rxv$.
	The continuity established above ensures that the mapping $t \mapsto \langle f(t), \psi \rangle$ is continuous on $[0,T]$.
	Since $T>0$ is arbitrary, the desired continuity in the sense of distributions follows.
	Furthermore, evaluating \eqref{eq-approx-almost-distrib} at $t = 0$ immediately yields $f(0) = f_0$ in the sense of distributions.

	Next, we prove that the function $f$ is continuous on $[0,\infty) \times \Rxv$.
	The H\"{o}lder bound \eqref{eq:uniform-holder-fk} on $f^k$, being uniform in $k$, implies that the limit function $f$ belongs to $L^\infty_{loc}([0,\infty); C^\gamma(\Rxv))$.
	Fix a compact time interval $[0,T]$ and a compact set $K \subset \Rxv$.
	There exists a constant $M > 0$ such that $\| f(t) \|_{C^\gamma(K)} \leq M$ for almost every $t \in [0,T]$.
	Since $f \in C([0,\infty); \mathcal{D}'(\Rxv))$, we can identify $f$ with its continuous representative.
	By approximating any time $t \in [0,T]$ with a sequence of times where the bound holds, an application of the Arzelà-Ascoli theorem ensures this $C^\gamma$ bound holds for every $t \in [0,T]$.

	We then establish strong continuity in time.
	For a sequence $t_n \to t$ in $[0,T]$, the family $\{f(t_n,\cdot)\}_{n \in \mathbb{N}}$ is uniformly bounded and equicontinuous on $K$ due to the uniform $C^\gamma$ bound.
	By the Arzelà-Ascoli theorem, any subsequence possesses a further subsequence converging uniformly on $K$ to some continuous function.
	Since $f(t_n, \cdot) \to f(t, \cdot)$ in $\mathcal{D}'(K)$, it follows that this limit function is actually $f(t, \cdot)$.

	By the Urysohn subsequence principle, because every subsequence contains a further subsequence converging uniformly to the same limit $f(t, \cdot)$, we have that $f(t_n, \cdot) \to f(t, \cdot)$ in $L^\infty(K)$.
	The joint continuity on $[0,\infty) \times \Rxv$ then follows. In particular, $f$ attains the initial data $f_0$ strongly and therefore we have a solution in the sense of Definition \ref{def-weak-solution}.

	Finally, the local H\"{o}lder bound on $\nabla^2_v f^k$ established in Lemma \ref{derivative-infty-est} holds uniformly in $k$.
	By the Arzelà-Ascoli theorem, passing to a further subsequence if necessary, $(\nabla^2_v f^k)_k$ converges uniformly on compact subsets of $(0,\infty) \times \Rxv$ to a continuous limit.

	Since $f^k \to f$ in the sense of distributions, this limit must be $\nabla^2_v f$.
	Therefore, the limit $\nabla^2_v f$ inherits the local H\"{o}lder regularity, which implies in particular that the limit function $f$ is $C^2$ with respect to the velocity variable $v$ in $(0,\infty) \times \Rxv$.
	Moreover, the estimates \eqref{ineq-iteration-f-est} and \eqref{eq:uniform-deriv-fk-est}, being uniform in $k$, are still valid in the limit.
\end{proof}

We will now show that the conservation laws are approximately obeyed by the solutions of the approximate equation, a result that will be crucial for deducing that our solution of the LFD equation obey the standard conservation laws.
\begin{proposition}
	\label{prop-approx-a-priori-estimates}
	The approximated solution $f \in C((0,\infty); \mathcal{D}'(\Rxv)) \cap L^2((0,\infty) \times \Rx; H^1(\Rv))$ to \eqref{regularized} constructed in Proposition \ref{prop-existence-approx} satisfies, for almost every $t > 0$,

	1) Conservation of mass and linear momentum:
	\[
		\iint f(t,x,v) \, dxdv
		= \iint f_0(x,v) \,dxdv,
	\]\[
		\iint f(t,x,v) v_i \, dxdv
		= \iint f_0(x,v) v_i \,dxdv
		\quad \forall i \in \{ 1, \cdots, N \}.
	\]

	2) Kinetic energy:
	\[
		\iint |v|^2 f(t,x,v) \,dxdv
		= \iint |v|^2 f_0(x,v) \,dxdv
		+ 2\varepsilon t \| f_0 \|_{L^1}.
	\]

	3) Moment of inertia:
	\[
		\iint f(t,x,v) |x - tv|^2 \, dxdv
		= \iint f_0(x,v) |x|^2 \, dxdv
		+ \frac{2\varepsilon t^3}{3} \| f_0 \|_{L^1}.
	\]

	4) Entropy inequality:
	\[
		\iint s(t,x,v) \, dxdv
		+ \int_0^t \iint d(\tau,x,v) \, dxdvd\tau
		\leq \iint s(0,x,v) \, dxdv,
	\]
	where $s(t,x,v) = f \log f + (1-f) \log(1-f)$, and the dissipation term $d(t,x,v)$ is defined as in \eqref{dissipation-def}, substituting $a$ with $\tilde{a} = a / (1 + \varepsilon \| f(t,x) \|_{L^1})$.
\end{proposition}
\begin{proof}

	Let $(\varphi^n)_n$ be a sequence of smooth cutoff functions defined by $\varphi^n(x,v) = \varphi((x,v)/n)$, where $\varphi \in C^\infty_c(\Rxv)$ satisfies $0\leq \varphi \leq 1$ and $\varphi(0) = 1$.
	It follows that $\varphi^n \to 1$ and $\nabla_{x,v} \varphi^n \to 0$ pointwise, while both these sequences of functions remain uniformly bounded in $L^\infty(\Rxv)$.
	Let $\psi \in \{1, v_i, |v|^2, |x - tv|^2 \}$ and consider $(f^k)_k$ be the sequence of solutions to \eqref{iteration-eq-k} constructed in the proof of Proposition \ref{prop-existence-approx}.

	Fix $t > 0$.
	Multiplying equation \eqref{iteration-eq-k} by $\varphi^k \psi$ we have, after integrating by parts over $(0,t) \times \Rxv$,
	\begin{equation}
		\label{approx-conservation-laws-nk}
		\begin{split}
			&\intxv f^k(t) \varphi^k \psi(t)
			- \intxv f_0 \, \varphi^k \psi_0
			- \inttxv f^k
				\left(
					\pdv{(\varphi^k \psi)}{t} + v_i \pdv{(\varphi^k \psi)}{x_i}
				\right) \\
			&= - \inttxv \Abar_{ij}^{k-1} \pdv{f^k}{v_i}
				\pdv{(\varphi^k \psi)}{v_j}
			- \inttxv \bbar_j^{k-1} (1-2 f^{k-1})
				\pdv{f^k}{v_j} \varphi^k \psi \\
			&\quad - \inttxv \pdv{\bbar_j^{k-1}}{v_j} f^{k-1} (1-f^k)
				\varphi^k \psi,
		\end{split}
	\end{equation}
	where we note $\inttxv$ the integral $\int_0^t \intxv$, depending on $t$, and $\Abar_{ij}^{k-1} = \abar_{ij}^{k-1} + \varepsilon \delta_{ij}$.

	Let us pass each of the integrals to the limit $k \to \infty$.
	Thanks to the pointwise convergence of $f^k$, the uniform exponential bounds \eqref{ineq-iteration-f-est} on $f^k$ and the polynomial growth of $\psi$, we can apply the dominated convergence theorem to pass the first two integrals to the limit.

	For the third term on the left-hand side of \eqref{approx-conservation-laws-nk}, we expand the integrand to obtain
	\[
		f^k
		\left(
			\pdv{(\varphi^k \psi)}{t}
			+ v_i \pdv{(\varphi^k \psi)}{x_i}
		\right)
		= f^k \left(
			\pdv{\varphi^k}{t}
			+ v_i \pdv{\varphi^k}{x_i}
		\right) \psi
		+ f^k \varphi^k \left(
			\pdv{\psi}{t}
			+ v_i \pdv{\psi}{x_i}
		\right).
	\]
	Because the derivatives of $\varphi^k$ vanish as $k \to \infty$, the first term on the right-hand side of this expansion tends to zero almost everywhere.
	Furthermore, our choices for $\psi$ satisfy $\partial_t \psi + v \cdot \nabla_x \psi = 0$, meaning the second term vanishes entirely and therefore the whole expansion converges almost everywhere.

	Since $\varphi^k$ and its derivatives are uniformly bounded and the uniform exponential decay of $f^k$ controls the polynomial growth of $\psi$, the entire expression is dominated by an integrable function.
	Therefore, by the dominated convergence theorem, the third integral on the left-hand side of \eqref{approx-conservation-laws-nk} vanishes in the limit as $k \to \infty$.

	We now turn to the integrals on the right-hand side of \eqref{approx-conservation-laws-nk}.
	As before, we expand the integrand of the first term to obtain
	\[
		\Abar_{ij}^{k-1} \pdv{f^k}{v_i}
			\pdv{(\varphi^k \psi)}{v_j}
		= \Abar_{ij}^{k-1} \pdv{f^k}{v_i}
			\varphi^k \pdv{\psi}{v_j}
		+ \Abar_{ij}^{k-1} \pdv{f^k}{v_i}
			\pdv{\varphi^k}{v_j} \psi.
	\]
	By the pointwise convergence of $\Abar^{k-1}_{ij}$, $\nabla_v f^k$ and $\varphi^k$ we have that the first term on the right-hand side of this expansion converges almost everywhere.
	Furthermore, since $\nabla_v \varphi^k \to 0$ as $k \to \infty$, the second term vanishes.
	Recall that the sequence $\Abar^{k-1}_{ij}$ is uniformly bounded in $L^\infty_{loc}((0,\infty); L^\infty(\Rxv))$, and that $\varphi^k$ and its derivatives $\nabla_v \varphi^k$ are also bounded in $L^\infty$.
	Since the uniform exponential decay \eqref{eq:iteration-pointwise-derivatives-bound} of $\nabla_v f^k$ controls the polynomial growth of $\psi$, we can apply dominated convergence to pass the first term on the right-hand side of \eqref{approx-conservation-laws-nk} to the limit.

	For the remaining integrals, the integrands converge almost everywhere due to the pointwise convergence of $\varphi^k$, $\bbar^{k-1}_j$, $\partial_{v_j} \bbar^{k-1}_j$, $f^k$ and $\nabla_v f^k$.
	We note that the coefficient sequences $\bbar^{k-1}_j$ and $\partial_{v_j} \bbar^{k-1}_j$ are uniformly bounded in $L^\infty$, as are the sequences $(1-2 f^{k-1})$ and $(1-f^k)$.
	Combining these bounds with the uniform $L^\infty$ bound on $\varphi^k$, the uniform exponential decay of $\nabla_v f^k$ and $f^k$ again control the polynomial growth of $\psi$.
	Therefore, we may pass these terms to the limit $k \to \infty$ by dominated convergence, yielding
	\begin{equation}
	\label{approx-conservation-laws}
	\begin{split}
		\intxv f(t) \psi(t)
		- \intxv f_0 \, \psi_0
		&= - \inttxv \Abar_{ij} \pdv{f}{v_i}
			\pdv{\psi}{v_j}
		- \inttxv \bbar_j (1-2 f)
			\pdv{f}{v_j} \psi \\
		&\quad - \inttxv \pdv{\bbar_j}{v_j} f (1-f) \psi.
	\end{split}
	\end{equation}

	Expanding the convolutions and symmetrizing the collision integral yields, denoting $\tilde{a}(t,x,v) = a(v) / (1 + \varepsilon \| f(t,x) \|_{L^1})$,
	\begin{equation*}
		\begin{split}
			\intxv f(t) \psi
			- \intxv f_0 \psi_0
			&= - \int_0^t \intxvv \tilde{a}_{ij}(v-v_*)
			\left(
				f_* (1-f_*) \pdv{f}{v_j}
				- f (1-f) \pdv{f_*}{v_{*,j}}
			\right)
			\left(
				\pdv{\psi}{v_i} - \pdv{\psi_*}{v_{*,i}}
			\right) \\
			&\quad + \varepsilon \int_0^t \intxv f \Delta_v\psi.
		\end{split}
	\end{equation*}
	Setting $\psi = 1$ and $\psi = v_k$, both integrals on the right-hand side vanish, which immediately yields the conservation of mass and linear momentum, respectively.
	Choosing $\psi = |v|^2$, the collision integral vanishes due to the third property of \eqref{property:approx-a}.
	Integrating the diffusion term and applying the previously established conservation of mass, we obtain
	\[
		\intxv f(t) |v|^2
		- \intxv f_0 |v|^2
		= 2\varepsilon t \intxv f_0.
	\]
	Similarly, taking $\psi = |x - tv|^2$ causes the collision integral to vanish once again because of \eqref{property:approx-a}.
	Applying conservation of mass to the remaining diffusion term, we find
	\[
		\intxv f(t) |x - tv|^2
		- \intxv f_0 |x|^2
		= \frac{2\varepsilon t^3}{3} \intxv f_0.
	\]

	Now we pass to the proof of the entropy inequality.
	To avoid the singularity of the logarithm at zero, let $\lambda \in (0,1)$ be a parameter which we will let tend to zero later.
	We multiply \eqref{iteration-eq-k} by $\log((f^k + \lambda)/(1-f^k)) \varphi^k$ and we integrate by parts over $(t,x,v) \in (0,T) \times \Rxv$.
	Observing that $x \mapsto \log((x+\lambda)/(1-x))$ is the derivative of the regularized entropy density $s_\lambda(x) = (x+\lambda) \log(x+\lambda) + (1-x) \log(1-x)  - \lambda \log \lambda$, we obtain
	\begin{equation}
		\label{eq:approx-conservation-laws-entropy-k-lambda}
		\begin{split}
			&\intxv s_\lambda(f^k(t)) \varphi^k
			- \intxv s_\lambda(f_0) \varphi^k
			- \inttxv v_i s_\lambda(f^k)
				\pdv{\varphi^k}{x_i} \\
			&= - \inttxv \Abar^{k-1}_{ij} \frac{1+\lambda}{(f^k+\lambda)(1-f^k)} \pdv{f^k}{v_i} \pdv{(\varphi^k f^k)}{v_j} \\
			&\quad - \intxv \bbar^{k-1}_j (1-2f^{k-1}) \pdv{f^k}{v_j} \log\left(\frac{f^k + \lambda}{1-f^k}\right) \varphi^k \\
			&\quad - \intxv \pdv{\bbar^{k-1}_j}{v_j} f^{k-1} (1-f^k) \log\left(\frac{f^k + \lambda}{1-f^k}\right) \varphi^k.
		\end{split}
	\end{equation}

	We establish a bound for $s_\lambda$.
	For a fixed $x \geq 0$ the function $\lambda \mapsto s_\lambda(x)$ is increasing on $(0,\infty)$.
	Therefore, for any $\lambda \in [0,1]$, we have $s_0(x) \leq s_\lambda(x) \leq s_1(x)$, which directly implies that $|s_\lambda(x)| \leq |s_0(x)| + |s_1(x)|$.
	Expanding these terms and applying the triangle inequality we obtain, for $x \in [0,1]$,
	\[
		\begin{split}
		|s_\lambda(x)|
		&\leq
		\left|
			x \log\left(
				\frac{x}{1-x}
			\right)
		\right|
		+ \left|
			(1+x) \log \left(
				\frac{1+x}{1-x}
			\right)
		\right|
		+ 3 | \log(1-x) | \\
		&= (1-x)
			\left|
				\frac{x}{1-x}
				\log\left(
					\frac{x}{1-x}
				\right)
			\right|
			+ \left|
				(1+x) \log \left(
					1+
					\frac{2x}{1-x}
				\right)
			\right|
			+ 3 \log\left(
					1+
					\frac{x}{1-x}
				\right).
		\end{split}
	\]
	To simplify this expression, we use that $y|\log(y)| \leq \sqrt{y}$ for all $y \in [0,1]$ and that $|\log(1+y)| \leq y$ for all $y \in [0,2]$.
	Restricting our domain to $x \in [0,1/2]$, we obtain that there exists some constant $C_s > 0$, independent of $\lambda$, such that
	\[
		|s_\lambda(x)|
		\leq \sqrt{x(1-x)}
			+ \frac{(5+2x)x}{1-x}
		\leq C_s \sqrt{x}.
	\]
	On the remaining interval $x \in [1/2, 1]$, the continuous function $|s_0(x)| + |s_1(x)|$ is bounded by some maximum $M$.
	Since $\sqrt{x} \geq 1/\sqrt{2}$, by enlarging the value of $C_s$ if necessary, we have that $|s_\lambda(x)| \leq C_s \sqrt{x}$ for every $x \in [0,1]$.

	Finally, we apply this result to $f^k$.
	Fix $T>0$.
	Substituting the exponential decay \eqref{ineq-iteration-f-est} yields a constant $C_T > 0$, independent of $k$ and $\lambda$, such that
	\begin{equation}
		\label{approx-conservation-laws-entropy-bound}
		|s_\lambda(f^k)| \leq C_T e^{-\alpha \weightxv / 2}
		\quad \text{on }
		(0,T) \times \Rxv.
	\end{equation}

	Let us pass the integrals on the left-hand side of \eqref{eq:approx-conservation-laws-entropy-k-lambda} to the limit as $k \to \infty$.
	From the pointwise convergence of $f^k$ and $\varphi^k$, the integrands of the first two terms converge almost everywhere.
	Furthermore, since $\nabla_x \varphi^k \to 0$ pointwise as $k\to\infty$, the integrand of the third term vanishes almost everywhere.

	Applying the bounds \eqref{approx-conservation-laws-entropy-bound} established above, along with an analogous estimate for $s_\lambda(f_0)$, and recalling that $\varphi^k$ and its derivatives are uniformly bounded in $L^\infty$, we find that the integrands on the left-hand side of \eqref{eq:approx-conservation-laws-entropy-k-lambda} are dominated by integrable functions.
	This allows us to pass to the limit by the dominated convergence theorem.

	For the right-hand side, the pointwise convergence \eqref{eq:linearized-existence-bbar-convergence} of $\Abar^{k-1}_{ij}$, combined with that of $f^k$, $\varphi^k$ and their gradients $\nabla_v f^k$ and $\nabla_v \varphi^k$, implies the expansion
	\begin{equation*}
	\begin{split}
		\Abar^{k-1}_{ij} \frac{1+\lambda}{(f^k+\lambda)(1-f^k)}
			\pdv{f^k}{v_i}
			\pdv{(\varphi^k f^k)}{v_j}
		&= \Abar^{k-1}_{ij} \frac{1+\lambda}{(f^k+\lambda)(1-f^k)}
			\pdv{f^k}{v_i} \varphi^k \pdv{f^k}{v_j} \\
		&\quad + \Abar^{k-1}_{ij} \frac{1+\lambda}{(f^k+\lambda)(1-f^k)} \pdv{f^k}{v_i} \pdv{\varphi^k}{v_j} f^k
	\end{split}
	\end{equation*}
	converges almost everywhere.
	Furthermore, for any fixed $\lambda \in (0,1)$, the uniform $L^\infty$ bounds on $\Abar^{k-1}_{ij}$ and $1/((f^k+\lambda)(1-f^k))$, together with the exponential decay \eqref{eq:iteration-pointwise-derivatives-bound} of $\nabla_v f^k$, allow us to apply the dominated convergence theorem.
	This guarantees the convergence of the first integral on the right-hand side of \eqref{eq:approx-conservation-laws-entropy-k-lambda}.

	For the remaining terms, recall that the coefficients $\bbar^{k-1}_j$ and $\partial_{v_j} \bbar^{k-1}_j$, alongside $(1-2f^{k-1})$ and $(1-f^k)$, are uniformly bounded in $L^\infty_{loc}((0,\infty); L^\infty(\Rxv))$.
	To control the log factor, we use the uniform bounds on $f^k$ given by \eqref{ineq-iteration-f-est}.
	Indeed, using that $e^{-\alpha \weightxv} \leq 1$, we obtain that $1 - f^k \geq 1/(1+C_1)$.
	Moreover, since $\lambda \leq f^k + \lambda \leq 2$, it follows from the triangle inequality that
	\begin{equation*}
		\begin{split}
			\left|
				\log\left(\frac{f^k+\lambda}{1-f^k}\right)
			\right|
			&\leq | \log(f^k+\lambda) |
			+ \left|
				\log \left(
					\frac{1}{1 - f^k}
				\right)
			\right| \\
			&\leq \max(|\log(\lambda)|, \log(2))
			+ \log\left(1+C_1\right).
		\end{split}
	\end{equation*}
	With this uniform bound established, the exponential decay of $f^k$ and $\nabla_v f^k$ guarantees that we can pass these remaining integrals to the limit $k\to\infty$ by dominated convergence.
	Therefore, for every $\lambda \in (0,1)$, we obtain
	\begin{equation*}
		\begin{split}
			\intxv s_\lambda(f(t)) - \intxv s_\lambda(f_0)
			&\leq -\inttxv \Abar_{ij} \frac{1+\lambda}{(f+\lambda)(1-f)} \pdv{f}{v_i} \pdv{f}{v_j} \\
			&\quad - \intxv \bbar_j (1-2f) \pdv{f}{v_j} \log\left(\frac{f + \lambda}{1-f}\right) \\
			&\quad - \intxv \pdv{\bbar_j}{v_j} f (1-f) \log\left(\frac{f + \lambda}{1-f}\right),
		\end{split}
	\end{equation*}
	where $\Abar_{ij} = \abar_{ij} + \varepsilon \delta_{ij}$.
	Combining the last two integrands using the product rule, we integrate the resulting term by parts.
	Expanding the definition of $\Abar$ then yields
	\begin{equation}
		\label{eq:approx-conservation-laws-entropy-lambda}
		\begin{split}
			\intxv s_\lambda(f(t)) - \intxv s_\lambda(f_0)
			&\leq
			-\inttxv \abar_{ij} \frac{1+\lambda}{(f+\lambda)(1-f)} \pdv{f}{v_i} \pdv{f}{v_j}
			- \inttxv
				(1+\lambda) \,
				\bbar_j \frac{f}{f+\lambda} \pdv{f}{v_j}.
		\end{split}
	\end{equation}

	Let us pass the above inequality to the limit $\lambda \to 0$.
	From the bound \eqref{approx-conservation-laws-entropy-bound}, which holds analogously for the initial data $f_0$, it follows that the two integrands on the left-hand side are dominated by integrable functions independent of $\lambda$.
	Combined with the pointwise convergence $s_\lambda(x) \to s_0(x)$ for all $x \in [0,1]$, the dominated convergence theorem guarantees the convergence of these two integrals.

	For the last integral on the right-hand side, we note that $x/(x+\lambda) \leq 1$ for every $x, \lambda \geq 0$.
	Since $\lambda \in (0,1)$, we can bound the integrand uniformly in $\lambda$,
	\[
		\left|
			(1+\lambda) \,
			\bbar_j \frac{f}{f+\lambda} \pdv{f}{v_j}
		\right|
		\leq 2 \| \bbar_j \|_{L^\infty((0,T) \times \Rxv)}
		\left|
			\pdv{f}{v_j}
		\right|.
	\]
	By the exponential decay estimate \eqref{eq:iteration-pointwise-derivatives-bound}, this implies this integrand is also dominated uniformly in $\lambda$.
	To establish the pointwise limit, observe that as $\lambda \to 0$, the factor $(1+\lambda)f/(f+\lambda)$ converges pointwise to $1$ wherever $f > 0$.
	On the set where $f = 0$ this factor converges to zero, but a standard property of Sobolev spaces ensures that the weak derivative $\partial_{v_j} f$ vanishes almost everywhere in this set.
	Therefore, the entire integrand converges almost everywhere to $\bbar_j f$, which allows us to pass this final term to the limit by dominated convergence.

	Finally, the matrix $\abar$ is positive semidefinite due to the first property in \eqref{property:approx-a} and the fact that $0 \leq f \leq 1$.
	Therefore, the integrand of the first term on the right-hand side of \eqref{eq:approx-conservation-laws-entropy-lambda} is nonnegative for every $\lambda \in (0,1)$.
	Therefore we have by Fatou's lemma that
	\[
		- \liminf_{\lambda \to 0}
		\inttxv \abar_{ij} \frac{1+\lambda}{(f+\lambda)(1-f)} \pdv{f}{v_i} \pdv{f}{v_j}
		\leq -\inttxv \abar_{ij} \frac{1}{f(1-f)} \pdv{f}{v_i} \pdv{f}{v_j}.
	\]
	Taking the $\liminf_{\lambda \to 0}$ of the inequality \eqref{eq:approx-conservation-laws-entropy-lambda}, we substitute this bound alongside the convergences established previously.
	Expanding the convolutions for $\abar_{ij}$ and $\bbar_j$, we obtain
	\begin{equation*}
		\intxv s_0(f(t)) - \intxv s_0(f_0)
		\leq
		- \int_{txvv_*} \tilde{a}_{ij}(v-v_*)
		\left(
			\frac{f_* (1-f_*)}{f(1-f)} \pdv{f}{v_i} \pdv{f}{v_j}
			- \pdv{f_*}{v_{*,i}} \pdv{f}{v_j}
		\right).
	\end{equation*}
	By symmetrizing the variables $v \leftrightarrow v_*$ and using the property $a_{ij}(v-v_*) = - a_{ij}(v-v_*)$ from \eqref{property:approx-a}, this yields the desired entropy inequality.
\end{proof}

\section{Proof of Theorem \ref{existence-lfd-indicator}}
\label{sec:existence-theorem-indicator}

In this section, we introduce the construction that, combining the results in Sections \ref{sec:compactness-lfd} and \ref{sec:cauchy-approximated-equation} allows us to prove Theorem \ref{existence-lfd-indicator}.
We start by introducing a notion of convergence by decomposition, which will be used throughout the whole section.

We say that a sequence $(\Psi^n)_n$ converges in $L^p(\mathbb{R}^N) + L^{\infty,*}(\mathbb{R}^N)$, for some $1 \leq p \leq \infty$, if we can decompose it as $\alpha^n + \beta^n$, where $\alpha^n$ converges strongly $L^p(\mathbb{R}^N)$ and $\beta^n$ converges weakly-* in $L^{\infty}(\mathbb{R}^N)$, that is,
\[
	\begin{tikzpicture}
		\matrix[row sep=1cm, column sep=10pt]
		{
			\node{$\Psi^n(z)$}; & \node{$=$}; & \node(alpha1){$\alpha^n(z)$};
					& \node{$+$}; & \node(beta1){$\beta^n(z)$};    & &\\
			\node{};         &             & \node(alpha2){$\alpha(z)$};
  &\node{$+$};  & \node(beta2){$\beta(z)$};      &\node{$=$}; &\node{$\Psi(z).$};\\
		};
		\draw [->] (alpha1) -- (alpha2) node [pos=0.45,right,font=\footnotesize] {$L^p(\mathbb{R}^N)$};
		\draw [-left to] (beta1)  -- (beta2)  node [pos=0.5,right,font=\footnotesize] {$L^\infty(\mathbb{R}^N)$}
												node[above right, xshift=0ex,yshift=1ex]  {\scriptsize $*$};
	\end{tikzpicture}
\]

We start by constructing a sequence of regularized collision kernels that satisfy the assumptions for Proposition \ref{prop-existence-approx} and approximate the kernels associated with very soft and Coulomb potentials.
\begin{lemma}
	\label{lem:approximated-collision-kernel}
	Let $a$ be the collision kernel defined in \eqref{collision-kernel-formula}.
	There exists a sequence $\{a^n\} \subset C^\infty(\mathbb{R}^N) \cap W^{3,\infty}(\mathbb{R}^N)$ satisfying the properties \eqref{property:approx-a}, such that $a^n \to a$ and $\Div_z a^n \to \Div_z a$ almost everywhere and in $L^1(\mathbb{R}^N) + L^{\infty,*}(\mathbb{R}^N)$.
	If $-N < \gamma < -2$, then $\Div_z (\Div_z a^n) \to \Div_z (\Div_z a)$ almost everywhere and in $L^1(\mathbb{R}^N) + L^{\infty,*}(\mathbb{R}^N)$.
	If $\gamma = -N$, then $\Div_z (\Div_z a^n)$ is uniformly bounded in $L^1(\mathbb{R}^N)$.

	Moreover, for every $R > 0$ there exists a constant $C_R > 0$, independent of $n$, such that for all $|z| \leq R$, in the sense of positive semi-definite matrices,
	\begin{equation}
		\label{eq:approximate-aij-ellpiticity}
		a^n(z) \geq C_R \left( I - \frac{z \otimes z}{|z|^2} \right).
	\end{equation}
	Finally, the matrix square roots and their divergences satisfy $\sqrt{a^n} \to \sqrt{a}$ and $\Div_z \sqrt{a^n} \to \Div_z \sqrt{a}$ almost everywhere and in $L^1_{loc}(\mathbb{R}^N)$.
\end{lemma}
\begin{proof}
	Let $\varepsilon_n \to 0$ be a sequence of real numbers such that $0 \leq \varepsilon_n \leq 1$.
	For the sake of simplicity in dealing with exponents, let $\alpha = -(\gamma+2)$.
	The Coulomb and very soft potentials correspond therefore to the range $0 < \alpha \leq N-2$.

	We define the approximation of the collision kernel as
	\begin{equation}
		\label{eq:approximate-aij}
		a^n(z) = \frac{1}{(|z|^2 + \varepsilon_n^2)^{\alpha/2}}
		\left(
			I - \frac{z \otimes z}{|z|^2 + \varepsilon_n^{2}}
		\right).
	\end{equation}
	For every $n$, this function is in $C^\infty(\mathbb{R}^N)$ as it is a quotient of infinitely differentiable functions with non-zero denominator.
	Hence, the function and all of its derivatives are bounded on every compact set of $\mathbb{R}^N$.
	To show that $a^n \in W^{k,\infty}(\mathbb{R}^N)$, it therefore suffices to verify that its derivatives up to order $k$ remain bounded as $|z| \to \infty$.

	Consider the expanded function $a(z, \varepsilon)$, defined by the formula \eqref{eq:approximate-aij} but treating $\varepsilon$ as an independent variable, such that $a^n(z) = a(z, \varepsilon_n)$.
	The function $a(z, \varepsilon)$ is homogeneous of degree $-\alpha$.
	Therefore, for every $k \in \mathbb{N}$, its $k$-th order derivatives are homogeneous functions of degree $-(\alpha + k)$, meaning they are bounded as $|z| \to \infty$.
	Hence, for each $n$, this approximation $a^n(z)$ is in $W^{k, \infty}(\mathbb{R}^N)$ for every $k \in \mathbb{N}$.

	Note that \eqref{eq:approximate-aij} is indeed a symmetric matrix satisfying the first three properties of \eqref{property:approx-a} (we will prove the final property later).
	To prove the convergence of $a^n$, we first establish a uniform domination for this sequence.
	Consider the entry $a^n_{ij}$ given by equation \eqref{eq:approximate-aij}.
	By the triangle inequality, the $i, j$ entry of the term in parentheses is bounded by $2$.
	Thus, for all $z \in \mathbb{R}^N \setminus \{0\}$, we have
	\begin{equation}
		\label{eq:approximate-aij-bound}
		|a^n_{ij} (z)|
		\leq \frac{2}{(|z|^2 + \varepsilon_n^2)^{\alpha / 2}}
		\leq \frac{2}{|z|^\alpha}.
	\end{equation}

	We split the sequence as
	$
		a^n_{ij}(z)
		= a^n_{ij}(z) \mathbbm{1}_{|z| \leq 1}
		+ a^n_{ij}(z) \mathbbm{1}_{|z| > 1}
	$.
	Note that $a^n_{ij}(z)$ converges pointwise almost everywhere to the collision kernel $a_{ij}(z)$ given in \eqref{collision-kernel-formula}.

	For the $|z| \leq 1$ component, the bound \eqref{eq:approximate-aij-bound} yields $| a^n_{ij}(z) \mathbbm{1}_{|z| \leq 1}| \leq 2 |z|^{-\alpha} \mathbbm{1}_{|z| \leq 1}$.
	Since $\alpha \leq N-2 < N$, this dominating function is in $L^1(\mathbb{R}^N)$.
	By the dominated convergence theorem, $a^n_{ij} \mathbbm{1}_{|z| \leq 1}$ converges strongly to $a_{ij} \mathbbm{1}_{|z| \leq 1}$ in $L^1(\mathbb{R}^N)$.

	For the $|z| > 1$ component, the bound \eqref{eq:approximate-aij-bound} yields $| a^n_{ij}(z) \mathbbm{1}_{|z| > 1}| \leq 2 |z|^{-\alpha} \mathbbm{1}_{|z| > 1} \leq 2$.
	Therefore, the sequence $a^n_{ij}(z) \mathbbm{1}_{|z| > 1}$ is uniformly bounded in $L^\infty(\mathbb{R}^N)$.
	Since it also converges pointwise almost everywhere, the dominated convergence theorem ensures it converges weakly-* to $a_{ij} \mathbbm{1}_{|z| > 1}$ in $L^\infty(\mathbb{R}^N)$.
	Combining these convergences, we conclude that $a^n_{ij} \to a_{ij}$ in $L^1(\mathbb{R}^N) + L^{\infty, *}(\mathbb{R}^N)$.

	Taking the derivative of $a_{ij}^n(z)$ with respect to $z_j$ and summing over $j \in \{1, 2, \dots, N\}$, we obtain
	\[
		( \Div_z a^n )_i
		= \sum_{j=1}^N \pdv{a^n_{ij}}{z_j}
		= - \left[
			\frac{N-1}{(|z|^2 + \varepsilon_n^2)^{\alpha/2 + 1}}
			+ \frac{(\alpha + 2) \varepsilon_n^2}{(|z|^2 + \varepsilon_n^2)^{\alpha/2 + 2}}
		\right]
		z_i.
	\]
	Note that this sequence of functions converges pointwise almost everywhere to $\Div_z a$ in $\mathbb{R}^N$.
	We now establish a dominating bound independent of $n$.
	Using the fact that $\varepsilon_n^2 / (|z|^2 + \varepsilon_n^2) \leq 1$ on the second term in the square brackets, we find that
	\[
		\left| ( \Div_z a^n )_i \right|
		\leq \frac{N + \alpha + 1}{(|z|^2 + \varepsilon_n^2)^{\alpha/2 + 1}} |z|
		\leq \frac{N + \alpha + 1}{|z|^{\alpha + 1}}.
	\]
	Splitting the divergence as
	$
		\Div_z a^n
		= \mathbbm{1}_{|z| \leq 1} \Div_z a^n
		+ \mathbbm{1}_{|z| > 1} \Div_z a^n,
	$
	we have that the first term is dominated by $(N + \alpha + 1) |z|^{-(\alpha + 1)} \mathbbm{1}_{|z| \leq 1}$, which is an integrable function since $\alpha + 1 \leq N - 1 < N$.
	Therefore, by the dominated convergence theorem, this first term converges strongly in $L^1(\mathbb{R}^N)$.

	The second term is dominated by $(N + \alpha + 1) |z|^{-(\alpha + 1)} \mathbbm{1}_{|z| > 1} \leq (N + \alpha + 1)$.
	Since the sequence is uniformly bounded by a constant and converges almost everywhere, it converges weakly-* in $L^\infty(\mathbb{R}^N)$.
	Combining these results, we conclude that $\Div_z a^n \to \Div_z a$ in $L^1(\mathbb{R}^N) + L^{\infty,*}(\mathbb{R}^N)$.

	Differentiating the expression for $\Div_z a^n$ with respect to $z_i$, and summing over $i \in \{1, 2, \dots, N\}$ yields
	\begin{equation}
		\label{eq:approximate-div-div-a}
		\begin{split}
			\Div_z (\Div_z a^n)
			= \sum_{i,j = 1}^N \mpdv{a^n_{ij}}{z_i}{z_j}
			&= - (N-1) (N - \alpha - 2) \frac{1}{(|z|^2 + \varepsilon_n^2)^{\alpha/2 + 1}} \\
			&\quad - (\alpha + 2) (2N - 5 - \alpha) \frac{\varepsilon_n^2}{(|z|^2 + \varepsilon_n^2)^{\alpha/2 + 2}} \\
			&\quad - (\alpha + 2) (\alpha + 4) \frac{\varepsilon_n^4}{(|z|^2 + \varepsilon_n^2)^{\alpha/2 + 3}}.
		\end{split}
	\end{equation}
	Observe that for $N \geq 3$ and $0 < \alpha \leq N - 2$, the coefficients $(N-1) (N - \alpha - 2)$, $(\alpha + 2) (2N - 5 - \alpha)$ and $(\alpha + 2) (\alpha + 4)$ are nonnegative.
	Consequently, the entire expression is nonpositive, satisfying the final property of \eqref{property:approx-a}.

	We first consider the very soft case, where $0 < \alpha < N - 2$.
	By applying the bound $\varepsilon_n^2 / (|z|^2 + \varepsilon_n^2) \leq 1$ to the second and third terms, and noting that $|z|^2 + \varepsilon_n^2 \geq |z|^2$ throughout, we find that there exists a constant $C_1 > 0$ such that
	\[
		|\Div_z (\Div_z a^n)|
		\leq C_1 \frac{1}{|z|^{\alpha + 2}}.
	\]

	As in previous steps, we decompose the domain to write
	$
		\Div_z (\Div_z a^n)
		= \mathbbm{1}_{|z| \leq 1} \Div_z (\Div_z a^n)
		+ \mathbbm{1}_{|z| > 1} \Div_z (\Div_z a^n)
	$.
	The first term is dominated by $C_1 |z|^{-(\alpha + 2)} \mathbbm{1}_{|z| \leq 1}$, which is an integrable function over $\mathbb{R}^N$, since $\alpha < N - 2$.
	Therefore, by the dominated convergence theorem, this term converges strongly in $L^1(\mathbb{R}^N)$.

	The second term is dominated by
	$
		C_1 |z|^{-(\alpha + 2)} \mathbbm{1}_{|z| > 1}
		\leq C_1.
	$
	Because this sequence is uniformly bounded and converges pointwise almost everywhere, it converges weakly-* in $L^\infty(\mathbb{R}^N)$.
	Combining these results, we have that $\Div_z (\Div_z a^n) \to \Div_z (\Div_z a)$ in $L^1(\mathbb{R}^N) + L^{\infty,*}(\mathbb{R}^N)$.

	We now turn to the Coulomb case, where $\alpha = N-2$.
	In this regime, the first term in \eqref{eq:approximate-div-div-a} vanishes because its coefficient is zero.
	Applying the bound $\varepsilon_n^2 / (|z|^2 + \varepsilon_n^2) \leq 1$ to the third term and using the triangle inequality we find that there exists a constant $C_2 > 0$ such that
	\[
			|\Div_z (\Div_z a^n)|
			\leq C_2 \frac{\varepsilon_n^2}{(|z|^2 + \varepsilon_n^2)^{(N+2)/2}}.
	\]

	To see that the dominating sequence on the right-hand side is uniformly bounded in $L^1(\mathbb{R}^N)$, we first integrate over an arbitrary ball of radius $R > 0$.
	Passing to polar coordinates and applying the scaling $r = \varepsilon_n s$, we obtain
	\[
		\begin{split}
			\int_{|z| \leq R} \frac{\varepsilon_n^2}{(|z|^2 + \varepsilon_n^2)^{(N+2)/2}} \, dz
			&= C_N \int_0^R \frac{\varepsilon_n^2 r^{N-1}}{(r^2 + \varepsilon_n^2)^{(N+2)/2}} \, dr\\
			&= C_N \int_0^{R/\varepsilon_n} \frac{s^{N-1}}{(s^2+1)^{(N+2)/2}}\, ds \\
			&\leq C_N \int_0^{+ \infty}
			\frac{1}{(s^2+1)^{3/2}} \, ds
		\end{split}
	\]
	where $C_N$ denotes the surface area of the unit sphere in $\mathbb{R}^N$.
	For the last inequality, we have used that $s \leq (s^2 + 1)^{1/2}$ for every $s \geq 0$, together with the fact that the integrand is positive.
	Since the function $s \mapsto (s^2 + 1)^{-3/2}$ is integrable on $[0, +\infty)$, this provides a finite upper bound independent of both $R$ and $n$.
	Taking the supremum over $R > 0$ thus yields that the sequence $\Div_z (\Div_z a^n)$ is uniformly bounded in $L^1(\mathbb{R}^N)$.

	Next, we establish the uniform quasi-ellipticity bound \eqref{eq:approximate-aij-ellpiticity}.
	Because the matrix $z \otimes z$ is positive semi-definite and $(|z|^2 + \varepsilon_n^2)^{-1} \leq |z|^{-2}$, we have $(|z|^2 + \varepsilon_n^2)^{-1} (z \otimes z) \leq |z|^{-2} (z \otimes z)$ in the sense of symmetric matrices.
	Additionally, for any $z$ in the ball $|z| \leq R$, the inequality $\varepsilon_n \leq 1$ implies the scalar bound $|z|^2 + \varepsilon_n^2 \leq R^2 + 1$.
	Combining the matrix inequality with this scalar bound, we obtain
	\[
		\begin{split}
			a^n(z)
			&= \frac{1}{(|z|^2 + \varepsilon_n^2)^{\alpha/2}}
				\left(
					I - \frac{z \otimes z}{|z|^2 + \varepsilon_n^{2}}
				\right) \\
			&\geq \frac{1}{(|z|^2 + \varepsilon_n^2)^{\alpha/2}}
				\left(
					I - \frac{z \otimes z}{|z|^2}
				\right) \\
			&\geq \frac{1}{(R^2 + 1)^{\alpha/2}}
				\left(
					I - \frac{z \otimes z}{|z|^2}
				\right).
		\end{split}
	\]
	Setting $C_R = (R^2 + 1)^{-\alpha / 2}$, we recover \eqref{eq:approximate-aij-ellpiticity}.

	Finally, by the uniqueness of the square root of a positive semi-definite matrix, we obtain
	\[
		\sqrt{a^n(z)}
		= \frac{1}{(|z|^2 + \varepsilon_n^2)^{\alpha/4}}
		\left(
			I - \frac{z \otimes z}{\sqrt{|z|^2 + \varepsilon_n^2}(\sqrt{|z|^2 + \varepsilon_n^{2}} + \varepsilon_n)}
		\right).
	\]
	Since the denominator inside the parentheses is bounded below by $|z|^2$, its $i, j$-th entry is bounded by $2$.
	Therefore, for all $z \in \mathbb{R}^N \setminus \{0\}$, we have
	\[
		\left| \left[ \sqrt{a^n(z)} \right]_{ij} \right|
		\leq \frac{2}{(|z|^2 + \varepsilon_n^2)^{\alpha/4}}
		\leq \frac{2}{|z|^{\alpha/2}}.
	\]
	As $\sqrt{a^n(z)}$ converges pointwise to $\sqrt{a(z)}$ and the dominating function $2/|z|^{\alpha/2}$ belongs to $L^1_{loc}(\mathbb{R}^N)$, the dominated convergence theorem ensures that this sequence converges strongly in this space.

	Computing the divergence yields
	\[
		\left( \Div_z \sqrt{a^n(z)} \right)_i
		= - \frac{1}{S_n^{\alpha/2}}
		\frac{1}{S_n (S_n + \varepsilon_n)}
		 \left[
			N-1
			+ \frac{\alpha + 2}{2} \frac{\varepsilon_n}{S_n}
			\left(
				1 + \frac{\varepsilon_n}{S_n}
			\right)
		\right] z_i,
	\]
	where $S_n = \sqrt{|z|^2 + \varepsilon_n^2}$.
	Since $S_n \to |z|$, this expression converges pointwise to $\Div_z \sqrt{a(z)}$.
	Applying the bound $\varepsilon_n / S_n \leq 1$ to the terms in the square brackets and $S_n (S_n + \varepsilon_n) \geq S_n^2$ to the denominator, we obtain
	\[
		\left|
			\left( \Div_z \sqrt{a^n(z)} \right)_i
		\right|
		\leq
		\frac{N + \alpha + 1}{(|z|^2 + \varepsilon_n^2)^{\alpha/4 + 1}} |z_i|
		\leq \frac{N + \alpha + 1}{|z|^{\alpha/2 + 1}}.
	\]
	Because $\alpha/2 + 1 \leq \alpha + 1 < N$, the dominating function on the right-hand side is locally integrable.
	Therefore, by the dominated convergence theorem we have that this sequence converges strongly in $L^1_{loc}(\mathbb{R}^N)$, concluding the proof.
\end{proof}

Next, we construct a sequence of regularized initial data satisfying the assumptions of Proposition \ref{prop-existence-approx} and accurately approximates the initial kinetic energy and moment of inertia.
Furthermore, if the original initial data happens to be uniformly non-saturated, our approximation scheme strictly preserves this bound.

\begin{lemma}
	\label{lem:approximated-initial-data}
	Let $f_0 \in L^1(\Rxv; (1+|x|^2+|v|^2) dxdv)$ be a function such that $0 \leq f_0 \leq 1$.
	There exists a sequence $(f_0^n)_n$ in $C^\infty_c(\Rxv)$ converging to $f_0$ in $L^1(1+|x|^2+|v|^2,\, \Rxv)$ such that for each $n$, there exists a constant $C_n > 0$ such that
	\[
		f^n_0 \leq \frac{C_n e^{-\weightxv}}{1 + C_n e^{-\weightxv}}
		\quad\text{on }\Rxv.
	\]

	Moreover, if there exists a $\delta > 0$ such that $0 \leq f_0 \leq 1 - \delta$ almost everywhere in $\Rxv$ then the sequence $(f_0^n)_n$ is such that $0 \leq f_0^n \leq 1 - \delta$ in $\Rxv$.
\end{lemma}
\begin{proof}
	Let $\phi \in C^\infty_c(\Rxv)$ be a standard nonnegative mollifier supported on $B_1(0,0)$, the unit ball of $\Rxv$, with $\iint \phi \, dxdv = 1$.
	For each $n$, we define the scaled mollifier $\phi^n(x,v) = n^{2N} \phi(nx, nv)$, which is supported in $B_{1/n}(0,0)$.

	Next, let $\psi \in C^\infty_c(\Rxv)$ be a smooth cutoff function satisfying $0 \leq \psi \leq 1$, also supported in $B_1(0,0)$, with $\psi \equiv 1$ in $B_{1/2}(0,0)$.
	For each $n$, we define the truncation $\psi^n(x,v) = \psi(x/n, v/n)$, which is supported in $B_n(0,0)$.

	We define the sequence of approximations by
	\[
		\begin{split}
			f_0^n (x,v)
			&= \left( 1 - \frac{1}{n} \right) \psi^n(x,v) (f_0 * \phi^n)(x,v) \\
			&= \left( 1 - \frac{1}{n} \right) \psi^n(x, v) \iint f_0(x_*, v_*) \phi^n(x-x_*, v-v_*) \, dx_* dv_*.
		\end{split}
	\]

	Let $w(x,v) = 1 + |x|^2 + |v|^2$ denote the quadratic weight.
	To show the convergence of $(f_0^n)_n$ in the weighted space, we expand the expression for $| f_0^n - f_0 |$ using the definition above.
	Adding and subtracting $(1 - 1/n) \psi^n f_0$, applying the triangle inequality and the bound $0 \leq \psi^n \leq 1$, we obtain
	\begin{equation}
		\begin{split}
			\label{eq:approximated-f0-weighted-convergence}
			\iint | f_0^n - f_0 | w \, dx dv
			&\leq \left(1 - \frac{1}{n}\right) \iint | (f_0 * \phi^n) - f_0| w \, dx dv
			+ \iint | \psi^n - 1 | f_0 w \, dx dv \\
			&\quad+ \frac{1}{n} \iint f_0 w \, dx dv
		\end{split}
	\end{equation}

	Consider the second integral on the right-hand side.
	Because $\psi^n \to 1$ pointwise, the integrand $| \psi^n - 1 | f_0 w$ vanishes almost everywhere.
	Furthermore, since $0 \leq \psi^n \leq 1$, this integrand is dominated by $f_0 w$, which is an integrable function over $\Rxv$.
	Therefore, by the dominated convergence theorem, it follows that this second integral converges to zero as $n \to \infty$.

	The last term on the right-hand side of \eqref{eq:approximated-f0-weighted-convergence} vanishes and for the first integral, we introduce the commutator of the mollification with the weight $w$.
	Adding and subtracting $(f_0 w) * \phi^n$ and applying the triangle inequality yields
	\[
			\iint | (f_0 * \phi^n) - f_0| w \, dx dv
			\leq \iint \left|
				(f_0 * \phi^n) w - (f_0 w) * \phi^n
				\right| dx dv
			+ \iint \left| (f_0 w) * \phi^n - f_0 w \right| dxdv.
	\]
	Because $f_0 w \in L^1(\Rxv)$, standard properties of approximation by convolution guarantee that the second term on the right-hand side vanishes as $n \to \infty$.

	To bound the remaining term, we first establish a translation estimate for the weight $w$.
	Writing $x = (x-x_*) + x_*$ and $v = (v-v_*) + v_*$, we expand the difference of the weights and apply the Cauchy-Schwarz inequality to the dot products to obtain that
	\[
		\begin{split}
			w(x,v) - w(x_*, v_*)
			&= |x - x_*|^2 + 2 x_* \cdot (x-x_*)
			+ |v - v_*|^2 + 2 v_* \cdot (v-v_*) \\
			&\leq |x - x_*|^2 + |v - v_*|^2
			+ 2 |x_*| |x-x_*| + 2 |v_*| |v-v_*| \\
			&\leq |x - x_*|^2 + |v - v_*|^2
			+ 2 (|x_*| + |v_*|) (|x-x_*| + |v-v_*|).
		\end{split}
	\]
	Next, we apply this pointwise bound on the commutator.
	Expanding the convolutions and using the fact that both $f_0$ and $\phi^n$ are nonnegative, we find that
	\[
		\begin{split}
			\big| (f_0 * {}& \phi^n ) w - (f_0 w) * \phi^n \big| \\
			&= \left|
				\iint \phi^n(x - x_*, v - v_*) f_0(x_*, v_*)
					[ w(x,v) - w(x_*, v_*) ] \, dxdv
			\right| \\
			&\leq \big[ \phi^n (|x|^2 + |v|^2) \big] * f_0
			+ 2 \big[ \phi^n (|x| + |v|)\big] * \big[ f_0 (|x| + |v|) \big]
 		\end{split}
	\]
	Integrating this pointwise inequality over $\Rxv$ and applying Young's inequality for convolutions allows us to estimate the $L^1$ norm of the commutator.
	Then, applying the inequality $|x| + |v| \leq 2 w(x,v)$ on the second convolution term, we obtain
	\begin{equation*}
	\begin{split}
		\iint | (f_0 *_{x,v} \phi^n) w - (f_0 w) *_{x,v} \phi^n | \, dxdv
		&\leq \left\| \phi^n (|x|^2 + |v|^2) \right\|_{L^1(\Rxv)} \| f_0 \|_{L^1(\Rxv)}\\
		&\quad + 4 \left\| \phi^n (|x| + |v|) \right\|_{L^1(\Rxv)} \left\| w f_0 \right\|_{L^1(\Rxv)}.
	\end{split}
	\end{equation*}
	The right-hand side above vanishes as $n \to \infty$.
	This follows directly from the scaling of the mollifier $\phi^n$, since in this case we have
	\begin{align*}
		\left\| \phi^n \big(|x| + |v|\big) \right\|_{L^1(\Rxv)}
		&= \frac{1}{n}\left\| \phi \big(|x| + |v|\big) \right\|_{L^1(\Rxv)} \\
		\left\| \phi^n \big( |x|^2 + |v|^2 \big) \right\|_{L^1(\Rxv)}
		&= \frac{1}{n^2} \left\| \phi \big( |x|^2 + |v|^2 \big) \right\|_{L^1(\Rxv)}.
	\end{align*}
	Consequently, the first integral on the right-hand side of \eqref{eq:approximated-f0-weighted-convergence} goes to zero, therefore implying the convergence of the sequence $f_0^n \to f_0$ in the desired weighted space.

	Next, we establish the upper bound for the approximation $f^n_0$.
	Recall that $f^n_0$ is supported in the closed ball $B_n(0,0)$.
	Since $0 \leq f_0 \leq 1$ and $\iint \phi^n \, dxdv = 1$, we have $0 \leq f^n_0 \leq (1 - 1/n) < 1$.
	Thus, the quotient $f^n_0/(1- f^n_0)$ is strictly positive, well-defined and belongs to $C^\infty_c(\Rxv)$ with its support in $B_n(0,0)$.

	Because this quotient is continuous and compactly supported, it attains a finite maximum.
	On the other hand, the continuous function $e^{- \weightxv}$ is strictly positive and thus attains a positive minimum on $B_n(0,0)$.
	Therefore, we can choose a sufficiently large constant $C_n > 0$ such that
	\[
		0 \leq \frac{f^n_0}{1 - f^n_0}
		\leq C_n e^{- \weightxv}
	\]
	holds for all $(x,v) \in B_n(0,0)$.
	Since $f^n_0$ vanishes outside of $B_n(0,0)$, this inequality trivially holds over all of $\Rxv$.
	Observing that the map $z \mapsto z/(1+z)$ is strictly increasing for $z \geq 0$, we apply this function to both sides of the inequality to obtain
	\[
		0 \leq f^n_0 \leq \frac{C_n e^{-\weightxv}}{1 + C_n e^{-\weightxv}}.
	\]

	Finally, if $0 \leq f_0 \leq 1 - \delta$, the definition of $f^n_0$, combined with $\iint \phi^n \, dxdv = 1$ and $(1 - 1/n) \psi^n \leq 1$, directly implies that $0 \leq f^n_0 \leq 1 - \delta$.
\end{proof}

We now have all the ingredients necessary to prove Theorem \ref{existence-lfd-indicator}, and we proceed in five steps.
In Step 1, we construct a sequence of approximate solutions using Proposition \ref{prop-existence-approx} and establish its compactness by applying Proposition \ref{prop-compactness}.
In Step 2, we prove the propagation of these compactness properties to the coefficients $\abar$ and $\bbar$, while in Step 3, we establish pointwise-in-time compactness for this scheme.

In Step 4, we use the compactness properties proven thus far to pass to the limit in the approximate equation.
Finally, in Step 5, we recover the conservation laws and inequalities stated in the theorem by passing to the limit in the \emph{a priori} estimates from Proposition \ref{prop-approx-a-priori-estimates}.

\begin{proof}[Proof of Theorem \ref{existence-lfd-indicator}]
	\proofpart{1}{Compactness}

	For each integer $n \geq 1$, let $f^n_0$ be the approximated initial data constructed in Lemma \ref{lem:approximated-initial-data}, and let $a^n$ be the approximated collision kernel constructed in Lemma \ref{lem:approximated-collision-kernel}.
	Let $\varepsilon_n$ be a sequence of strictly positive real numbers converging to zero.
	By Proposition \ref{prop-existence-approx} with $\varepsilon = \varepsilon_n$, there exists a corresponding sequence of solutions $f^n$ to the approximated problem \eqref{regularized}.

	Let us verify the hypotheses needed to apply Proposition \ref{prop-compactness}.
	The approximate a priori estimates of Proposition \ref{prop-approx-a-priori-estimates} give us the bound \eqref{eq:uniform-a-priori} on $f^n (1 + |x-vt|^2 + |v|^2)$.
	For the ellipticity estimate \eqref{ineq-uniform-ellip} let $T, R, R' > 0$.
	Using the inequality $\Abar^n \geq \abar^n$ we obtain that for every $\eta, v \in \mathbb{R}^N_v$,
	\[
		\Abar^n_{ij} \eta_i \eta_j
		\geq \frac{1}{1 + \varepsilon_n \| f^n(t,x) \|_{L^1}}
			\int a^n_{ij}(v-v_*) \eta_i \eta_j f^n_* (1 - f^n_*) \, dv_* .
	\]
	The integrand being nonnegative, we can bound this integral from below by restricting the domain to $\{v_* \in \mathbb{R}^N : |v_*| \leq R'\}$.
	Since whenever $|v| \leq R$ and $|v_*| \leq R'$ we have $|v - v_*| \leq R + R'$, applying the lower matrix bound \eqref{eq:approximate-aij-ellpiticity} we obtain that for any $\eta \in \mathbb{R}^N$,
	\[
		\Abar^n_{ij} \eta_i \eta_j
		\geq \frac{C_{R+R'}}{1 + \varepsilon_n \| f^n(t,x) \|_{L^1}}
			\int_{|v_*| \leq R'} \left[
				1 - \biggl( \frac{v-v_*}{|v-v_*|} \cdot \frac{\eta}{|\eta|} \biggr)^2
			\right] |\eta|^2 f^n_* (1 - f^n_*) \, dv_*.
	\]

	Define the sequence of functions $g^n(t,x) = \| f^n(t,x) \|_{L^1}$.
	From the boundedness of the sequence $f^n$, we have that $g^n$ is uniformly bounded in $L^1_{loc}((0,\infty); L^1(\mathbb{R}^N_x))$.
	Since $\varepsilon_n \to 0$, it follows that $\varepsilon_n g^n \to 0$ strongly in this space.
	Passing to a subsequence, this convergence also holds pointwise almost everywhere in $(0,\infty) \times \mathbb{R}^N_x$.
	We summarize this limit as
	\begin{equation}
		\label{eq:eps-fn-limit}
		\varepsilon_n \| f^n(t,x) \|_{L^1} \to 0
		\quad\text{a.e. and in }
		L^1_{loc}((0,\infty); L^1(\mathbb{R}^N_x)).
	\end{equation}

	Fix $T, R > 0$.
	Restricting ourselves to the compact subset $[0,T] \times B_R$, we apply Egorov's theorem to deduce that for every $\varepsilon > 0$ there exists a measurable set $E \subset [0,T] \times B_R$ such that $|E| \leq \varepsilon$ and $\| \varepsilon_n g^n \|_{L^\infty(E^c)} \to 0$.
	Therefore, there exists an $n_0 \in \mathbb{N}$ such that $1 + \varepsilon_n \|f^n(t,x)\|_{L^1} \leq 2$ uniformly on $E^c$ for all $n \geq n_0$, which implies the second hypothesis of Proposition \ref{prop-compactness} with $\nu = C_{R+R'}/2$.

	We verify that for each fixed $n$, the coefficients $\Abar^n$, $\Div_v \Abar^n$, $\Bbar^n$ and $\Div_v \Bbar^n$ belong to $L^\infty_{loc}((0,\infty) \times \Rxv)$.
	Comparing the regularized problem \eqref{regularized} to the general form \eqref{eq-landau-like}, we identify the coefficients as
	\[
		\Abar^n = \abar^n + \varepsilon_n I
		\quad\text{and}\quad
		\Bbar^n = \bbar^n.
	\]
	The terms $\abar^n$, $\Div_v \abar^n$, $\bbar^n$ and $\Div_v \bbar^n$ can be expressed in the general form
	\[
		\frac{\mathcal{A}^n *_v \mathcal{F}^n}{1 + \varepsilon_n \| f^n(t,x) \|_{L^1}},
	\]
	where the kernel $\mathcal{A}^n$ represents $a^n_{ij}$, $(\Div_z a^n)_i$ or $\Div_z (\Div_z a^n)$ and the function $\mathcal{F}^n$ is either $f^n$ or $f^n(1-f^n)$.
	From Lemma \ref{lem:approximated-collision-kernel}, the kernels $\mathcal{A}^n$ are all uniformly bounded in $L^1(\mathbb{R}^N) + L^\infty(\mathbb{R}^N)$.
	This allows us to decompose $\mathcal{A}^n = \alpha^n + \beta^n$, where the sequence $(\alpha^n)_n$ is uniformly bounded in $L^1(\mathbb{R}^N)$ and $(\beta^n)_n$ is uniformly bounded in $L^\infty(\mathbb{R}^N)$.
	We expand the definition of $\mathcal{A}^n_{ij}$ and we apply the triangle inequality.
	Using Young's inequality for convolutions, together with the bound $|\mathcal{F}^n| \leq |f^n|$, we obtain
	\[
		\begin{split}
			\left| \frac{\mathcal{A}^n *_v \mathcal{F}^n}{1 + \varepsilon_n \| f^n(t,x) \|_{L^1}} \right|
			&\leq \frac{| \alpha^n_{ij} *_v \mathcal{F}^n |}{1 + \varepsilon_n \| f^n(t,x) \|_{L^1}}
			+ \frac{| \beta^n_{ij} *_v \mathcal{F}^n |}{1 + \varepsilon_n \| f^n(t,x) \|_{L^1}} \\
			&\leq \frac{ \| \alpha^n_{ij} \|_{L^1} \| f^n(t,x) \|_{L^\infty} }{1 + \varepsilon_n \| f^n(t,x) \|_{L^1}}
			+ \frac{ \| \beta^n_{ij} \|_{L^\infty} \| f^n(t,x) \|_{L^1} }{1 + \varepsilon_n \| f^n(t,x) \|_{L^1}}.
		\end{split}
	\]

	For the first term, we use the fact that the denominator is greater than $1$ and that $\| f^n(t,x) \|_{L^\infty} \leq 1$.
	For the second term, we use that $y / (1 + \varepsilon y) \leq \varepsilon^{-1}$ for every $y \geq 0$ to obtain
	\begin{equation}
		\label{eq:mathcal-A-L-infty-bound}
		\left| \frac{\mathcal{A}^n *_v \mathcal{F}^n}{1 + \varepsilon_n \| f^n(t,x) \|_{L^1}} \right|
		\leq \| \alpha^n_{ij} \|_{L^1}
		+ \varepsilon_n^{-1} \| \beta^n_{ij} \|_{L^\infty}.
	\end{equation}
	Therefore for each fixed $n$, these convolution terms are bounded in $L^\infty((0,\infty) \times \Rxv)$.
	Since $\varepsilon_n \to 0$, the additional matrix $\varepsilon_n I$ is trivially bounded in $L^\infty$ and thus all coefficients are bounded in $L^\infty_{loc}((0,\infty) \times \Rxv)$.

	Finally, we verify that these coefficients are uniformly bounded in $L^1_{loc}((0,\infty) \times \Rxv)$.
	For the terms taking the general convolution form, the denominator is always strictly greater than one, so they are bounded by $| \mathcal{A}^n *_v \mathcal{F}^n |$.
	From Young's inequality for convolutions and the fact that $f^n$ is uniformly bounded in $L^\infty_{loc}((0,\infty); L^1(\Rxv))$, it follows that $| \mathcal{A}^n *_v \mathcal{F}^n |$ is bounded in $L^\infty_{loc}((0,\infty); L^1(\Rxv)) + L^\infty_{loc}((0,\infty); L^1(\mathbb{R}^N_x; L^\infty(\mathbb{R}^N_v)))$ and therefore in $L^\infty_{loc}((0,\infty); L^1_{loc}(\Rxv))$.
	As $\varepsilon_n \leq 1$ for $n$ sufficiently large, the matrix sequence $\varepsilon_n I$ is uniformly bounded over any compact domain, meaning $\Abar^n$ is also uniformly bounded in $L^\infty_{loc}((0,\infty); L^1_{loc}(\Rxv))$.

	Therefore, applying Proposition \ref{prop-compactness} and extracting a subsequence, we find a function $F$ such that
	\begin{equation}
		\label{eq:existence-fn-1-fn-strong-convergence}
		f^n(1-f^n) \to F
		\quad\text{a.e. and strongly in  }
		L^1_{loc}((0,\infty); L^1(\Rxv)).
	\end{equation}
	This proposition also establishes that $f^n$ converges weakly in $L^1_{loc}$ to a function $f$.
	Moreover, by the Banach-Alaoglu theorem, the bound $0 \leq f^n \leq 1$ allows us to extract a further subsequence converging weakly-* in $L^\infty$.
	By uniqueness of the limit in the sense of distributions, this limit must also be $f$, yielding
	\begin{equation}
		\label{eq:existence-fn-weak-convergence}
		\begin{split}
			f^n &\weakto f
			\quad\text{weakly in  }
			L^1_{loc}((0,\infty); L^1(\Rxv)), \\
			f^n &\weakstarto f
			\quad\text{weakly-* in  }
			L^\infty((0,\infty) \times \Rxv). \\
		\end{split}
	\end{equation}
	Furthermore, for every $T > 0$, the sequence $f^n \to f$ strongly in $L^1(\mathcal{R}_T \times \mathbb{R}^N_v)$, where $\mathcal{R}_T = \{(t,x) \in (0,T) \times \mathbb{R}^N_x : \int F \, dv > 0 \text{ or } \int f \, dv = 0 \}$.

	Let $\mathcal{R}_\infty$ denote the corresponding set when $T = \infty$.
	Passing to a further subsequence, we obtain that
	\begin{equation}
		\label{eq:existence-fn-ae-convergence-R}
		f^n \to f
		\quad\text{a.e. in  }
		\mathcal{R}_\infty \times \mathbb{R}^N_v.
	\end{equation}
	This implies that $f^n(1-f^n) \to f(1-f)$ a.e. on $\mathcal{R}_\infty \times \mathbb{R}^N_v$.
	By the uniqueness of pointwise limits, we conclude that $F = f(1-f)$ almost everywhere on this set.

	On the other hand, on the complement $\mathcal{R}_\infty^c \times \mathbb{R}^N_v$, the definition of the set implies that $\int F \, dv = 0$.
	Since $F \geq 0$ a.e. (being the limit of the nonnegative sequence $f^n(1-f^n)$), it follows that $F = 0$ almost everywhere on $\mathcal{R}_\infty^c \times \mathbb{R}^N_v$.

	Combining these two limits, we deduce that
	\begin{equation}
		\label{eq:existence-F-and-f-1-f}
		F(t,x,v) = f(t,x,v)(1-f(t,x,v)) \mathbbm{1}_{\mathcal{R}_\infty} (t,x)
	\end{equation}
	for almost every $(t,x,v) \in (0,\infty) \times \Rxv$.
	Defining $\mathcal{S} = \mathcal{R}_\infty$, we obtain from the above identity that $\mathcal{S} \subset \{(t,x) \in (0,\infty) \times \mathbb{R}^N_x : \int f(1-f) \, dv > 0 \; \text{or} \int f \, dv = 0\}$.

	\proofpart{2}{Convergence of the coefficients}

	Next, we establish the convergence of the coefficients.
	Recall that if $A^n \to A$ in $L^1(\mathbb{R}^N) + L^{\infty, *}(\mathbb{R}^N)$ and $F^n \to F$ in $L^1_{loc}((0,\infty); L^1(\Rxv))$, then
	\[
		A^n *_v F^n \to A *_v F
		\quad\text{in }
		L^1_{loc}((0,\infty) \times \Rxv).
	\]
	Combining the convergence properties of $a^n$ established in Lemma \ref{lem:approximated-collision-kernel} with the convergence \eqref{eq:existence-fn-1-fn-strong-convergence}, we obtain that $a^n *_v (f^n(1-f^n)) \to a *_v F$ strongly in $L^1_{loc}((0,\infty) \times \Rxv)$.
	Passing to a subsequence in $n$, this convergence also holds pointwise almost everywhere in $(0,\infty) \times \Rxv$.
	Combining \eqref{eq:eps-fn-limit} with this limit, we find
	\[
		\abar^n_{ij}
		= \frac{a^n_{ij} *_v (f^n(1-f^n))}{1 + \varepsilon_n \| f^n(t,x) \|_{L^1}} + \varepsilon_n \delta_{ij}
		\to a_{ij} *_v F
	\]
	almost everywhere in $(0,\infty) \times \Rxv$.
	Applying the triangle inequality and noting that the denominator is bounded below by $1$, we obtain that
	\[
		\left|
			\frac{a^n_{ij} *_v (f^n(1-f^n))}{1 + \varepsilon_n \| f^n(t,x) \|_{L^1}}
			+ \varepsilon_n \delta_{ij}
		\right|
		\leq |a^n_{ij} *_v (f^n(1-f^n))| + \varepsilon_n.
	\]
	Since the right-hand side converges in $L^1_{loc}((0,\infty) \times \Rxv)$, we deduce by the general dominated convergence theorem (see, for example, Section 4.4, Theorem 19 of \cite{royden_2010}), that
	\begin{equation}
		\label{eq:existence-aij-convergence}
		\abar^n + \varepsilon_n I \to a *_v F
	\end{equation}
	strongly in $L^1_{loc}((0,\infty) \times \Rxv)$.

	We establish that $\Div_v \abar^n \to \Div_v \abar$ in $L^1_{loc}((0,\infty) \times \Rxv)$ by the exact same procedure.
	Indeed, returning to Lemma \ref{lem:approximated-collision-kernel} and the previously established limit \eqref{eq:existence-fn-1-fn-strong-convergence} of $f^n(1-f^n)$, we obtain
	\begin{equation*}
		\Div_z a^n *_v (f^n(1-f^n)) \to \Div_z a *_v F
	\end{equation*}
	strongly in $L^1_{loc}((0,\infty) \times \Rxv)$ and, up to passing to a subsequence, pointwise almost everywhere.
	The expressions in \eqref{eq:definition-aij-g-bj-g} then imply that $(\Div_v \abar^n)_j \to (\Div_v \abar)_j$ almost everywhere in $(0,\infty) \times \Rxv$.
	Because the denominators in these expressions is bounded below by $1$, we have the pointwise bounds
	\begin{equation*}
		|(\Div_v \abar^n)_j| \leq | \Div_z a^n *_v (f^n(1-f^n)) |
	\end{equation*}
	Since this right-hand side majorant converges strongly in $L^1_{loc}((0,\infty) \times \Rxv)$, the general dominated convergence theorem once more yields
	\begin{equation}
		\label{eq:existence-div-aij-convergence}
		(\Div_v \abar^n)_j \to (\Div_z a *_v F)_j
	\end{equation}
	strongly in $L^1_{loc}((0,\infty) \times \Rxv)$.

	Handling the $\bbar^n$ terms requires a slightly different approach because the convolutions involve $f^n$, which is only known to converge weakly.
	Therefore, we rely on a weak-strong convergence property for convolutions.
	Observe that if $A^n \to A$ in $L^1(\mathbb{R}^N) + L^{\infty, *}(\mathbb{R}^N)$ and $F^n \weakto F$ weakly in $L^1_{loc}((0,\infty); L^1(\Rxv))$, then
	\[
		A^n *_v F^n \weakto A *_v F
		\quad\text{weakly in }
		L^1_{loc}((0,\infty) \times \Rxv).
	\]
	Applying this property alongside the convergence of $a^n$ established in Lemma \ref{lem:approximated-collision-kernel}, we deduce that $(\Div_z a^n) *_v f^n \weakto (\Div_z a) *_v f$ weakly in $L^1_{loc}((0,\infty) \times \Rxv)$.
	Moreover, since $1 / (1 + \varepsilon_n \| f^n(t,x) \|_{L^1}) \leq 1$, it follows from the pointwise limit \eqref{eq:eps-fn-limit} and the product limit theorem (see Proposition A.2 of \cite{saint-raymond-2009}) that
	\begin{equation}
		\label{eq:existence-bj-convergence}
		\bbar^n_j
		= \frac{(\Div_z a^n)_j *_v f^n}{1 + \varepsilon_n \| f^n(t,x) \|_{L^1}}
		\weakto (\Div_z a)_j *_v f
		= \bbar_j
	\end{equation}
	weakly in $L^1_{loc}((0,\infty) \times \Rxv)$.

	\proofpart{3}{Time continuity and pointwise-in-time convergence}

	To establish this property, we apply the Arzelà-Ascoli theorem to the weak formulation of the equation for $f^n$.
	Specifically, for any test function $\varphi \in \mathcal{D}([0,\infty) \times \Rxv)$ and almost every $t > 0$, $f^n$ satisfies
	\begin{equation}
		\label{eq:fn-to-the-limit}
		\begin{split}
			&\intxv f^n(t) \varphi(t)
			- \intxv f^n_0 \varphi(0)
			- \int_0^t \intxv f^n \pdv{\varphi}{\tau}
			- \int_0^t  \intxv f^n v_i \pdv{\varphi}{x_i} \\
			&\quad = \int_0^t \intxv ( \abar^n_{ij} + \varepsilon_n \delta_{ij} ) f^n \mpdv{\varphi}{v_i}{v_j}
			+ \int_0^t \intxv \left( \pdv{\abar^n_{ij}}{v_i} + \bbar^n_j (1-f^n) \right) f^n \pdv{\varphi}{v_j}.
		\end{split}
	\end{equation}

	Fix a test function $\psi \in \mathcal{D}(\Rxv)$.
	We define the functions
	\[
		g^n(t) \coloneq \intxv f^n(t,x,v) \psi(x,v)
		\quad \text{and} \quad
		g(t) \coloneq \intxv f(t,x,v) \psi(x,v).
	\]
	We aim to demonstrate that for every $t > 0$, we have the pointwise-in-time convergence $g^n(t) \to g(t)$.

	Fix an arbitrary time $T > 0$.
	From the definition of $g^n(t)$, we apply an $L^\infty(\Rxv)$ bound to $\psi$ and use the conservation of mass established in Proposition \ref{prop-approx-a-priori-estimates} on the remaining integral.
	Since $f^n_0$ converges strongly in $L^1(\Rxv)$ by Lemma \ref{lem:approximated-initial-data}, the sequence of initial masses $\| f^n_0 \|_{L^1(\Rxv)}$ is uniformly bounded by a constant $M_0$, yielding
	\[
		\begin{split}
			|g^n(t)|
			&\leq \| \psi \|_{L^\infty(\Rxv)} \| f^n(t) \|_{L^1(\Rxv)} \\
			&= \| \psi \|_{L^\infty(\Rxv)} \| f^n_0 \|_{L^1(\Rxv)} \\
			&\leq M_0 \| \psi \|_{L^\infty(\Rxv)}
		\end{split}
	\]
	for all $t \in [0,T]$.
	This establishes that the sequence $(g^n)_n$ is uniformly bounded on the interval $[0,T]$, independently of $n$.

	Let $0 < s < t \leq T$ be two positive times.
	Choosing a test function $\varphi \in \mathcal{D}([0,\infty) \times \Rxv)$ such that $\varphi(\tau,x,v) = \psi(x,v)$ for all $\tau$ in a neighborhood of $[s,t]$ and taking the difference of the weak formulation \eqref{eq:fn-to-the-limit} evaluated at time $t$ and at time $s$, we obtain
	\begin{equation}
		\begin{split}
			g^n(t) - g^n(s)
			&= \int_s^t  \intxv f^n v_i \pdv{\psi}{x_i}
			+ \int_s^t \intxv ( \abar^n_{ij} + \varepsilon_n \delta_{ij} ) f^n \mpdv{\psi}{v_i}{v_j} \\
			&\quad + \int_s^t \intxv \left( \pdv{\abar^n_{ij}}{v_i} + \bbar^n_j (1-f^n) \right) f^n \pdv{\psi}{v_j}.
		\end{split}
	\end{equation}
	Let $K = \supp \psi \subset \Rxv$ denote the compact support of the test function.
	Using the bound $0 \leq f^n \leq 1$, we can estimate the absolute value of the right-hand side by
	\[
		\begin{split}
			|g^n(t) - g^n(s)|
			\leq (t-s) &\bigg[ \| v \cdot \nabla_x \psi \|_{L^1(K)}
			+ \| \abar^n_{ij} + \varepsilon_n \delta_{ij} \|_{L^\infty((0,T); L^1(K))}
			\| \nabla^2_v \psi \|_{L^\infty(K)} \\
			&\quad + \left(
			  \| (\Div_v \abar^n)_j \|_{L^\infty((0,T); L^1(K))}
			  + \| \bbar^n_j \|_{L^\infty((0,T); L^1(K))}
			\right) \| \nabla_v \psi \|_{L^\infty(K)} \bigg].
		\end{split}
	\]
	As established in Part 1 of the proof, the coefficients $\abar^n$, $\Div_v \abar^n$ and $\bbar^n$ are uniformly bounded in $L^\infty_{loc}((0,\infty); L^1_{loc}(\Rxv))$.
	Since $K$ is compact, the term in square brackets on the right-hand side is bounded by a constant $C_\psi$ independent of $n$, yielding
	\[
		|g^n(t) - g^n(s)| \leq C_\psi |t - s|.
	\]
	This implies that the sequence $(g^n)_n$ is uniformly equicontinuous on $[0,T]$.

	By the Arzelà-Ascoli theorem, combined with a standard diagonal extraction over an increasing sequence of times $T_m \to \infty$ and a countable dense subset of $\mathcal{D}(\Rxv)$, we can extract a single subsequence such that for every $\psi \in \mathcal{D}(\Rxv)$, the sequence $(g^n)_n$ converges locally uniformly on $[0,\infty)$ to a continuous limit function $g_\psi$.

	Furthermore, because $f^n \weakto f$ weakly in $L^1_{loc}((0,\infty); L^1(\Rxv))$, it follows from uniqueness of the limit in the sense of distributions that for almost every $t > 0$, we have $g_\psi(t) = \int_{xv} f(t,x,v) \psi(x,v)$.
	Redefining $f(t)$ on a null set to be weakly continuous, we conclude that for every $\psi \in \mathcal{D}(\Rxv)$ and $t > 0$,
	\begin{equation}
		\label{eq:existence-fn-t-convergence}
		\int_{xv} f^n(t) \psi \to \int_{xv} f(t) \psi.
	\end{equation}

	Let us now extend this convergence to a broader class of test functions.
	Specifically, we consider functions $\psi = \psi(t,x,v) \in C([0,\infty); L^1_{loc}(\Rxv))$ such that for every $T > 0$ there exists a $C_T > 0$ such that $|\psi(t,x,v)| \leq C_T (1 + |x| + |v|)$ for all $t \in [0,T]$.

	From the elementary inequality $|x|^2 + |v|^2 \leq 2 |x - tv|^2 + (1 + 2t^2) |v|^2$, there exists a constant $K_T > 0$ such that $1 + |x|^2 + |v|^2 \leq K_T (1 + |x - tv|^2 + |v|^2)$ for all $t \in [0,T]$.
	Integrating this pointwise inequality against $f^n$ and recalling from Proposition \ref{prop-approx-a-priori-estimates} that the integral of $f^n$ against the weight $(1 + |x-tv|^2 + |v|^2)$ is uniformly bounded in time by a constant $M_T > 0$, we deduce that
	\[
		\sup_{t \in [0,T]}
			\int_{xv} f^n(1 + |x|^2 + |v|^2)
		\leq K_T \sup_{t \in [0,T]}
			\int_{xv} f^n(1 + |x-tv|^2 + |v|^2)
		\leq M_T K_T.
	\]

	Let $(\psi_m)_m$ be a sequence of $C^\infty([0,\infty) \times \Rxv)$ functions such that for every $t \geq 0$, $\psi_m(t) \to \psi(t)$ in $L^1_{loc}(\Rxv)$.
	Let $\chi \in C^\infty([0,\infty])$ be a fixed, non-increasing function such that $\chi(s) = 1$ for $s \leq 1$ and $\chi(s) = 0$ for $s \geq 2$.
	For $R > 0$, we define the family of cutoff functions $\chi_R \in \mathcal{D}(\Rxv)$ by
	\[
		\chi_R(x,v)
		= \chi \left(
			\frac{\sqrt{|x|^2 + |v|^2}}{R}
		\right).
	\]
	By construction, $0 \leq \chi_R \leq 1$, with $\chi_R \equiv 1$ on the ball $B_R$ and $\chi_R \equiv 0$ on $B^c_{2R}$.
	Furthermore, since $\chi$ is non-increasing, the sequence $\chi_R(x,v)$ increases monotonically with $R$.

	We decompose the difference as
	\begin{equation}
		\begin{split}
			\label{eq:existence-fn-t-extension-decomposition}
			\left| \int_{xv} (f^n(t) - f(t)) \psi(t) \right|
			&\leq \left| \int_{xv} (f^n(t) - f(t)) \psi_m(t) \chi_R \right| \\
			&\quad + \int_{xv} |f^n(t) - f(t)| |\psi_m(t) - \psi(t)| \chi_R \\
			&\quad + \int_{xv} | f^n(t) - f(t) | |\psi(t)| (1-\chi_R).
		\end{split}
	\end{equation}

	For the third integral, the function $1 - \chi_R$ is supported entirely on the set $B^c_R$.
	On this region, we have $\sqrt{|x|^2 + |v|^2} > R$, which yields the bound $1 < (\sqrt{|x|^2 + |v|^2}) / R \leq (1 + |x| + |v|) / R$.
	Using the linear growth bound on $\psi$, the positivity of $f^n$ and the inequality $(1 + |x| + |v|)^2 \leq 3 (1 + |x|^2 + |v|^2)$, we can estimate the $f^n$ part of the third term as
	\[
		\begin{split}
			\int_{xv} f^n(t) |\psi(t)| (1-\chi_R)
			&\leq C_T \int_{xv} f^n(t) (1 + |x| + |v|) (1-\chi_R) \\
			&\leq \frac{C_T}{R} \int_{xv} f^n(t) (1 + |x| + |v|)^2 \\
			&\leq \frac{3 C_T M_T K_T}{R}.
		\end{split}
	\]

	This bound propagates to the limit $f$ by approximating the positive weight $|\psi(t)| (1-\chi_R)$ from below with the monotone sequence $\big( |\psi(t)| (1-\chi_R) \chi_M \big)_M$ of compactly supported test functions.
	Applying the convergence \eqref{eq:existence-fn-t-convergence} for each $M$ and invoking the monotone convergence theorem as $M \to \infty$ yields the result.

	Applying the bounds $0 \leq f^n, f \leq 1$ and $0 \leq \chi_R \leq 1$, the second integral of \eqref{eq:existence-fn-t-extension-decomposition} is bounded by $2 \| \psi_m(t) - \psi(t) \|_{L^1(\supp\chi_R)}$.
	Next, using the triangle inequality $|f^n - f| \leq f^n + f$, we see that the third integral is bounded by $6 C_T M_T K_T / R$ uniformly in $n$.
	Since for a fixed $t$ and $m$, the function $(x,v) \mapsto \psi_m(t,x,v) \chi_R(x,v)$ belongs to $\mathcal{D}(\Rxv)$, taking the limit superior as $n \to \infty$ on both sides of \eqref{eq:existence-fn-t-extension-decomposition}, the first integral vanishes by \eqref{eq:existence-fn-t-convergence}, yielding the bound
	\[
		\limsup_{n\to\infty}
		\left| \int_{xv} (f^n(t) - f(t)) \psi(t) \right|
		\leq 2 \| \psi_m(t) - \psi(t) \|_{L^1(\supp\chi_R)}
		+ \frac{6 C_T M_T K_T}{R}.
	\]
	Taking the limit as $m \to \infty$, the first term on the right-hand side above vanishes by the convergence of $\psi_m(t)$ to $\psi(t)$ in $L^1_{loc}(\Rxv)$.
	Finally, taking the limit as $R \to \infty$ gives that for any test function $\psi \in C([0,\infty); L^1_{loc}(\Rxv))$ satisfying $|\psi(t,x,v)| \leq C_T(1 + |x| + |v|)$ for all $t \in [0,T]$, we have the pointwise-in-time convergence
	\begin{equation}
		\label{eq:existence-fn-t-convergence-extended}
		\int_{xv} f^n(t) \psi(t) \to \int_{xv} f(t) \psi(t)
		\quad \text{for all } t > 0.
	\end{equation}

	\proofpart{4}{Passing the equation to the limit}

	We now use the established convergences to pass the weak formulation \eqref{eq:fn-to-the-limit} to the limit.

	Consider the integrals on the left-hand side.
	For a fixed $t > 0$, the first integral converges due to the pointwise-in-time limit \eqref{eq:existence-fn-t-convergence-extended}, while the second integral converges by the strong convergence $f^n_0 \to f_0$ in $L^1(\Rxv)$.
	The remaining two integrals converge as a consequence of the weak convergence \eqref{eq:existence-fn-weak-convergence} of $f^n$.
	Thus, the entire left-hand side converges to
	\[
		\intxv f(t) \varphi(t)
		- \intxv f_0 \varphi(0)
		- \int_0^t \intxv f \pdv{\varphi}{\tau}
		- \int_0^t  \intxv f v_i \pdv{\varphi}{x_i}.
	\]

	For the integrals on the right-hand side of \eqref{eq:fn-to-the-limit}, the strong convergence \eqref{eq:existence-aij-convergence} of $\abar^n_{ij} + \varepsilon_n \delta_{ij}$, together with the weak limit \eqref{eq:existence-fn-weak-convergence} of $f^n$ implies that $( \abar^n_{ij} + \varepsilon_n \delta_{ij} ) f^n \weakto (a *_v F) f$ in $L^1_{loc}((0,\infty) \times \Rxv)$.

	Next, combining the strong convergence \eqref{eq:existence-div-aij-convergence} of $\Div_v a^n$ with the weak limit of $f^n$, we obtain $(\partial_{v_i} \abar^n_{ij}) f^n \weakto (\partial_{z_i} a_{ij} *_v F) f$ in $L^1_{loc}((0,\infty) \times \Rxv)$.
	Finally, the weak convergence \eqref{eq:existence-bj-convergence} of $\bbar^n$, together with the pointwise convergence \eqref{eq:existence-fn-1-fn-strong-convergence} of the uniformly bounded sequence $0 \leq f^n(1-f^n) \leq 1/4$, implies by the product limit theorem that $\bbar^n_j f^n (1 - f^n) \weakto (\Div_z a *_v f) F$ in $L^1_{loc}((0,\infty) \times \Rxv)$.

	Substituting $F = f(1-f) \mathbbm{1}_{\mathcal{R}_\infty}$ from \eqref{eq:existence-F-and-f-1-f}, the right-hand side converges to
	\[
		\begin{split}
			\int_0^t \intxv &(a_{ij} * f(1-f)) f \mathbbm{1}_{\mathcal{R}_\infty} \mpdv{\varphi}{v_i}{v_j} \\
			&+ \int_0^t \intxv \left[
				\left(\pdv{a_{ij}}{z_i} * f(1-f)\right) f
				+ \left(\pdv{a_{ij}}{z_i} * f\right) f(1-f)
			\right]
			\mathbbm{1}_{\mathcal{R}_\infty}
			\pdv{\varphi}{v_j}.
		\end{split}
	\]
	Equating the limits on the left and right-hand sides, we recover the weak formulation for the desired equation.

	\proofpart{5}{A priori estimates}

	We first establish the conservation of mass and linear momentum for the limit function $f$.
	Let $\chi = \chi(t,x,v) = a + b \cdot v$ be a linear combination of the corresponding collision invariants.
	By Proposition \ref{prop-approx-a-priori-estimates}, the approximate solutions $f^n$ satisfy $\int_{xv} f^n(t) \chi = \int_{xv} f^n_0 \chi$ for all $t > 0$.
	We decompose the difference for the limit function as
	\[
		\int_{xv} (f(t) - f_0)\chi
		= \int_{xv} (f(t) - f^n(t)) \chi
		+ \int_{xv} (f^n(t) - f^n_0) \chi
		+ \int_{xv} (f^n_0 - f_0) \chi.
	\]

	The first integral vanishes as $n \to \infty$ due to the pointwise-in-time convergence \eqref{eq:existence-fn-t-convergence}, since $|\chi| \leq |a| + |b| |v|$.
	The second integral is identically zero for every $n$ by the conservation properties of the approximate solutions.
	Finally, the third integral vanishes in the limit $n \to \infty$ because the approximate initial data $f^n_0$ converges to $f_0$ in $L^1(1 + |x|^2 + |v|^2; \Rxv)$ from Lemma \ref{lem:approximated-initial-data}.

	Therefore, the left-hand side vanishes for every $t > 0$.
	By choosing $a = 1, b = 0$ and $a = 0, b = e_k$ (for $k = 1, \dots, N$), we recover the conservation of mass and linear momentum, respectively.

	For the energy and momentum of inertia inequalities, consider the non-negative collision invariant $\zeta(t,x,v) = a |x - tv|^2 + b |v|^2$, where $a, b \geq 0$.
	Because $\zeta$ grows faster than linearly, we localize it using the cutoff functions $\chi_R$ from Part 3.

	Since the product $\zeta \chi_R \in \mathcal{D}(\Rxv)$, the convergence \eqref{eq:existence-fn-t-convergence} guarantees that for all $t > 0$ and $R > 0$, we have $\int_{xv} f^n(t) \zeta \chi_R \to \int_{xv} f(t) \zeta \chi_R$.
	Furthermore, since $f^n(t) \geq 0$ and $\chi_R \leq 1$, we can bound the truncated integral by the non-truncated one, giving $\int_{xv} f^n(t) \zeta \chi_R \leq \int_{xv} f^n(t) \zeta$ for all $n$.
	Taking the limit inferior as $n \to \infty$ on both sides yields $\int_{xv} f(t) \zeta \chi_R \leq \liminf_{n\to\infty} \int_{xv} f^n(t) \zeta$.

	We evaluate the limit on the right-hand side using Proposition \ref{prop-approx-a-priori-estimates}, which gives
	\[
		\int_{xv} f^n(t) \zeta
		= \int_{xv} f^n_0 (a |x|^2 + b |v|^2)
		+ 2\varepsilon_n t \left(a \frac{t^2}{3} + b\right) \| f^n_0 \|_{L^1}.
	\]
	Recall that $\varepsilon_n \to 0$ as $n \to \infty$.
	Moreover, since the approximate initial data $f^n_0$ converges to $f_0$ in $L^1(1 + |x|^2 + |v|^2; \Rxv)$, the sequence of norms $\| f^n_0 \|_{L^1}$ is bounded.
	Therefore, the second term on the right-hand side vanishes.
	The first term converges and the whole expression converges to $\int_{xv} f_0 (a |x|^2 + b |v|^2)$.
	Thus, for all $t > 0$ and $R > 0$, we have $\int_{xv} f \zeta \chi_R \leq \int_{xv} f_0 (a |x|^2 + b |v|^2)$.

	We finally remove the truncation.
	Because $f(t) \geq 0$, $\zeta \geq 0$ and the sequence of cutoff functions $\chi_R$ increases monotonically to $1$ as $R \to \infty$, we can apply the monotone convergence theorem to take the limit as $R \to \infty$ and we obtain that
	\[
		\int_{xv} f \zeta \leq \int_{xv} f_0 (a |x|^2 + b |v|^2).
	\]
	Choosing $a = 1$ and $b = 0$ yields the moment of inertia inequality, while choosing $a = 0$ and $b = 1$ yields the kinetic energy inequality.

	\proofpart{6}{Entropy inequality}

	We now establish the entropy inequality for the limit function $f$ from the corresponding inequality satisfied by the approximated solutions.
	For almost every $t > 0$, we have
	\begin{equation}
		\label{eq-entropy-inequality-for-n}
		\int_{xv} s^{n}(t) + \int_{0}^{t}\int_{xv} d^{n}(\tau) d\tau \leq \int_{xv} s^{n}(0)
	\end{equation}
	We will pass both sides to the limit, starting with the initial entropy on the right-hand side.

	Because $0 \leq f^n_0 \leq 1$, the quantum entropy $s^n(0)$ is non-positive.
	Using that $- (1-x)\log(1-x) \leq x$ for $x \in [0,1]$, we can bound the absolute value of the initial entropy as
	\[
		|s^n(0)| \leq f^n_0 \log \left(\frac{1}{f^n_0}\right) + f^n_0.
	\]

	We split the right-hand side according to the exponential threshold $f^n_0 = \exp{(-|x| - |v|)}$.
	Where $f^n_0 \leq \exp{(-|x| - |v|)}$, we apply the bound $x \log(1/x) \leq C_0 \sqrt{x}$ for $x \in (0,1)$ and we observe that $\sqrt{f^n_0} \leq \exp(-(|x| + |v|) / 2)$.
	Where $f^n_0 > \exp{(-|x| - |v|)}$, we use the fact that $\log(1/f^n_0) \leq (|x| + |v|)$ to obtain
	\begin{align*}
		|s^n(0)|
		&\leq f_0^n \log\left( \frac{1}{f_0^n} \right) \mathbbm{1}_{f_0^n \leq \exp{(-|x| - |v|)}}
			+ f_0^n \log\left( \frac{1}{f_0^n} \right) \mathbbm{1}_{f_0^n > \exp{(-|x| - |v|)}}
			+ f_0^n\\
		&\leq C_0 \sqrt{f_0^n} \mathbbm{1}_{f_0^n \leq \exp{(-|x| - |v|)}}
			+ f_0^n (|x| + |v|) + f_0^n\\
		&\leq C_0 \exp \left( -\frac{1}{2}(|x| + |v|) \right) + f_0^n (1 + |x| + |v|).
	\end{align*}
	The right-hand side above converges in $L^1(\Rxv)$.
	Passing to a subsequence if necessary, the initial data converges almost everywhere, which implies $s^n(0) \to s(0)$ almost everywhere in $\Rxv$.
	Thus, by the general convergence theorem, we obtain that
	\[
		\int_{xv}  s^n(0) \to \int_{xv}  s(0).
	\]

	For the first integral on the left-hand side of \eqref{eq-entropy-inequality-for-n}, applying the same approach for the shifted function $(x', v) \mapsto f^n(t, x'+tv, v)$ and substituting $x = x'+tv$ to return to the original spatial coordinate yields
	\begin{equation}
		\label{eq:existence-sn-bound}
		|s^n(t,x,v)| \leq C_0 \exp \left( -\frac{1}{2}(|x-tv| + |v|) \right) + f^n(t,x,v) (1 + |x-tv| + |v|).
	\end{equation}

	Because $s^n$ is non-positive, this bound gives a non-negative sequence $s^n + C_0 \exp ( -(|x-tv| + |v|) / 2 ) + f^n (1 + |x-tv| + |v|) \geq 0$.
	For a fixed $t > 0$, recall from \eqref{eq:existence-fn-t-convergence-extended} that the sequence $f^n(t)$ converges weakly to $f(t)$ in $L^1(\Rxv)$.
	Because the function $s$ is convex, we can apply weak lower semicontinuity (see \cite{dacorogna-2008}[Theorem 3.23]) to this sequence, yielding
	\[
		\begin{split}
			&\int_{xv}
				\left[
					s(f(t))
					+ C_0 \exp \left( -\frac{1}{2}(|x-tv| + |v|) \right)
					+ f (1 + |x-tv| + |v|)
				\right]\\
			&\leq \liminf_{n\to\infty}
			\int_{xv}
				\left[
					s(f^n(t))
					+ C_0 \exp \left( -\frac{1}{2}(|x-tv| + |v|) \right)
					+ f^n (1 + |x-tv| + |v|)
				\right].
		\end{split}
	\]
	We cancel the exponential term, which is independent of $n$, from both sides.
	Moreover, testing the convergence \eqref{eq:existence-fn-t-convergence-extended} against the admissible test function $\psi(t,x,v) = (1 + |x-tv| + |v|)$ ensures that the integral linear in $f^n$ on the right-hand side converges to its $f$ counterpart on the left-hand side.
	Subtracting these limits, we obtain
	\[
		\int_{xv} s(f(t))
		\leq \liminf_{n\to\infty} \int_{xv} s(f^n(t)).
	\]

	Let us now pass the entropy dissipation $d^n$ part of \eqref{eq-entropy-inequality-for-n} to the limit.
	We first express $d^n$ as the quadratic form
	\begin{equation*}
		\begin{split}
			d^n =
			\bigg\langle &A^n(t,x,v-v_*) \sqrt{f^n(1-f^n) f_*^n(1-f_*^n)}
			\left(
				\frac{\nabla_v f^n}{f^n(1-f^n)} -
				\frac{\nabla_{v_*} f_*^n}{f_*^n(1-f_*^n)}
			\right),\\
			&\sqrt{f^n(1-f^n) f_*^n(1-f_*^n)}
			\left(
				\frac{\nabla_v f^n}{f^n(1-f^n)} -
				\frac{\nabla_{v_*} f_*^n}{f_*^n(1-f_*^n)}
			\right)
			\bigg\rangle,
		\end{split}
	\end{equation*}
	where $\langle \cdot, \cdot \rangle$ denotes the standard inner product on $\mathbb{R}^N$.
	Here, the matrix $A^n$ is defined as
	\begin{equation}
		\label{eq:existence-An-definition}
		A^n(t,x,z)
		= \frac{a^n(z)}{1 + \varepsilon_n \| f^n(t,x) \|_{L^1}}.
	\end{equation}

	Let $\sqrt{A^n}$ denote the matrix square root of the positive semi-definite matrix $A^n$.
	Using that $\nabla_v f / (\sqrt{f(1-f)}) = 2 \nabla_v(\arcsin\sqrt{f})$, we rewrite $d^n$ as
	\begin{align*}
		d^n
		&= \left|
			\sqrt{A^n(v-v_*)} \sqrt{f^n(1-f^n) f_*^n(1-f_*^n)}
			\left(
				\frac{\nabla_v f^n}{f^n(1-f^n)} -
				\frac{\nabla_{v_*} f_*^n}{f_*^n(1-f_*^n)}
			\right)
	 \right|^2\\
		&=\left|
			\sqrt{A^n(v-v_*)} \left(
				\sqrt{f_*^n(1-f_*^n)}
				\frac{\nabla_v f^n}{\sqrt{f^n(1-f^n)}} -
				\sqrt{f^n(1-f^n)}
				\frac{\nabla_{v_*}f_*^n}{\sqrt{f_*^n(1-f_*^n)}}
			\right)
	 \right|^{2}\\
		&= 4 \left|
			\sqrt{A^n (v-v_*)} \left(
				\sqrt{f_*^n(1-f_*^n)} \nabla_v (\arcsin\sqrt{f^n}) -
				\sqrt{f^n(1-f^n)} \nabla_{v_*} (\arcsin\sqrt{f_*^n})
				\right)
			\right|^2.
	\end{align*}

	Let us study the convergence of $\sqrt{A^n}$ and $\Div_v (\sqrt{A^n})$.
	By the uniqueness of the square root for positive semi-definite matrices, we have
	\[
		\sqrt{A^n(t,x,z)}
		= \frac{1}{\sqrt{1 + \varepsilon_n \| f^n(t,x) \|_{L^1}}} \sqrt{a^n(z)}.
	\]
	The factor $c_n(t,x) = (1 + \varepsilon_n \| f^n(t,x) \|_{L^1})^{-1/2}$ satisfies $0 < c_n \leq 1$ and converges to $1$ almost everywhere due to \eqref{eq:eps-fn-limit}.
	By the triangle inequality, we have
	\[
		\left| \sqrt{A^n(t,x,z)} - \sqrt{a(z)} \right|
		\leq \left| \sqrt{a^n(z)} - \sqrt{a(z)} \right|
		+ |c_n(t,x) - 1| \left| \sqrt{a(z)} \right|.
	\]
	Integrating over any compact set in $(t,x,z)$, the first term vanishes by Lemma \ref{lem:approximated-collision-kernel}, while the second one vanishes by the dominated convergence theorem.
	Thus, $\sqrt{A^n} \to \sqrt{a}$ in $L^1_{loc}((0,\infty) \times \mathbb{R}^N_x \times \mathbb{R}^N_z)$.

	Similarly, taking the divergence with respect to $z$ yields
	\[
		\Div_z \sqrt{A^n(t,x,z)}
		= \frac{1}{\sqrt{1 + \varepsilon_n \| f^n(t,x) \|_{L^1}}} \Div_z \sqrt{a^n(z)}.
	\]
	Since Lemma \ref{lem:approximated-collision-kernel} also guarantees that $\Div_z \sqrt{a^n} \to \Div_z \sqrt{a}$ in $L^1_{loc}(\mathbb{R}^N)$, identical reasoning ensures that $\Div_z \sqrt{A^n(t,x,z)}$ converges to $\Div_z \sqrt{a(z)}$ in the same space.

	Expanding the divergence by the product rule, we obtain
	\begin{equation}
		\label{eq:sqrt-an-decomposition}
		\begin{split}
			&\sqrt{A^n (v-v_*)} \sqrt{f_*^n(1-f_*^n)} \nabla_v (\arcsin\sqrt{f^n}) \\
			&= \Div_v \left(
				\sqrt{A^n (v-v_*)} \sqrt{f_*^n(1-f_*^n)} \arcsin\sqrt{f^n}
			\right)\\
			&\quad - \Div_v \left( \sqrt{A^n (v-v_*)} \right) \sqrt{f_*^n(1-f_*^n)}
			\arcsin\sqrt{f^n}.
		\end{split}
	\end{equation}

	To pass the first term of \eqref{eq:sqrt-an-decomposition} to the limit in the sense of distributions, we first establish the convergence of the vector field inside the divergence.
	Denoting $\Rxvv = \mathbb{R}^N_x \times \mathbb{R}^N_v \times \mathbb{R}^N_{v_*}$, let $\varphi \in \mathcal{D}((0,\infty) \times \Rxvv)$ be a test function.
	Evaluating the i-th component of the vector field against $\varphi$, we must show the convergence of the integral
	\[
		\intxvv \left( \sqrt{A^n (v-v_*)} \right)_{ij} \sqrt{f_*^n(1-f_*^n)} \arcsin\sqrt{f^n} \varphi.
	\]
	To do this, we split the domain of integration into the region $\mathcal{R}_\infty \times \mathbb{R}^{2N}_{v,v_*}$, where the sequence $f^n$ converges pointwise almost everywhere, and its complement $\mathcal{R}^c_\infty \times \mathbb{R}^{2N}_{v,v_*}$.

	On the complement region $\mathcal{R}^c_\infty \times \mathbb{R}^{2N}_{v,v_*}$, recall that the sequence $f^n(1-f^n) \to 0$ pointwise almost everywhere.
	Because the $\arcsin$ function is uniformly bounded by $\pi / 2$ on $[0,1]$, the product $\sqrt{f_*^n(1-f_*^n)} \arcsin\sqrt{f^n}$ vanishes almost everywhere on this region.

	On the region $\mathcal{R}_\infty \times \mathbb{R}^{2N}_{v,v_*}$, the pointwise convergence guarantees that $\sqrt{f_*^n(1-f_*^n)} \arcsin\sqrt{f^n}$ converges almost everywhere to $\sqrt{f_*(1-f_*)} \arcsin\sqrt{f}$.
	Combining the behavior on these two disjoint regions, we obtain the pointwise convergence
	\[
		\sqrt{f_*^n(1-f_*^n)} \arcsin\sqrt{f^n}
		\to \mathbbm{1}_{\mathcal{R}_\infty}(t,x)
		\sqrt{f_*(1-f_*)} \arcsin\sqrt{f}
	\]
	almost everywhere on $(0,\infty) \times \Rxvv$.

	Finally, because $(v,v_*) \mapsto v - v_*$ maps compact sets in $\mathbb{R}^{2N}_{x,v}$ to compact sets in $\mathbb{R}^N_z$, the previously established $L^1_{loc}$ convergence of $\sqrt{A^n(z)}$ implies the shifted matrix elements $( \sqrt{A^n (v-v_*)} )_{ij}$ converge strongly in $L^1_{loc}((0,\infty) \times \Rxvv)$.
	Since the sequence $\sqrt{f_*^n(1-f_*^n)} \arcsin\sqrt{f^n}$ is bounded in $L^\infty$ and converges almost everywhere, the dominated convergence theorem ensures that the product of these two terms also converges in $L^1_{loc}((0,\infty) \times \Rxvv)$.
	Pulling the indicator function $\mathbbm{1}_{\mathcal{R}_\infty}$, which is independent of $v$, out of the divergence, we have
	\[
		\Div_v \left(
			\sqrt{A^n (v-v_*)} \sqrt{f_*^n(1-f_*^n)} \arcsin\sqrt{f^n}
		\right)
		\to \mathbbm{1}_{\mathcal{R}_\infty}
		\Div_v \left(
			\sqrt{a(v-v_*)} \sqrt{f_*(1-f_*)} \arcsin\sqrt{f}
		\right)
	\]
	in $\mathcal{D}'((0,\infty) \times \Rxvv)$.

	By an identical argument, since we established earlier that $\Div_v (\sqrt{A^n(v-v_*)})$ converges strongly in $L^1_{loc}$, the dominated convergence theorem yields that $\Div_v ( \sqrt{A^n (v-v_*)} ) \sqrt{f_*^n(1-f_*^n)} \arcsin\sqrt{f^n}$ converges to $\mathbbm{1}_{\mathcal{R}_\infty}
	\Div_v ( \sqrt{a(v-v_*)} ) \sqrt{f_*(1-f_*)}
	\arcsin\sqrt{f}$ in $\mathcal{D}'((0,\infty) \times \Rxvv)$.

	Subtracting this limit from the previous one gives the limit in the sense of distributions of the unexpanded term.
	Because $f$ lacks the regularity to justify the classical product rule, we formally define the distribution
	\[
		\begin{split}
			\sqrt{a (v-v_*)} \sqrt{f_*(1-f_*)}\nabla_v(\arcsin\sqrt{f})
			&\coloneq
			\Div_v \left(
				\sqrt{a(v-v_*)} \sqrt{f_*(1-f_*)} \arcsin\sqrt{f}
			\right) \\
			&\quad - \Div_v \left( \sqrt{a(v-v_*)} \right) \sqrt{f_*(1-f_*)}
			\arcsin\sqrt{f}.
		\end{split}
	\]

	With this definition, we can recombine the subtracted limits.
	Furthermore, by exchanging $v$ and $v_*$ and using that $A^n(v-v_*) = A^n(v_* - v)$, an identical argument yields the corresponding limit for the gradient with respect to $v_*$.
	This implies we have, in $\mathcal{D}'((0,\infty) \times \Rxvv)$,
	\begin{align*}
		\sqrt{A^n (v-v_*)} \sqrt{f_*^n(1-f_*^n)} \nabla_v (\arcsin\sqrt{f^n})
		&\to \mathbbm{1}_{\mathcal{R}_\infty} \sqrt{a (v-v_*)} \sqrt{f_*(1-f_*)}\nabla_v(\arcsin\sqrt{f}), \\[5pt]
		\sqrt{A^n(v-v_*)} \sqrt{f^n(1-f^n)}\nabla_{v_*}(\arcsin\sqrt{f_*^n})
		&\to \mathbbm{1}_{\mathcal{R}_\infty}
		\sqrt{a (v-v_*)}
		\sqrt{f(1-f)}\nabla_{v_*}(\arcsin\sqrt{f_*}),
	\end{align*}
	where the limit distribution on the right-hand side of the second convergence is defined analogously.
	Subtracting both convergences, we conclude that
	\begin{equation*}
		\begin{split}
			&\sqrt{A^n (v-v_*)} \bigg(
				\sqrt{f_*^n(1-f_*^n)} \nabla_v(\arcsin\sqrt{f^n})
				- \sqrt{f^n(1-f^n)}\nabla_{v_*}(\arcsin\sqrt{f_*^n})
				\bigg) \\
			&\to \mathbbm{1}_{\mathcal{R}_\infty}
			\sqrt{a(v-v_*)} \bigg(
				\sqrt{f_*(1-f_*)}\nabla_v(\arcsin\sqrt{f})
				- \sqrt{f(1-f)}\nabla_{v_*}(\arcsin\sqrt{f_*})
			\bigg)
		\end{split}
	\end{equation*}
	in $\mathcal{D}'((0,\infty) \times \Rxvv)$.

	The inequality \eqref{eq-entropy-inequality-for-n}, together with the bound \eqref{eq:existence-sn-bound} for $s^n$, implies that the pre-limit sequence on the left-hand side is uniformly bounded in $L^2((0,\infty) \times \Rxvv)$.
	Therefore, the convergence above actually holds weakly in $L^2((0,\infty) \times \Rxvv)$, and the limit distribution on the right-hand side can be identified with an $L^2$ function.

	Fixing $t > 0$ and applying the weak lower semicontinuity of the $L^2$ norm to the sequence truncated on the time interval $[0,t]$, it follows that for every $t > 0$, we have
	\[
		\int_0^t \int_{x,v,v_*} d(\tau) \; d\tau \leq
		\liminf_{n\to\infty} \int_0^t \int_{x,v,v_*} d^n(\tau) \; d\tau,
	\]
	where the limit entropy dissipation density is defined as
	\[
		d = \mathbbm{1}_{\mathcal{R}_\infty}
		4 \left|
			\sqrt{a (v-v_*)} \left(
				\sqrt{f_*(1-f_*)}\nabla_v(\arcsin\sqrt{f})
				- \sqrt{f(1-f)}\nabla_{v_*}(\arcsin\sqrt{f_*})
			\right)
		\right|^{2}.
	\]
	Taking the $\liminf$ in \eqref{eq-entropy-inequality-for-n}, the final entropy inequality follows, concluding the proof.

\end{proof}

\section{Proof of Theorem \ref{existence-lfd}}
\label{sec:existence-theorem-coulomb}

To prove Theorem \ref{existence-lfd}, we first establish the propagation of non-saturation.
Specifically, we show that if the initial data is non-saturated, then the sequence of approximated solutions we have constructed in Step 1 of Theorem \ref{existence-lfd-indicator} asymptotically preserves this non-degeneracy over any bounded time interval.

As anticipated in Section \ref{sec:main-results}, we achieve this by applying a maximum principle to the sequence of approximated solutions.
Crucially, this method requires the sequence $(\Div_v \bbar^n)_n$ to be uniformly bounded in $L^\infty$, which in turn demands a uniform $L^1$ bound on $\Div_z (\Div_z a^n)$.
However, based on the available a priori estimates and the bounds established in Lemma \ref{lem:approximated-collision-kernel}, we are only able to guarantee this bound in the Coulomb case ($\gamma = -N$), which motivates the restriction in Theorem \ref{existence-lfd}.

\begin{lemma}
	\label{lem:nonsaturation-preservation}
	Let $(f^n)_n$ be the sequence of solutions defined in the proof of Theorem \ref{existence-lfd-indicator} in the Coulomb case $\gamma = -N$.
	Suppose that there exists a constant $\delta > 0$ such that $f^n_0(x,v) \leq 1 - \delta$ in $\Rxv$.
	Then, for every $T > 0$, there exists a $\delta_T > 0$, such that $\limsup_{n\to\infty} f^n \leq 1 - \delta_T$ pointwise on $[0,T] \times \Rxv$.
\end{lemma}
\begin{proof}
	First, we establish that the sequence $(\Div_v \bbar^n)_n$ is uniformly bounded in $L^\infty((0,\infty) \times \Rxv)$.
	Taking the divergence of the expression in \eqref{eq:definition-aij-g-bj-g}, applying Young's inequality for convolutions in the $v$ variable and taking the supremum over $(t,x)$, we obtain
	\[
		\begin{split}
			|\Div_v \bbar^n|
			&= \frac{\left|
				\big( \Div_z (\Div_z a^n) \big) *_v f^n
				\right|
			}{1 + \varepsilon_n \| f^n(t,x)) \|_{L^1}} \\
			&\leq \left|
					\big( \Div_z (\Div_z a^n) \big) *_v f^n
				\right| \\
			& \leq \| \Div_z (\Div_z a^n) \|_{L^1(\mathbb{R}^N)} \| f^n \|_{L^\infty((0, \infty) \times \Rxv)}
		\end{split}
	\]
	By Lemma \ref{lem:approximated-collision-kernel}, the sequence $(\Div_z (\Div_z a^n))_n$ is uniformly bounded in $L^1(\mathbb{R}^N)$, which makes $(\Div_v \bbar^n)_n$ uniformly bounded in the desired space.

	Let $g^n = 1 - f^n$.
	Substituting $g^n$ into equation \eqref{regularized} and expanding the right-hand side yields
	\[
		\pdv{g^n}{t} + v_i \pdv{g^n}{x_i} =
			\Abar^n_{ij} \mpdv{g^n}{v_i}{v_j}
			+ \left(
				\pdv{\abar^n_{ij}}{v_i}
				+ \bbar^n_j (1-2g^n)
			\right) \pdv{g^n}{v_j}
			+ \pdv{\bbar^n_j}{v_j} g^n (1-g^n),
	\]
	where we denote  $\abar^n = \abar^{f^n}$, $\bbar^n = \bbar^{f^n}$ and $\Abar^n = \abar^n + \varepsilon_n I$.
	Using the fact that $0 \leq g^n \leq 1$ and the boundedness of $\Div_v \bbar^n$ shown above, there exists a constant $M > 0$ such that
	\[
		- \pdv{\bbar^n_j}{v_j}  (1-g^n) g^n
		\geq - \| \Div_v \bbar^n \|_{L^\infty((0,\infty) \times \Rxv)} \, g^n
		= - M g^n.
	\]
	Therefore, $g^n$ satisfies the differential inequality
	\begin{equation}
		\label{eq:nondegeneracy-inequality}
		\pdv{g^n}{t} + v_i \pdv{g^n}{x_i} + M g^n
			\geq \Abar^n_{ij} \mpdv{g^n}{v_i}{v_j}
			+ \left(
				\pdv{\abar^n_{ij}}{v_i}
				+ \bbar^n_j (1-2g^n)
			\right) \pdv{g^n}{v_j}.
	\end{equation}

	Consider $\psi(t,x,v) = \sqrt{1 + |x - tv|^2 + |v|^2}$.
	For each $n \in \mathbb{N}$, we define the auxiliary function $u^n(t,x,v)$, depending on parameters $\lambda_n, \theta_n > 0$ which will be chosen suitably small later, as
	\[
		u^n(t,x,v)
		= g^n(t,x,v) e^{Mt}
		+ \lambda_n t
		+ \theta_n \psi(t,x,v).
	\]
	This function is continuous on $[0,\infty) \times \Rxv$.
	Because $g^n e^{Mt}$ and $\lambda_n t$ are nonnegative, we have $u^n \geq \theta_n \psi$.
	For any $t \in [0,T]$, applying the triangle inequality in $x = (x-tv) + tv$ yields
	\[
		|x| + |v|
		\leq |x - tv| + (1+t)|v|
		\leq (2+T) \psi.
	\]
	Therefore, $u^n \to \infty$ as $|x| + |v| \to \infty$, uniformly in $t$.
	This guarantees that $u^n$ must attain a minimum on the set $[0,T] \times \Rxv$ at some finite point $(t^n_*, x^n_*, v^n_*)$.
	Let us show that $t^n_* = 0$.

	Fix $n \in \mathbb{N}$.
	To simplify the notation for the following bounds, we drop the superscript and denote this minimum simply by $(t^n_*, x^n_*, v^n_*)$.
	Suppose, for the sake of contradiction, that $t_* > 0$.
	Since the first and second derivatives with respect to $v$ exist and are continuous, the minimality at $(t_*, x_*, v_*)$ ensures $\nabla_v u^n(t_*, x_*, v_*)  = 0$ and the Hessian matrix $\nabla^2_v u^n(t_*, x_*, v_*)$ is positive semi-definite.
	Substituting the definition of $u^n$, this implies
	\begin{equation}
		\label{eq:nondegeneracy-gn-derivatives-critical-point}
		\begin{split}
			e^{Mt_*} \nabla_v g^n(t_*, x_*, v_*)
			&= - \theta_n \nabla_v \psi(t_*, x_*, v_*),\\
			e^{Mt_*}  \nabla^2_v g^n(t_*, x_*, v_*)
			&\geq - \theta_n \nabla^2_v \psi(t_*, x_*, v_*).
		\end{split}
	\end{equation}

	The derivatives $\nabla_v \psi$ and $\nabla^2_v \psi$ evaluated at $(t_*,x_*,v_*)$ are bounded by a constant depending only on $T$.
	Indeed, letting $w_* = v_* - t_* (x_* - t_* v_*)$ and noting that the derivative $\nabla_v (1 + |x - tv|^2 + |v|^2)$ at $(t_*, x_*, v_*)$ is $2 w_*$ we obtain, by explicit computation,
	\[
		\nabla_v \psi(t_*, x_*, v_*) = \frac{w_*}{\psi(t_*, x_*, v_*)}.
	\]
	Expanding $|w_*|^2$ in terms of dot products and applying the Cauchy-Schwarz inequality, we obtain that
	\begin{equation*}
		\begin{split}
			|w_*|^2 &\leq 2 |v_*|^2 + 2 |t_*|^2 |x_* - t_* v_*|^2 \\
			&\leq 2(1 + |t_*|^2) (|v_*|^2 + |x_* - t_* v_*|^2) \\
			&\leq 2(1 + |t_*|^2) \psi(t_*, x_*, v_*)^2,
		\end{split}
	\end{equation*}
	which implies the uniform bound $| \nabla_v \psi(t_*, x_*, v_*)| \leq \sqrt{2(1 + t_*^2)} \leq \sqrt{2(1 + T^2)}$.

	Next, the Hessian evaluated at the minimum is given by
	\[
		\nabla^2_v \psi(t_*, x_*, v_*) =
		\frac{(1+t_*^2) I}{\psi(t_*, x_*, v_*)}
		- \frac{ w_* \otimes w_* }{\psi(t_*, x_*, v_*)^3}.
	\]
	To prove that this is uniformly bounded for all $(t_*,x_*,v_*) \in [0,T] \times \Rxv$, it suffices to show that the spectral norm of this matrix is uniformly bounded.
	For any unit vector $\xi \in \mathbb{R}^N$, consider the quadratic form associated to the hessian $\nabla^2_v \psi(t_*,x_*,v_*)$.
	Applying the Cauchy-Schwarz inequality to the second term and using the $|w_*|^2$ bound established above yields
	\[
		\begin{split}
			\left|
				\xi^T \big( \nabla^2_v \psi(t_*, x_*, v_*) \big) \xi
			\right|
			&\leq \frac{1 + t^2_*}{\psi(t_*, x_*, v_*)}
			+ \frac{(\xi \cdot w_*)^2}{\psi(t_*, x_*, v_*)^3} \\
			&\leq \frac{3 (1 + t^2_*)}{\psi(t_*, x_*, v_*)}.
		\end{split}
	\]
	Since $\psi \geq 1$, the right-hand side is uniformly bounded by $3 (1 + t^2_*) \leq 3 (1 + T^2)$.

	We now seek a contradiction using the derivatives of $u^n$ at the minimum point.
	These derivatives are not guaranteed to exist in the classical sense, since the solutions constructed in Proposition \ref{prop-existence-approx} lack separate regularity for $\partial_t f^n$ and $\nabla_x f^n$.

	However, since the functions $f^n$ are $C^2$ in the velocity variable $v$, the right-hand side of \eqref{regularized} is continuous.
	This implies that the transport term $\partial_t f^n + v \cdot \nabla_x f^n$ is also continuous, which prompts us to examine the behavior of $u^n$ along the characteristic curves passing through the minimum.

	Since $t_* > 0$, there exists some $\delta > 0$ such that $\delta < t_*$.
	We define the function $h : (-\delta, \delta) \to \mathbb{R}$ by
	\[
		\begin{split}
			h(s)
			&= u^n(t_* + s, x_* + v_* s, v_*) \\
			&= g^n(t_* + s, x_* + v_* s, v_*) e^{M(t_* + s)}
				+ \lambda_n (t_* + s)
				+ \theta_n \psi(t_* + s, x_* + v_* s, v_*).
		\end{split}
	\]

	To show that this function is differentiable in the classical sense, let $(\eta_\sigma)_\sigma$ be a standard mollifier in $(t,x)$ supported in $(-\sigma,\sigma) \times B_\sigma(0)$.
	Consider the mollified version of the function $h$,
	\[
		h_\sigma(s) = \big[
			(g^n e^{Mt} + \lambda_n t + \theta_n \psi) * \eta_\sigma
		\big] (t_* + s, x_* + v_* s, v_*).
	\]
	Since $g^n$ is defined only for $t \geq 0$, we require $\sigma > 0$ to be sufficiently small (specifically, $\sigma < t_* - \delta$) to ensure that the convolution remains well-defined for all $s \in (-\delta, \delta)$.

	Differentiating with respect to $s$ using the chain rule and noting that $\partial_t \psi + v \cdot \nabla_x \psi = 0$, we obtain
	\[
		\begin{split}
		h'_\sigma(s)
		&= \Big[
			\Big(
				\partial_t (g^n e^{Mt}
					+ \lambda_n t
					+ \theta_n \psi)
				+ v \cdot \nabla_x
				(g^n e^{Mt}
				+ \lambda_n t
				+ \theta_n \psi)
			\Big) \ast \eta_\sigma
		\Big] (t_* + s, x_* + v_* s, v_*) \\
		&= \Big[
			\Big(
				e^{Mt} (\partial_t g^n
					+ v \cdot \nabla_x g^n
					+ M g^n)
			+ \lambda_n \Big) \ast \eta_\sigma
		\Big] (t_* + s, x_* + v_* s, v_*).
		\end{split}
	\]
	Since the function $\partial_t g^n + v \cdot \nabla_x g^n = - (\partial_t f^n + v \cdot \nabla_x f^n)$ is continuous, the mollified expression in parentheses converges in $L^\infty((0,\infty) \times \Rxv)$.
	Letting $\sigma \to 0$, the equation becomes
	\[
		h'(s)
		= e^{M(t_* + s)}
		\big[
			\partial_t g^n
			+ v \cdot \nabla_x g^n
			+ M g^n
		\big] (t_* + s, x_* + v_* s, v_*)
		+ \lambda_n.
	\]

	The weak derivative on the left-hand side is therefore a continuous function and this equality holds everywhere.
	Since $u^n$ attains a minimum on $[0,T] \times \Rxv$ at $(t_*, x_*, v_*)$, the function $h$, restricted to $s \in (-\delta, \delta) \cap [-t_*, T - t_*]$, attains a minimum at $s = 0$.
	Because $\delta < t_*$, the interval $(-\delta, 0]$ is strictly contained in this intersection.
	From the minimality at zero, it follows that $h(0) \leq h(s)$ for every $s \in (-\delta, 0)$.
	Therefore, passing the difference quotient to the limit $s \to 0^-$, we conclude that $h'(0) \leq 0$.

	We evaluate $h'$ at $s=0$, applying the differential inequality \eqref{eq:nondegeneracy-inequality}.
	Substituting the relations \eqref{eq:nondegeneracy-gn-derivatives-critical-point} between the first and second derivatives of $g^n$ and $\psi$ at the point of minimum and using that $\Abar^n_{ij}$ is positive semi-definite, we obtain
	\begin{equation}
		\label{eq:nondegeneracy-h-prime-zero-bound}
		\begin{split}
			h'(0)
			&= e^{M t_*}
				[ \partial_t g^n + v \cdot \nabla_x g^n + g^n M] (t_*, x_*, v_*)
				+ \lambda_n \\
			&\geq e^{M t_*}
				\left[
					\Abar^n_{ij} \mpdv{g^n}{v_i}{v_j}
					+ \left(
						\pdv{\abar^n_{ij}}{v_i}
						+ \bbar^n_j (1-2g^n)
					\right) \pdv{g^n}{v_j}
				\right]
				+ \lambda_n \\
			&\geq - \theta_n
			\left[
				\Abar^n_{ij} \mpdv{\psi}{v_i}{v_j}
				+ \left(
					\pdv{\abar^n_{ij}}{v_i}
					+ \bbar^n_j (1-2g^n)
				\right) \pdv{\psi}{v_j}
			\right]
			+ \lambda_n.
		\end{split}
	\end{equation}

	In order to estimate the term on the right-hand side above from below, we first seek an $L^\infty$ bound for the term in square brackets.

	We begin with $\Abar^n_{ij}$.
	By Lemma \ref{lem:approximated-collision-kernel}, the sequence $(a^n)_n$ is uniformly bounded in $L^1(\mathbb{R}^N) + L^\infty(\mathbb{R}^N)$.
	Applying the same decomposition and convolution bounds as in the proof of inequality \eqref{eq:mathcal-A-L-infty-bound}, we obtain that
	\begin{equation}
		\label{eq:Aij-L-infty-bound}
		\left| \Abar^n_{ij} \right|
		\leq \| \alpha^n_{ij} \|_{L^1}
		+ \varepsilon_n^{-1} \| \beta^n_{ij} \|_{L^\infty}
		+ \varepsilon_n.
	\end{equation}
	Since the sequence $\varepsilon_n \to 0$, this implies that there exists a constant $C_A$ such that the $L^\infty((0,\infty) \times \Rxv)$ norm of $\Abar^n_{ij}$ is bounded by $C_A \varepsilon_n^{-1}$.

	Furthermore, Lemma \ref{lem:approximated-collision-kernel} guarantees that the sequence $(\Div_z a^n)_n$ is uniformly bounded in $L^1(\mathbb{R}^N) + L^\infty(\mathbb{R}^N)$.
	The terms $\Div_v \abar^n$ and $\bbar^n$ are defined by convolving $\Div_z a^n$ with the bounded functions $f^n(1-f^n)$ and $f^n$, respectively, then dividing by $1 + \varepsilon_n \| f^n(t,x) \|_{L^1}$.
	The same decomposition and convolution bounds yield that the $L^\infty((0,\infty) \times \Rxv)$ norms of $\Div_v \abar^n_{ij}$ and $\bbar^n_j$ are bounded by $C \varepsilon_n^{-1}$ for some constant $C > 0$.

	Combining these bounds with the uniform estimates for $\nabla_v \psi$ and $\nabla^2_v \psi$ established previously, we can now uniformly bound the term in square brackets in \eqref{eq:nondegeneracy-h-prime-zero-bound}.
	Noting that $\partial_{v_i} \abar^n_{ij} = (\Div_v \abar^n)_j$ and that $|1 - 2g^n| \leq 1$, we deduce that there exists a constant $C_1 > 0$ such that
	\[
		\left|
			\Abar^n_{ij} \mpdv{\psi}{v_i}{v_j}
			+ \left(
				(\Div_v \abar^n)_j
				+ \bbar^n_j (1-2g^n)
			\right) \pdv{\psi}{v_j}
		\right|
		\leq C_1 \varepsilon_n^{-1}.
	\]
	This implies that $h'(0) \geq \lambda_n - C_1 \theta_n \varepsilon_n^{-1}$.
	Choosing $\lambda_n = (C_1 + 1) \varepsilon_n$ and $\theta_n = \varepsilon_n^2$ implies that $h'(0) \geq \varepsilon_n > 0$, which contradicts $h'(0) \leq 0$.
	Therefore, we must have that $t_* = 0$.

	Returning to our original sequence notation, we have shown that for any $n \in \mathbb{N}$, the minimum of $u^n$ is attained at a point of the form $(0, x^n_*, v^n_*)$.
	From the definition of $u^n$ and the choices of $\lambda_n$ and $\theta_n$ above, we have
	\[
		g^n(t,x,v) e^{M t}
		+ (C_1+1) \varepsilon_n t
		+ \varepsilon_n^2 \psi(t,x,v)
		\geq g^n(0, x^n_*, v^n_*)
		+ \varepsilon_n^2 \psi(0, x^n_*, v^n_*).
	\]
	By hypothesis, the first term on the right-hand side is bounded from below by $\delta$, while the last term on the right-hand side is nonnegative by definition of $\psi$.
	Passing to the limit inferior for a fixed $(t,x,v)$, we obtain that
	\[
		e^{Mt} \liminf_{n\to\infty} g^n(t,x,v)
		\geq \delta
	\]
	for every $(t,x,v) \in [0,T] \times \Rxv$.
	Since $e^{Mt} \leq e^{MT}$ on this set, dividing both sides by this exponential gives the result.
\end{proof}

We conclude by proving Theorem \ref{existence-lfd}.
Using the non-saturation property established above, we will show that the loss of diffusion is entirely prevented, meaning the set $\mathcal{S}$ in Theorem \ref{existence-lfd-indicator} is simply $(0,\infty) \times \mathbb{R}^N_x$, up to a set of measure zero.

\begin{proof}[Proof of Theorem \ref{existence-lfd}]
	We construct approximate solutions $(f^n)_n$ as in Step 1 of the proof of Theorem \ref{existence-lfd-indicator}.
	Because $f_0 \leq 1 - \delta$, the initial data constructed in Lemma \ref{lem:approximated-initial-data} also satisfy $0 \leq f^n_0 \leq 1 - \delta$.
	We propagate this non-saturation in time using Lemma \ref{lem:nonsaturation-preservation}, from which we obtain that for every $T > 0$ there exists a $\delta_T > 0$ such that $\limsup_{n\to\infty} f^n \leq 1 - \delta_T$.

	Recall from \eqref{eq:existence-fn-1-fn-strong-convergence} and \eqref{eq:existence-fn-weak-convergence} that $F$ and $f$ are the respective limits of $f^n(1-f^n)$ and $f^n$.
	We now use the uniform bound on $f^n$ to show that the set $\mathcal{S} = \mathcal{R}_\infty = \{(t,x) \in (0,\infty) \times \mathbb{R}^N_x : \int F \, dv > 0 \text{ or } \int f \, dv = 0 \}$ coincides with $(0,\infty) \times \mathbb{R}^N_x$ up to a set of measure zero.
	Consider the sequence $h^n = \sup_{k \geq n} f^k$.
	By definition, $h^n$ decreases pointwise almost everywhere to the function $h = \limsup_{n\to\infty} f^n \leq 1 - \delta_T$.

	We claim that for every $T>0$ we have $F \geq \delta_T f$ a.e. in $[0,T] \times \Rxv$.
	Fix $T>0$ and let $K \subset [0,T] \times \Rxv$ be a compact set.
	By Egorov's theorem, for every $\varepsilon \in (0, \delta_T)$, there exists a measurable set $E \subset K$ with $|E| \leq \varepsilon$ such that $\| h^n - h \|_{L^\infty(K \setminus E)} \to 0$.

	Hence, for any $\eta > 0$, there exists an $N \in \mathbb{N}$ such that for all $n \geq N$, we have
	\[
		f^n \leq h^n \leq h + \eta \leq 1 - \delta_T + \eta
	\]
	on $K \setminus E$.
	Rearranging this inequality and multiplying by $f^n \geq 0$ directly yields $(\delta_T - \eta) f^n \leq f^n (1 - f^n)$ on this set.
	We test the sequence $f^n (1 - f^n)$ against a nonnegative $\varphi \in \mathcal{D}(K)$.
	Since $f^n (1 - f^n) \geq 0$, we can drop the integral over $E$ to estimate the integral over $K$ from below as
	\[
		\iiint_K f^n(1-f^n) \varphi \, dtdxdv
		\geq \iiint_{K \setminus E} f^n(1-f^n) \varphi \, dtdxdv
		\geq (\delta_T - \eta) \iiint_{K \setminus E} f^n \varphi \, dtdxdv.
	\]
	Splitting this last integral as $\int_{K \setminus E} = \int_K - \int_E$ and using the fact that $|E| \leq \varepsilon$ and $0 \leq f^n \leq 1$, we obtain
	\[
			\iiint_K f^n(1-f^n) \varphi \, dtdxdv
			\geq (\delta_T - \eta)
			\left(
				\iiint_{K} f^n \varphi \, dtdxdv
				- \varepsilon \| \varphi \|_{L^\infty}
			\right).
	\]
	Using the weak convergence of both $f^n$ and $f^n(1-f^n)$, we can pass to the limit as $n \to \infty$.
	Then, letting $\varepsilon \to 0$ and $\eta \to 0$, we obtain
	$\iiint F \varphi \, dtdxdv \geq \delta_T \iiint f \varphi \, dtdxdv$.
	Because $K$ and the nonnegative test function $\varphi$ are arbitrary, it follows that $F \geq \delta_T f$ almost everywhere in $[0,T] \times \Rxv$.
	Since $T > 0$ was arbitrary, this relation holds for all $T$.

	It follows that for almost every $(t,x) \in (0,\infty) \times \mathbb{R}^N_x$, if $\int f(t,x,v) \, dv > 0$, we necessarily have $\int F(t,x,v) \, dv > 0$.
	Therefore, up to a set of measure zero, $\mathcal{R}_\infty$ coincides with $(0,\infty) \times \mathbb{R}^N_x$.
	This implies the indicator function $\mathbbm{1}_{\mathcal{S}}$ in Theorem \ref{existence-lfd-indicator} is identically $1$ almost everywhere, which recovers the standard LFD equation and entropy dissipation, completing the proof.
\end{proof}

\bibliographystyle{acm}
\bibliography{references}

@article{bagland_2004,
	title={{Well-posedness for the spatially homogeneous Landau–Fermi–Dirac equation for hard potentials}},
	volume={134},
	number={3},
	journal={Proceedings of the Royal Society of Edinburgh Section A: Mathematics},
	publisher={Royal Society of Edinburgh Scotland Foundation},
	author={V. Bagland},
	year={2004},
	pages={415–447}
}

@article{desvillettes-villani_2000,
	author = {L. Desvillettes and C. Villani },
	title = {{On the spatially homogeneous landau equation for hard potentials part i : existence, uniqueness and smoothness}},
	journal = {Communications in Partial Differential Equations},
	volume = {25},
	number = {1-2},
	pages = {179-259},
	year  = {2000},
	publisher = {Taylor & Francis}
}

@article{lions_1994,
	author = {P. L. Lions},
	journal = {Philosophical Transactions: Physical Sciences and Engineering},
	number = {1679},
	pages = {191--204},
	publisher = {The Royal Society},
	title = {{On Boltzmann and Landau Equations}},
	urldate = {2023-03-30},
	volume = {346},
	year = {1994}
}

@article{villani_1996,
	author = {C{\'e}dric Villani},
	title = {{On the Cauchy problem for Landau equation: sequential stability, global existence}},
	volume = {1},
	journal = {Advances in Differential Equations},
	number = {5},
	publisher = {Khayyam Publishing, Inc.},
	pages = {793 -- 816},
	year = {1996}
}

@article{diperna-lions-1989,
	author = {R. J. DiPerna and P. L. Lions},
	journal = {Annals of Mathematics},
	number = {2},
	pages = {321--366},
	publisher = {Annals of Mathematics},
	title = {{On the Cauchy Problem for Boltzmann Equations: Global Existence and Weak Stability}},
	urldate = {2023-04-05},
	volume = {130},
	year = {1989}
}

@book{bouchut-2000,
	title       = {{Kinetic Equations and Asymptotic Theory}},
	author      = {Bouchut, Fran{\c c}ois and Golse, Fran{\c c}ois and Pulvirenti, Mario},
	publisher   = {{Elsevier}},
	series      = {Series in Applied Mathematics},
	pages       = {162},
	year        = {2000},
	hal_id      = {hal-00538692},
	hal_version = {v1},
}

@article{hormander-1967,
	author = {Lars H{\"o}rmander},
	title = {{Hypoelliptic second order differential equations}},
	volume = {119},
	journal = {Acta Mathematica},
	number = {none},
	publisher = {Institut Mittag-Leffler},
	pages = {147 -- 171},
	year = {1967},
	doi = {10.1007/BF02392081},
	URL = {https://doi.org/10.1007/BF02392081}
}

@article{alonso_2022,
	title = {About the {Landau}-{Fermi}-{Dirac} {Equation} {With} {Moderately} {Soft} {Potentials}},
	volume = {244},
	issn = {1432-0673},
	url = {https://doi.org/10.1007/s00205-022-01779-z},
	doi = {10.1007/s00205-022-01779-z},
	language = {en},
	number = {3},
	urldate = {2023-09-15},
	journal = {Archive for Rational Mechanics and Analysis},
	author = {Alonso, R. and Bagland, V. and Desvillettes, L. and Lods, B.},
	month = jun,
	year = {2022},
	pages = {779--875},
}

@article{golding_2022,
	title = {Existence of smooth solutions to the {Landau}–{Fermi}–{Dirac} equation with {Coulomb} potential},
	volume = {20},
	issn = {1945-0796},
	url = {https://www.intlpress.com/site/pub/pages/journals/items/cms/content/vols/0020/0008/a007/abstract.php},
	doi = {10.4310/CMS.2022.v20.n8.a7},
	language = {EN},
	number = {8},
	urldate = {2023-09-15},
	journal = {Communications in Mathematical Sciences},
	author = {Golding, William and Gualdani, Maria Pia and Zamponi, Nicola},
	year = {2022},
	pages = {2315--2365},
}

@article{dolbeault_1994,
	title = {Kinetic models and quantum effects: {A} modified {Boltzmann} equation for {Fermi}-{Dirac} particles},
	volume = {127},
	issn = {1432-0673},
	shorttitle = {Kinetic models and quantum effects},
	url = {https://doi.org/10.1007/BF00377657},
	doi = {10.1007/BF00377657},
	language = {en},
	number = {2},
	urldate = {2023-09-15},
	journal = {Archive for Rational Mechanics and Analysis},
	author = {Dolbeault, J.},
	month = jun,
	year = {1994},
	pages = {101--131},
}

@Article{ilin_1964,
	Author = {Il'in, A. M.},
	Title = {{On a class of ultraparabolic equations}},
	FJournal = {Soviet Mathematics. Doklady},
	Journal = {Sov. Math., Dokl.},
	ISSN = {0197-6788},
	Volume = {5},
	Pages = {1673--1676},
	Year = {1964},
}

@article{weber_1951,
	ISSN = {00029947},
	URL = {http://www.jstor.org/stable/1990857},
	author = {Maria Weber},
	journal = {Transactions of the American Mathematical Society},
	number = {1},
	pages = {24--37},
	publisher = {American Mathematical Society},
	title = {{The Fundamental Solution of a Degenerate Partial Differential Equation of Parabolic Type}},
	volume = {71},
	year = {1951}
}

@article{kolmogorov_1934,
	ISSN = {0003486X},
	URL = {http://www.jstor.org/stable/1968123},
	author = {A. Kolmogoroff},
	journal = {Annals of Mathematics},
	number = {1},
	pages = {116--117},
	publisher = {Annals of Mathematics},
	title = {{Zufallige Bewegungen (Zur Theorie der Brownschen Bewegung)}},
	volume = {35},
	year = {1934}
}

@article{golse_1988,
	title = {{Regularity of the moments of the solution of a Transport Equation}},
	journal = {Journal of Functional Analysis},
	volume = {76},
	number = {1},
	pages = {110-125},
	year = {1988},
	issn = {0022-1236},
	doi = {https://doi.org/10.1016/0022-1236(88)90051-1},
	author = {François Golse and Pierre-Louis Lions and Benoît Perthame and Rémi Sentis},
}

@article {agoshkov_1984,
    AUTHOR = {Agoshkov, V. I.},
     TITLE = {{Spaces of functions with differential-difference
              characteristics and the smoothness of solutions of the
              transport equation}},
   JOURNAL = {Dokl. Akad. Nauk SSSR},
  FJOURNAL = {Doklady Akademii Nauk SSSR},
    VOLUME = {276},
      YEAR = {1984},
    NUMBER = {6},
     PAGES = {1289--1293},
      ISSN = {0002-3264},
   MRCLASS = {46E35 (45K05)},
  MRNUMBER = {753365},
MRREVIEWER = {Ji\v r\'i\ R\'akosn\'ik},
}

@article{alexandre_2000,
	author = {Radjesvarane Alexandre},
	title = {{On some related non homogeneous 3D Boltzmann models in the non cutoff case}},
	volume = {40},
	journal = {Journal of Mathematics of Kyoto University},
	number = {3},
	publisher = {Duke University Press},
	pages = {493 -- 524},
	year = {2000},
	doi = {10.1215/kjm/1250517678},
	URL = {https://doi.org/10.1215/kjm/1250517678}
}

@book {friedman_1964,
    AUTHOR = {Friedman, Avner},
     TITLE = {{Partial differential equations of parabolic type}},
 PUBLISHER = {Prentice-Hall, Inc., Englewood Cliffs, NJ},
      YEAR = {1964},
     PAGES = {xiv+347},
   MRCLASS = {35.00 (35.62)},
  MRNUMBER = {181836},
MRREVIEWER = {B.\ Frank\ Jones, Jr.},
}

@book{royden_2010,
  title={{Real Analysis}},
  author={Royden, H.L. and Fitzpatrick, P.},
  isbn={9780131437470},
  lccn={2009048692},
  year={2010},
  publisher={Prentice Hall}
}

@article {golse-2019,
    AUTHOR = {Golse, Fran\c cois and Imbert, Cyril and Mouhot, Cl\'ement and
              Vasseur, Alexis F.},
     TITLE = {Harnack inequality for kinetic {F}okker-{P}lanck equations
              with rough coefficients and application to the {L}andau
              equation},
   JOURNAL = {Ann. Sc. Norm. Super. Pisa Cl. Sci. (5)},
  FJOURNAL = {Annali della Scuola Normale Superiore di Pisa. Classe di
              Scienze. Serie V},
    VOLUME = {19},
      YEAR = {2019},
    NUMBER = {1},
     PAGES = {253--295},
      ISSN = {0391-173X,2036-2145},
   MRCLASS = {35Q84 (35B45 35B65)},
  MRNUMBER = {3923847},
MRREVIEWER = {Barbara\ Brandolini},
}

@article{henderson2024,
	title = {{Kinetic Schauder estimates with time-irregular coefficients and uniqueness for the Landau equation}},
	journal = {Discrete and Continuous Dynamical Systems},
	volume = {44},
	number = {4},
	pages = {1026-1072},
	year = {2024},
	issn = {1078-0947},
	doi = {10.3934/dcds.2023137},
	url = {https://www.aimsciences.org/article/id/65572010fac71d5c94c3d51a},
	author = {Christopher Henderson and Weinan Wang},
}

@article {dong2025,
    AUTHOR = {Dong, Hongjie and Yastrzhembskiy, Timur},
     TITLE = {Global {S}chauder estimates for kinetic
              {K}olmogorov-{F}okker-{P}lanck equations},
   JOURNAL = {Adv. Nonlinear Stud.},
  FJOURNAL = {Advanced Nonlinear Studies},
    VOLUME = {25},
      YEAR = {2025},
    NUMBER = {2},
     PAGES = {422--455},
      ISSN = {1536-1365,2169-0375},
   MRCLASS = {35K70 (35B45 35H10 35K15 35Q84 35R05)},
  MRNUMBER = {4924753},
MRREVIEWER = {Yingzhe\ Fan},
       DOI = {10.1515/ans-2023-0167},
       URL = {https://doi.org/10.1515/ans-2023-0167},
}

@incollection{landau-translated-1936,
	author    = {Landau, L. D.},
	title     = {{The Transport Equation in the Case of Coulomb Interactions}},
	editor    = {D. {Ter Haar}},
	booktitle = {Collected Papers of L.D. Landau},
	publisher = {Pergamon Press},
	pages     = {163--170},
	year      = {1965},
	isbn      = {978-0-08-010586-4},
	note      = {Originally published in Phys. Z. Sowjet. 10, 154 (1936)}
}

@book {dacorogna-2008,
    AUTHOR = {Dacorogna, Bernard},
     TITLE = {{Direct methods in the calculus of variations}},
    SERIES = {Applied Mathematical Sciences},
    VOLUME = {78},
   EDITION = {Second},
 PUBLISHER = {Springer, New York},
      YEAR = {2008},
     PAGES = {xii+619},
      ISBN = {978-0-387-35779-9},
   MRCLASS = {49-02 (49J10 49J45 74B20)},
  MRNUMBER = {2361288},
MRREVIEWER = {Pietro\ Celada},
}

@book {saint-raymond-2009,
    AUTHOR = {Saint-Raymond, Laure},
     TITLE = {Hydrodynamic limits of the {B}oltzmann equation},
    SERIES = {Lecture Notes in Mathematics},
    VOLUME = {1971},
 PUBLISHER = {Springer-Verlag, Berlin},
      YEAR = {2009},
     PAGES = {xii+188},
      ISBN = {978-3-540-92846-1},
   MRCLASS = {82C40 (35Q20 76N10 76P05 82-02)},
  MRNUMBER = {2683475},
MRREVIEWER = {Nader\ Masmoudi},
       DOI = {10.1007/978-3-540-92847-8},
       URL = {https://doi.org/10.1007/978-3-540-92847-8},
}

@article{lu-2008,
	title = {{On the Boltzmann equation for Fermi–Dirac particles with very soft potentials: Global existence of weak solutions}},
	journal = {Journal of Differential Equations},
	volume = {245},
	number = {7},
	pages = {1705-1761},
	year = {2008},
	issn = {0022-0396},
	doi = {https://doi.org/10.1016/j.jde.2008.06.028},
	url = {https://www.sciencedirect.com/science/article/pii/S002203960800288X},
	author = {Xuguang Lu},
}

@article{lu-2006,
	author = {Lu, Xuguang},
	date = {2006/08/01},
	doi = {10.1007/s10955-006-9039-5},
	id = {Lu2006},
	isbn = {1572-9613},
	journal = {Journal of Statistical Physics},
	number = {2},
	pages = {517--547},
	title = {{On the Boltzmann Equation for Fermi--Dirac Particles with Very Soft Potentials: Averaging Compactness of Weak Solutions}},
	url = {https://doi.org/10.1007/s10955-006-9039-5},
	volume = {124},
	year = {2006}}

\end{document}